\documentclass[12pt, a4paper, twoside, onecolumn]{report}
\usepackage{amsmath,amssymb,amsthm,amsfonts,amsxtra}
\usepackage{enumerate,verbatim,mathrsfs,comment}
\usepackage[all,2cell,ps]{xy}
\usepackage[pagebackref]{hyperref}
\usepackage{graphicx}
\usepackage{pdfpages}

\usepackage[toc,page]{appendix}

\usepackage[top=30mm,headheight=3mm,headsep=12mm,bottom=22mm,foot=3mm,footskip=10mm,left=30mm,right=20mm]{geometry}
\usepackage{makeidx}
\makeindex  

\DeclareMathOperator{\ann}{Ann}
\DeclareMathOperator{\Ass}{Ass}

\DeclareMathOperator{\codim}{codim}
\DeclareMathOperator{\cx}{cx}
\DeclareMathOperator{\depth}{depth}
\DeclareMathOperator{\eend}{end}
\DeclareMathOperator{\Ext}{Ext}

\DeclareMathOperator{\gr}{gr}

\DeclareMathOperator{\Hom}{Hom}
\DeclareMathOperator{\Image}{Image}
\DeclareMathOperator{\injdim}{injdim}
\DeclareMathOperator{\Ker}{Ker}
\DeclareMathOperator{\Min}{Min}
\DeclareMathOperator{\Proj}{Proj}
\DeclareMathOperator{\projdim}{projdim}

\DeclareMathOperator{\rank}{rank}
\DeclareMathOperator{\reg}{reg}
\DeclareMathOperator{\Soc}{Soc}
\DeclareMathOperator{\Spec}{Spec}
\DeclareMathOperator{\Supp}{Supp}
\DeclareMathOperator{\Tor}{Tor}

\DeclareMathOperator{\Var}{Var}

\renewcommand{\ge}{\geqslant}
\renewcommand{\le}{\leqslant}

\theoremstyle{plain}
\newtheorem{theorem}{Theorem}[section]
\newtheorem{lemma}[theorem]{Lemma}
\newtheorem{proposition}[theorem]{Proposition}
\newtheorem{corollary}[theorem]{Corollary}

\theoremstyle{definition}
\newtheorem{definition}[theorem]{Definition}
\newtheorem{discussion}[theorem]{Discussion}
\newtheorem{example}[theorem]{Example}
\newtheorem{hypothesis}[theorem]{Hypothesis}

\newtheorem{question}[theorem]{Question}
\newtheorem{para}[theorem]{}

\newtheorem{innercustomnotations}{Notations}
\newenvironment{customnotations}[1]
  {\renewcommand\theinnercustomnotations{#1}\innercustomnotations}
  {\endinnercustomnotations}

\newtheorem{innercustomquestion}{Question}
\newenvironment{customquestion}[1]
  {\renewcommand\theinnercustomquestion{#1}\innercustomquestion}
  {\endinnercustomquestion}
  
\theoremstyle{remark}
\newtheorem{remark}[theorem]{Remark}

\numberwithin{equation}{section}

\begin{document}

\pagenumbering{roman}
\thispagestyle{empty} 
\begin{center}
{\Large \bf
Asymptotic Prime Divisors Related to Ext, Regularity of Powers of Ideals, and Syzygy Modules
}\\

\vspace{1cm}
Thesis\\
\vspace{0.4cm}
Submitted in partial fulfillment of the requirements\\
\vspace{0.4cm}
of the degree of\\
\vspace{0.8cm}

{\bf \large Doctor of Philosophy}\\
\vspace{0.4cm}
by\\
\vspace{0.4cm}

{\large \bf Dipankar Ghosh}\\
\vspace{0.4cm}
Roll No. 114090004\\

\vspace{0.8cm}
Under the supervision of\\
\vspace{0.4cm}
{\bf \large Prof. Tony J. Puthenpurakal}\\
\vspace {2cm}

\includegraphics[width=1.8in]{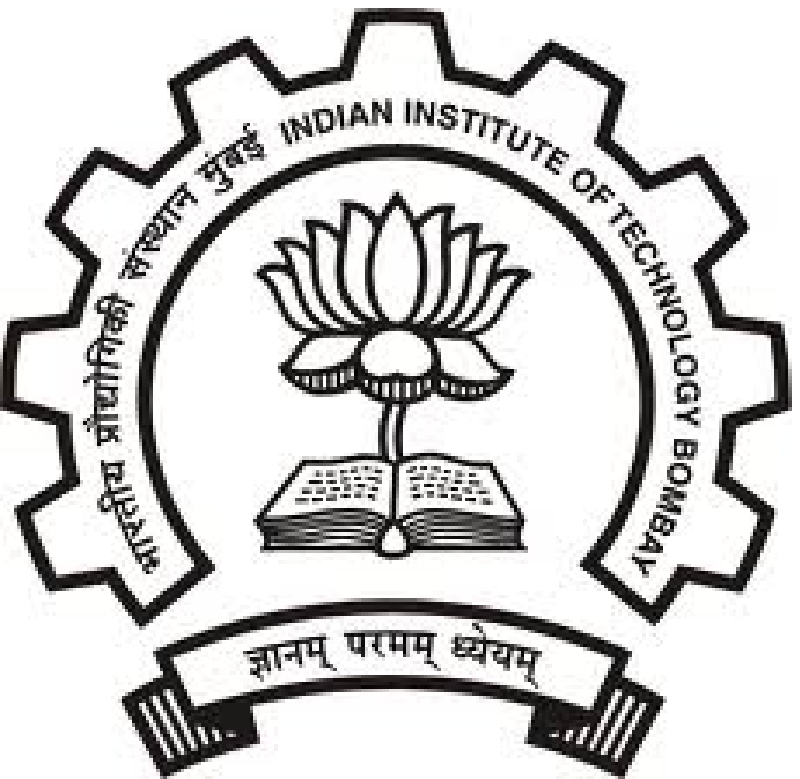}\\

\vspace {1cm}
{\bf \Large
Department of Mathematics \\
\vspace{0.4cm}
Indian Institute of Technology Bombay\\
\vspace{0.4cm}
2016
}
\end{center}

\newpage
\thispagestyle{empty} 	
\cleardoublepage      	

\thispagestyle{empty}

\vspace*{8cm}
\begin{center}
 {\huge \it 
  Dedicated To\\
  \vspace*{0.5cm}
  My Parents and Grandparents }
\end{center}

\pagenumbering{roman}

\newpage
\thispagestyle{empty}
\cleardoublepage

\thispagestyle{plain}

\includepdf{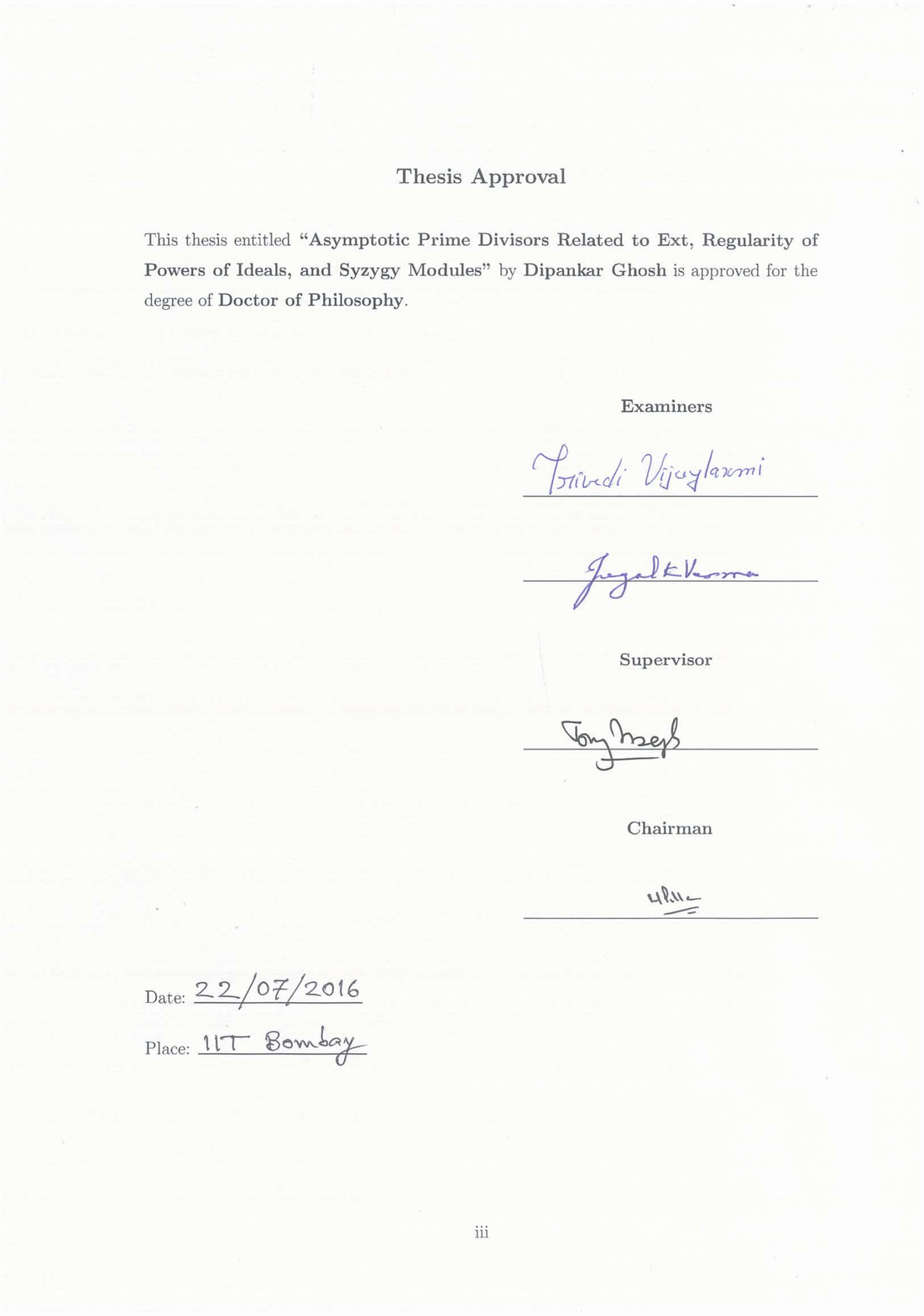}

%
%

\newpage
\thispagestyle{empty}
\cleardoublepage

\thispagestyle{plain}

\includepdf{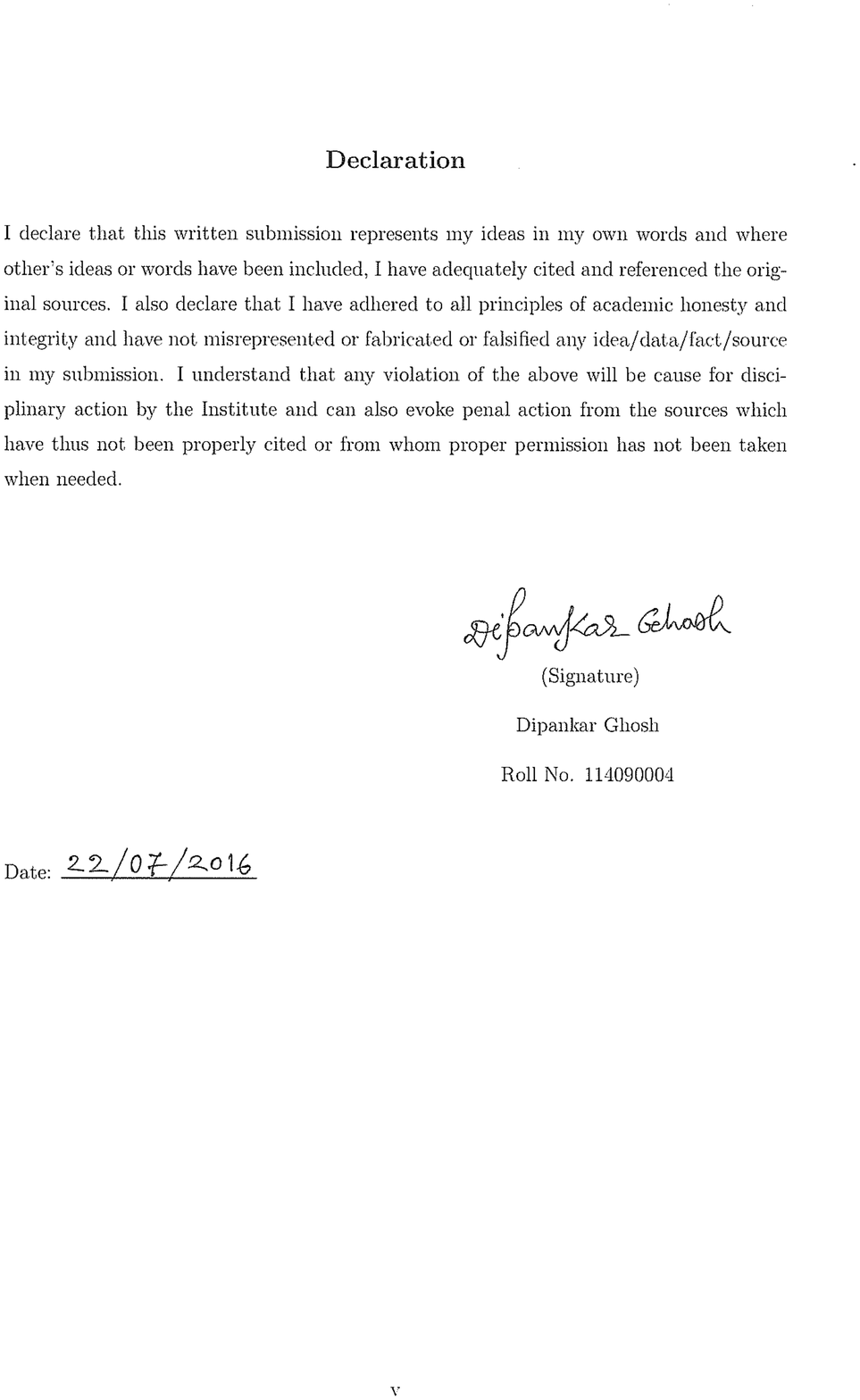}

%

\newpage
\thispagestyle{empty}
\cleardoublepage

\chapter*{Abstract}
\addcontentsline{toc}{chapter}{Abstract} 	
\markboth{Abstract}{Abstract}
 
 This dissertation focuses on the following topics: (1) asymptotic prime divisors over complete intersection rings,
 (2) asymptotic stability of complexities over complete intersection rings, (3) asymptotic linear bounds of Castelnuovo-Mumford
 regularity in multigraded modules, and (4) characterizations of regular local rings via syzygy modules of the residue field.
 
 Chapter~\ref{Introduction} introduces the concerned research problems and provides an overview of the works presented
 in this dissertation.
 
 Chapter~\ref{Review of Literature} gives a detailed literature review of the subjects studied in this dissertation.
 
 Now we concentrate on our main four topics mentioned in the first paragraph. For each topic, there is a separate chapter as follows.
 
 Chapter~\ref{Chapter: Asymptotic Prime Divisors over Complete Intersection Rings} deals with the study of asymptotic behaviour of
 certain sets of associated prime ideals related to Ext-modules. Let $A$ be a local complete intersection ring. Suppose $M$ and $N$
 are two finitely generated $A$-modules and $I$ is an ideal of $A$. We prove that
 \[
  \bigcup_{n \ge 1} \bigcup_{i \ge 0} \Ass_A\left( \Ext_A^i(M,N/I^n N) \right)
 \]
 is a finite set. Moreover, we analyze the asymptotic behaviour of the sets
 \[
  \Ass_A\left( \Ext_A^i(M, N/I^n N) \right) \quad\mbox{if $n$ and $i$ both tend to $\infty$}.
 \]
 We show that there are non-negative integers $n_0$ and $i_0$ such that for all $n \ge n_0$ and $i \ge i_0$ the set
 $\Ass_A\left( \Ext_A^i(M, N/I^n N) \right)$ depends only on whether $i$ is even or odd.
 We also prove the analogous results for complete intersection rings which arise in algebraic geometry.

 Chapter~\ref{Chapter: Asymptotic Stability of Complexities} shows that if $A$ is a local complete intersection ring, then the
 complexity $\cx_A(M,N/I^nN)$ is independent of $n$ for all sufficiently large $n$.
 
 Chapter~\ref{Chapter: Asymptotic linear bounds of Castelnuovo-Mumford regularity} is devoted to study the Castelnuovo-Mumford
 regularity of powers of several ideals.
 Let $A$ be a Noetherian standard $\mathbb{N}$-graded algebra over an Artinian local ring $A_0$. Let $I_1,\ldots,I_t$ be homogeneous
 ideals of $A$, and let $M$ be a finitely generated $\mathbb{N}$-graded $A$-module.
 We show that there exist two integers $k_1$ and $k_1'$ such that
 \[
   \reg(I_1^{n_1} \cdots I_t^{n_t} M) \le (n_1 + \cdots + n_t) k_1 + k'_1 \quad \mbox{for all }~ n_1, \ldots, n_t \in \mathbb{N}.
 \]  
 Moreover, we prove that if $A_0$ is a field, then there exist two integers $k_2$ and $k'_2$ such that
 \[
   \reg\left(\overline{I_1^{n_1}}\cdots \overline{I_t^{n_t}} M\right) \le (n_1 + \cdots + n_t) k_2 + k'_2
   \quad \mbox{for all }~ n_1, \ldots, n_t \in \mathbb{N},
 \]
 where $\overline{I}$ denotes the integral closure of an ideal $I$ of $A$.
 
 Chapter~\ref{Chapter: Characterizations of Regular Local Rings via Syzygy Modules} is allocated for the syzygy modules of the residue
 field of a Noetherian local ring. Let $A$ be a Noetherian local ring with residue field $k$. We show that if a finite direct sum of
 syzygy modules of $k$ maps onto `a semidualizing module' or `a non-zero maximal Cohen-Macaulay module of finite injective
 dimension', then $A$ is regular. We also prove that $A$ is regular if and only if some syzygy module of $k$ has a non-zero direct
 summand of finite injective dimension.
 
 We conclude this dissertation by presenting a few open questions in Chapter~\ref{Summary and Conclusions}.
 
 \noindent {\large \it Key words and phrases.} Associate primes; Graded rings and modules; Rees rings and modules; Ext; Tor; Complete
 intersection rings; Eisenbud operators; Support variety; Complexity; Hilbert functions; Local cohomology; Castelnuovo-Mumford regularity;
 Regular rings; Syzygy and cosyzygy modules; Semidualizing modules; Injective dimension.
 
\newpage
\thispagestyle{empty}
\cleardoublepage

\tableofcontents 	
\markboth{Contents}{Contents}

\newpage
\thispagestyle{empty}	
\cleardoublepage

\chapter*{List of Notations}
\addcontentsline{toc}{chapter}{List of Notations}
\markboth{List of Notations}{List of Notations}

 Throughout this dissertation, unless otherwise specified, {\it all rings} ({\it graded or not}) {\it are assumed to be commutative
 Noetherian rings with identity}.
 
 Let $A$ be a  ring, and let $I$ be an ideal of $A$. Suppose $M$ and $N$ are finitely generated $A$-modules. We use the following list of notations in this dissertation.
 \par \noindent
 {\bf Sets}
 \par \noindent
  \begin{tabular}{|p{2.5cm}|p{12.5cm}|}
  \hline 
  $\mathbb{N}$ & the set of all non-negative integers\\ \hline
  $\mathbb{Z}$ & the set of all integers\\ \hline
  $\Lambda$ & a finite collection of non-negative integers\\ \hline
  $\phi$ & empty set\\ \hline
  ${\bf f}$ & a finite sequence $f_1,\ldots,f_c$\\ \hline
  $\Spec(A)$ & the set of all prime ideals of $A$, {\it spectrum} of $A$\index{spectrum of a ring} \\ \hline
  $\Min(A)$ & the set of all minimal prime ideals of $A$ \\ \hline
  $\Supp(M)$ & $\{ \mathfrak{p} \in \Spec(A) : M_{\mathfrak{p}} \neq 0 \}$, {\it support} of $M$\index{support of a module}\\ \hline
  $\Proj(A)$ & the set of all homogeneous prime ideals of an $\mathbb{N}$-graded ring $A = \bigoplus_{n \ge 0}A_n$
		which do not contain the irrelevant ideal $A_{+} = \bigoplus_{n \ge 1}A_n$\index{Proj of a graded ring}\\ \hline
  $\Var(\mathfrak{a})$ & $\{ \mathfrak{p} \in \Spec(A) : \mathfrak{a} \subseteq \mathfrak{p} \}$,
                         {\it variety of an ideal} $\mathfrak{a}$ of $A$\index{variety of an ideal}\\ \hline
  $\Ass_A(M)$ & the set of all associated prime ideals of an $A$-module $M$\\ \hline
  $D(x)$ & $\{ \mathfrak{p} \in \Spec(A) : x \notin \mathfrak{p} \}$,
           where $x$ is an element of $A$\index{basic open set of $\Spec(A)$}\\ \hline
  ${}^*D(x)$ & $\{\mathfrak{p} \in \Proj(A) : x \notin \mathfrak{p}\}$, where $A$ is an $\mathbb{N}$-graded ring and
		$x$ is a homogeneous element of $A$\\\hline
  ${}^*\Ass_A(M)$ & $\Ass_A(M) \cap \Proj(A)$, the set of {\it relevant associated prime ideals}\index{relevant associate primes}
		     of an $\mathbb{N}$-graded module $M$ over an $\mathbb{N}$-graded ring $A$\\ \hline
  $\mathscr{V}(M,N)$ & support variety of $M$ and $N$\\ \hline
  \end{tabular}
  \newpage \noindent
  {\bf Graded Objects}\\
  Let $R = \bigoplus_{\underline{n} \in \mathbb{N}^t} R_{\underline{n}}$ be an $\mathbb{N}^t$-graded ring
  (where $t$ is a fixed positive integer), and let $L = \bigoplus_{\underline{n} \in \mathbb{N}^t} L_{\underline{n}}$ be an
   $\mathbb{N}^t$-graded $R$-module. Suppose $A$ is an $\mathbb{N}$-graded ring and $M$ is an $\mathbb{N}$-graded
  $A$-module. Assume $I$ and $J$ are two homogeneous ideals of $A$.
  \par \noindent
  \begin{tabular}{|p{2.5cm}|p{12.5cm}|}
  \hline 
  $\underline{n}$ & $t$-tuple $(n_1,n_2,\ldots,n_t)$ over $\mathbb{N}$, i.e., element of $\mathbb{N}^t$ \\ \hline
  $|\underline{n}|$ & sum of the components of $\underline{n}$, i.e., $(n_1 + n_2 + \cdots + n_t)$\\ \hline
  $\underline{e}^i$ & $t$-tuple with all components $0$ except the $i$th component which is $1$\\ \hline
  $\underline{0}$ & $t$-tuple with all components $0$\\ \hline
  $L_{\underline{n}}$ & the $\underline{n}$th graded component of an $\mathbb{N}^t$-graded module $L$\\ \hline
  $\eend(M)$ & $\sup\{ n : M_n \neq 0 \}$, end of $M$\index{end of an $\mathbb{N}$-graded module}\\ \hline
  $L(\underline{u})$ & same as $L$ but the grading is twisted by $\underline{u}$,
			i.e., $L(\underline{u})$ is an $\mathbb{N}^t$-graded module with $L(\underline{u})_{\underline{n}}
			:= L_{(\underline{n} + \underline{u})}$ for all $\underline{n} \in \mathbb{N}^t$\\ \hline
  $\deg(a)$ & homogeneous degree of an element $a$\\ \hline
  $d(J)$ & $\min\{d : J \mbox{ is generated by homogeneous elements of degree}\le d \}$\\ \hline
  $\rho_M(I)$ & $\min\{ d(J) : J \mbox{ is an $M$-reduction of } I \}$\\ \hline
  $\varepsilon(M)$ & the smallest degree of the (non-zero) homogeneous elements of $M$ \\ \hline
  $A_{+}$ & $\bigoplus_{n \ge 1}A_n$ which is called the {\it irrelevant ideal}\index{irrelevant ideal of a graded ring} of $A$\\ \hline
  $a_i(M)$ & $\eend\left( H_{A_{+}}^i(M) \right)$\\ \hline
  \end{tabular}
  \vspace*{0.2cm}
  \par \noindent
  {\bf Rings}
  \par \noindent
  \begin{tabular}{|p{2.5cm}|p{12.5cm}|}
  \hline
  $(A,\mathfrak{m})$ & local ring $A$ with its unique maximal ideal $\mathfrak{m}$ \\ \hline
  $(A,\mathfrak{m},k)$ & local ring $(A,\mathfrak{m})$ with its residue field $k := A/\mathfrak{m}$ \\ \hline
  $\widehat{A}$ & $\mathfrak{m}$-adic completion of $A$, where $(A,\mathfrak{m})$ is a local ring \\ \hline 
  $S^{-1}A$ & localization of $A$ by a multiplicatively closed subset $S$ of $A$\\ \hline
  $A_{\mathfrak{p}}$ & localization of $A$ at a prime ideal $\mathfrak{p}$ of $A$\\ \hline
  $A_{(\mathfrak{p})}$ & homogeneous localization (or degree zero localization) of a graded ring $A$ at a homogeneous prime ideal
                         $\mathfrak{p}$ of $A$\\ \hline
  $A[X_1,\ldots,X_d]$ & polynomial ring in $d$ variables $X_1,\ldots,X_d$ over $A$\\ \hline
  $\mathscr{R}(I)$ & $\bigoplus_{n\ge 0}I^nX^n$ which is called the {\it Rees ring} associated to $I$
                     \index{Rees ring associated to an ideal}\\ \hline
  $F(I)$ & $\mathscr{R}(I)\otimes_A k$. It is called {\it fiber cone}\index{fiber cone} of $I$
                                        (where $(A,\mathfrak{m},k)$ is a local ring)\\\hline
  $\mathscr{T}$ & polynomial ring $A[t_1,\ldots,t_c]$ over $A$ in the cohomology operators $t_j$ with $\deg(t_j) = 2$ for
                  all $1 \le j \le c$\index{polynomial rings in cohomology (Eisenbud) operators}\\ \hline
  $\mathscr{S}$ & bigraded ring $\mathscr{R}(I)[t_1,\ldots,t_c]$, where we set $\deg(t_j) = (0,2)$ for all $1 \le j \le c$
                  and $\deg(u) = (s,0)$ if $\deg(u)$ in $\mathscr{R}(I)$ is $s$
                  \index{bigraded ring in cohomology operators}\\ \hline
  \end{tabular}
  \newpage \noindent
  {\bf Ideals}
  \par \noindent
  \begin{tabular}{|p{2.5cm}|p{12.5cm}|}
  \hline
  $\ann_A(X)$ & $\{ a \in A : ax = 0\mbox{ for all }x \in X\}$, {\it annihilator} of $X$, where $X \subseteq M$
                \index{annihilator of an element}\index{annihilator of a module}\\ \hline
  $\sqrt{I}$ & $\{ a \in A : a^n \in I\mbox{ for some }n \ge 1\}$, {\it radical} of $I$\index{radical of an ideal}\\ \hline 
  $\Soc(A)$ & $\{ a \in A : a\mathfrak{m} = 0 \}$, {\it socle} of a local ring $(A,\mathfrak{m})$\index{socle of a local ring}\\ \hline
  \end{tabular}
  \par \noindent
  {\bf Modules}
  \par \noindent
  \begin{tabular}{|p{2.5cm}|p{12.5cm}|}
  \hline
  $\widehat{M}$ & $\mathfrak{m}$-adic completion of $M$, where $M$ is an $(A,\mathfrak{m})$-module \\ \hline
  $M_g$ & localization of $M$ by the multiplicatively closed set $\{g^i : i \ge 0\}$\\ \hline
  $M^n$ & direct sum of $n$ copies of $M$, where $n$ is a positive integer\\ \hline
  $\omega$ & canonical module of $A$\\ \hline
  $\Omega_n^A(M)$ & the $n$th syzygy module of $M$, where $n$ is a non-negative integer\\ \hline
  $\Omega_{-n}^A(M)$ & the $n$th cosyzygy module of $M$, where $n$ is a non-negative integer\\ \hline
  $\Image(\Phi)$ & image of a module homomorphism $\Phi$\\ \hline
  $(0 :_M I)$ & $\{ x \in M : Ix = 0 \}$, {\it colon submodule} of $M$\index{colon submodule}\\ \hline
  $\mathscr{R}(I,N)$ & {\it Rees module}\index{Rees module} $\bigoplus_{n\ge 0}I^nN$ of $N$ associated to $I$\\ \hline
  $N[X]$ & $N \otimes_A A[X]$\quad (tensor product of the $A$-modules $N$ and $A[X]$)\\ \hline
  $\gr_I(N)$ & {\it associated graded module}\index{associated graded module}
		$\bigoplus_{n \ge 0} (I^n N/I^{n+1} N)$ of $N$ with respect to $I$\\ \hline
  $\mathcal{N}$ & $\bigoplus_{n \ge 0} N_n$, an $\mathbb{N}$-graded $\mathscr{R}(I)$-module.
                  We set $\mathcal{L} := \bigoplus_{n \ge 0} (N/I^{n+1} N)$ \\ \hline
  \end{tabular}
  \par \noindent
  {\bf Few Invariants}
  \par \noindent
  \begin{tabular}{|p{2.7cm}|p{12.3cm}|}
  \hline
  $\depth(A)$ & depth of a local ring $A$\\ \hline
  $\dim(A)$ & $\sup\{n : \exists~ \mathfrak{p}_0 \subsetneq \mathfrak{p}_1 \subsetneq \cdots \subsetneq
	      \mathfrak{p}_n, ~\mathfrak{p}_i \in \Spec(A) \}$, {\it Krull dimension} of $A$\index{Krull dimension of a ring}\\\hline
  $\dim_A(M)$ & {\it dimension} of an  $A$-module $M$, i.e., $\dim(A/\ann_A(M))$ \index{dimension of a module}\\ \hline
  $\dim(\mathscr{V})$ & dimension of a variety $\mathscr{V}$\\ \hline
  $\rank_k(V)$ & dimension of $V$ as a vector space over a field $k$\\ \hline
  $\mu_A(M)$ & the minimal number of generators of $M$ as an $A$-module\\ \hline
  $\codim(A)$ & $\mu_A(\mathfrak{m}) - \dim(A)$, {\it codimension of a local ring} $(A,\mathfrak{m})$
                \index{codimension of a local ring}\\ \hline
  $\lambda_A(M)$ & length of $M$ as an $A$-module\\ \hline
  $\beta_n^A(M)$ & the $n$th Betti number of $M$\\ \hline
  $\projdim_A(M)$ & projective dimension of an $A$-module $M$\\ \hline
  $\injdim_A(M)$ & injective dimension of an $A$-module $M$\\ \hline
  $\cx_A(M)$ & complexity of $M$\\ \hline
  $\cx_A(M,N)$ & complexity of a pair of $A$-modules $(M,N)$\\ \hline
  $\reg(M)$ & Castelnuovo-Mumford regularity of $M$\\ \hline
  $\limsup_{n \rightarrow \infty} a_n$ & limit supremum value of a sequence $\{a_n\}$\\ \hline
  \end{tabular} 
  \newpage \noindent
  {\bf Homological Tools}
  \par \noindent
  \begin{tabular}{|p{2.5cm}|p{12.5cm}|}
  \hline
  $\Hom_A(M,N)$ & the collection of all $A$-linear maps from $M$ to $N$\\ \hline
  $\Ext^i_A(-,-)$ & the $i$th {\it extension functor}\index{extension functor} (the right derived functors of $\Hom$)\\ \hline
  $\Tor_i^A(-,-)$ & the $i$th {\it torsion functor}\index{torsion functor} (the left derived functors of tensor product)\\ \hline
  $H_{J}^i(M)$ & the $i$th local cohomology module of $M$ with respect to an ideal $J$\\ \hline
  $\Ext_A^{\star}(M,N)$ & $\bigoplus_{i \ge 0}\Ext_A^i(M,N)$, total Ext-module\\ \hline
  $\mathcal{C}(M,N)$ & $\Ext_A^\star(M,N)\otimes_A k$\quad (where $A$ is a local ring with its residue field $k$)\\ \hline 
  $\mathscr{E}(\mathcal{N})$ & $\bigoplus_{n \ge 0} \bigoplus_{i \ge 0}\Ext_A^i(M,N_n)$, where $\mathcal{N}=\bigoplus_{n\ge 0}N_n$\\\hline
  $\mathbb{F}$ & free resolution of $M$\\ \hline
  $\mathbb{F}(m)$ & same complex as $\mathbb{F}$ but the components are twisted by $m$ ($\in \mathbb{Z}$)\\ \hline
  $\Hom_A(\mathbb{F},N)$ & cochain complex induced by applying the functor $\Hom_A(-,N)$ on $\mathbb{F}$\\ \hline
  \end{tabular}
  
\newpage
\thispagestyle{empty}
\cleardoublepage

\pagenumbering{arabic}  

\chapter{Introduction}\label{Introduction}
\thispagestyle{empty}

 In this dissertation, we study the following topics: (1) asymptotic prime divisors over complete intersection rings,
 (2) asymptotic stability of complexities over complete intersection rings, (3) asymptotic linear bounds of Castelnuovo-Mumford
 regularity in multigraded modules, and (4) characterizations of regular local rings via syzygy modules of the residue field.
 
 \footnotetext[1]{The new results appear under the topics (1) and (2) are joint work with T. J. Puthenpurakal; while the results shown
 in topic (4) are joint work with A. Gupta and T. J. Puthenpurakal.}
 
 Throughout this dissertation, unless otherwise specified, {\it all rings} ({\it graded or not}) {\it are assumed to be commutative
 Noetherian rings with identity}.
 
\section{Asymptotic Prime Divisors Related to Ext}
 
 Let $A$ be a ring, and let $M$ be a finitely generated $A$-module. A prime ideal $\mathfrak{p}$ of $A$ is called an {\it associated prime ideal}\index{associated prime ideal} of $M$ if $\mathfrak{p}$ is the annihilator $\ann_A(x)$ of some element $x$ of $M$. The set of all associated prime ideals of $M$ is denoted by $\Ass_A(M)$. It is well-known that $\Ass_A(M)$ is a finite set.
 
 Throughout this section, let $M$ and $N$ be two finitely generated $A$-modules, and let $I$ be an ideal of $A$. In \cite{Bro79}, M. Brodmann proved that the set of associated prime ideals $\Ass_A(M/I^n M)$\index{Brodmann's result} is independent of $n$ for all sufficiently large $n$. Thereafter, L. Melkersson and P. Schenzel generalized Brodmann's result in \cite[Theorem~1]{MS93}\index{Melkersson and Schenzel's result} by showing that
 \[
	  \Ass_A\left( \Tor_i^A(M, A/I^n) \right)
 \]
 is independent of $n$ for all large $n$ and for every fixed $i \ge 0$.
 Later, D. Katz and E. West proved the above result in a more general way: For every fixed $i \ge 0$,
 the sets\index{Katz and West's result}
 \[
  \Ass_A\left( \Tor_i^A(M, N/I^n N) \right) \quad\mbox{and}\quad \Ass_A\left( \Ext_A^i(M, N/I^n N) \right)
 \]
 are independent of $n$ for all sufficiently large $n$ (see \cite[Corollary~3.5]{KW04}). In particular, for every fixed $i \ge 0$,
 one obtains that the sets
 \[
  \bigcup_{n\ge 1} \Ass_A\left( \Tor_i^A(M, N/I^n N) \right) \quad\mbox{and}\quad
  \bigcup_{n\ge 1} \Ass_A\left( \Ext_A^i(M, N/I^n N) \right) \quad\mbox{are finite}.
 \]
 
 The motivation for our main results on associated prime ideals came from the following two questions. They were raised by W. V.
 Vasconcelos \cite[Problem~3.15]{Vas98} and L. Melkersson and P. Schenzel \cite[page 936]{MS93} respectively.
 \index{Vasconcelos's question}\index{Melkersson and Schenzel's question}
 \begin{center}
  (A1)\hfill Is the set $~\displaystyle\bigcup_{i\ge 0}\Ass_A\left(\Ext_A^i(M,A)\right)~$ finite?\hfill \; \\
  (A2)\hfill Is the set $~\displaystyle \bigcup_{n \ge 1} \bigcup_{i \ge 0} \Ass_A\left( \Tor_i^A(M,A/I^n) \right)~$
                            finite? \hfill \;
 \end{center}
 Note that if $A$ is a Gorenstein local ring, then the question (A1) has a positive answer (trivially). If we change the question
 a little, then we may ask the following:
 \begin{center}
  (A3)\hfill Is the set $~\displaystyle \bigcup_{i \ge 0} \Ass_A\left( \Ext_A^i(M,N) \right)~$ finite? \hfill \;
 \end{center}
 This is not known for Gorenstein local rings. However, if $A = Q/({\bf f})$, where $Q$ is a  local ring and
 ${\bf f} = f_1,\ldots,f_c$ is a $Q$-regular sequence, and if $\projdim_Q(M)$ is finite, then the question (A3) has an
 affirmative answer. This can be seen by using the finite generation of $\bigoplus_{i \ge 0} \Ext_A^i(M,N)$ over the polynomial ring
 $A[t_1,\ldots,t_c]$ in the cohomology operators $t_j$ over $A$.
 
 The result in the same spirit proved by T. J. Puthenpurakal is the following:
 
\begin{theorem}{\rm \cite[Theorem~5.1]{Put13}}\label{theorem: Tony's result}\index{Puthenpurakal's result on prime divisors}
 Let $A$ be a local complete intersection ring. Suppose $\mathcal{N} = \bigoplus_{n \ge 0} N_n$ is a finitely generated graded module
 over the Rees ring $\mathscr{R}(I)$. Then
 \[
  \bigcup_{n \ge 0} \bigcup_{i \ge 0} \Ass_A\left( \Ext_A^i(M,N_n) \right) \quad \mbox{is a finite set}.
 \]
 Furthermore, there exist $n_0, i_0$ $(\ge 0)$ such that for all $n \ge n_0$ and $i \ge i_0$, we have
 \begin{align*}
  \Ass_A\left(\Ext_A^{2i}(M,N_n)\right) &= \Ass_A\left(\Ext_A^{2 i_0}(M,N_{n_0})\right), \\
  \Ass_A\left(\Ext_A^{2i+1}(M,N_n)\right) &= \Ass_A\left(\Ext_A^{2 i_0 + 1}(M,N_{n_0})\right).
 \end{align*}
\end{theorem}

 In particular, $\mathcal{N}$ can be taken as $\bigoplus_{n \ge 0} I^n N$ or $\bigoplus_{n \ge 0} I^n N/I^{n+1} N$ in the above
 theorem. Note that if $N \neq 0$, then $\bigoplus_{n \ge 0} N/I^n N$ is not finitely generated as an $\mathscr{R}(I)$-module.
 So we cannot take $\mathcal{N}$ as $\bigoplus_{n \ge 0} N/I^n N$ in Theorem~\ref{theorem: Tony's result}. In this theme,
 T. J. Puthenpurakal \cite[page 368]{Put13} raised the following question:
 \begin{center}
  (A4)\hfill Is the set $~\displaystyle \bigcup_{n \ge 1} \bigcup_{i \ge 0}\Ass_A\left( \Ext_A^i(M,N/I^n N) \right)~$
                  finite? \hfill \;
 \end{center}
 
 In Chapter~\ref{Chapter: Asymptotic Prime Divisors over Complete Intersection Rings}, we show that the question (A4) has an
 affirmative answer for local complete intersection rings. We also analyze the asymptotic behaviour of the sets
 \begin{center}
  $\Ass_A\left( \Ext_A^i(M, N/I^n N) \right)$\quad if $n$ and $i$ both tend to $\infty$.
 \end{center}
 If $A$ is a local complete intersection ring, then we prove that there exist non-negative integers $n_0, i_0$ such that for all
 $n \ge n_0$ and $i \ge i_0$, we have
 \begin{align*}
  \Ass_A\left(\Ext_A^{2i}(M,N/I^nN)\right) &= \Ass_A\left(\Ext_A^{2 i_0}(M,N/I^{n_0}N)\right), \\
  \Ass_A\left(\Ext_A^{2i+1}(M,N/I^nN)\right) &= \Ass_A\left(\Ext_A^{2 i_0 + 1}(M,N/I^{n_0}N)\right).
 \end{align*}
 We also show the analogous results for complete intersection rings which arise in algebraic geometry.
 
 Recall that a local ring $A$ is said to be a {\it local complete intersection ring}\index{complete intersection ring} if
 its completion $\widehat{A} = Q/({\bf f})$, where $Q$ is a complete regular local ring and ${\bf f} = f_1,\ldots,f_c$ is a $Q$-regular
 sequence. To prove our results, we may assume that $A$ is complete because of the following well-known fact: For a finitely generated $A$-module
 $D$, we have
 \[
	  \Ass_A(D) = \left\{ \mathfrak{q} \cap A : \mathfrak{q} \in \Ass_{\widehat{A}}\left(D \otimes_A \widehat{A}\right) \right\}.
 \]
 So, without loss of generality, we may assume that $A$ is a quotient of a regular local ring modulo a regular sequence. Then the
 desired results follow from more general results appeared below. Let us fix the following hypothesis for the rest of this section.
 
\begin{hypothesis}\label{hypothesis: A = Q/f}
	Let $Q$ be a ring of finite Krull dimension, and let ${\bf f} = f_1,\ldots,f_c$ be a $Q$-regular sequence. Set $A := Q/({\bf f})$. Let $M$ and $N$ be finitely generated $A$-modules, where $\projdim_Q(M)$ is finite. Let $I$ be an ideal of $A$.
\end{hypothesis}

 We prove the announced finiteness and stability results in this set-up. One of the main ingredients we use in the proofs of our
 results is the following theorem due to T. J. Puthenpurakal. This theorem concerns the finite generation of a family of Ext-modules.
 Here we need to understand the theory of cohomology operators which gives the bigraded module structure used in the following theorem.
 For details, we refer to Section~\ref{Module structure}.
 
\begin{theorem}{\rm \cite[Theorem~1.1]{Put13}}\label{theorem: finite generation result}
 With the {\rm Hypothesis~\ref{hypothesis: A = Q/f}}, let $\mathcal{N} = \bigoplus_{n \ge 0} N_n$ be a finitely generated graded module
 over the Rees ring $\mathscr{R}(I)$. Then 
 \[
  \bigoplus_{n \ge 0} \bigoplus_{i \ge 0} \Ext_A^i(M,N_n)
 \]
 is a finitely generated bigraded $\mathscr{S} := \mathscr{R}(I)[t_1,\ldots,t_c]$-module, where
 \[
  t_j : \Ext_A^i(M,N_n) \longrightarrow \Ext_A^{i+2}(M,N_n) \quad (j = 1,\ldots,c)
 \]
 are the cohomology operators over $A$.
\end{theorem}

 With the Hypothesis~\ref{hypothesis: A = Q/f}, we prove that the set
 \[
  \bigcup_{n \ge 1} \bigcup_{i \ge 0} \Ass_A\left( \Ext_A^i(M,N/I^n N) \right) \quad\mbox{is finite};
 \]
 see Theorem~\ref{theorem:Q mod f finiteness}. After having this finiteness result, we focus on the asymptotic stability of the sets
 of associated prime ideals which occurs periodically after a certain stage. We show that there exist $n', i' \ge 0$ such that for
 all $n \ge n'$ and $i \ge i'$, we have
 \begin{align*}
  \Ass_A\left(\Ext_A^{2i}(M,N/I^nN)\right) &= \Ass_A\left(\Ext_A^{2 i'}(M,N/I^{n'}N)\right), \\
  \Ass_A\left(\Ext_A^{2i+1}(M,N/I^nN)\right) &= \Ass_A\left(\Ext_A^{2 i' + 1}(M,N/I^{n'}N)\right);
 \end{align*}
 see Theorem~\ref{theorem:Q mod f stability}. To prove this result, we take advantage of the notion of Hilbert function. We set
 \[
  V_{(n,i)} := \Ext_A^i(M,N/I^n N) \quad\mbox{for all } n, i \ge 0.
 \]
 Since $\bigcup_{n \ge 0} \bigcup_{i \ge 0} \Ass_A\left( V_{(n,i)} \right)$ is finite, it is enough to prove that for each
 \[
  \mathfrak{p} \in \bigcup_{n \ge 0} \bigcup_{i \ge 0} \Ass_A\left( V_{(n,i)} \right),
 \]
 there exist some $n_l, i_l \ge 0$ such that exactly one of the following alternatives must hold:
 \begin{align*}
  \mbox{either}\quad \mathfrak{p} &\in \Ass_A\left( V_{(n,2i+l)} \right)    \quad\mbox{for all }n \ge n_l \mbox{ and } i \ge i_l; \\
  \mbox{or}\quad     \mathfrak{p} &\notin \Ass_A\left( V_{(n,2i+l)} \right) \quad\mbox{for all }n \ge n_l \mbox{ and } i \ge i_l,
 \end{align*}
 where $l = 0, 1$. Localizing at $\mathfrak{p}$, and replacing $A_{\mathfrak{p}}$ by $A$ and $\mathfrak{p} A_{\mathfrak{p}}$ by
 $\mathfrak{m}$, we may assume that $A$ is a local ring with maximal ideal $\mathfrak{m}$ and residue field $k$. For each
 $l \in \{0, 1\}$, it is now enough to prove that there exist some $n_l, i_l \ge 0$ such that
 \begin{align*}
  \mbox{either}\quad \mathfrak{m} &\in \Ass_A\left( V_{(n,2i+l)} \right) \quad \mbox{for all }n \ge n_l \mbox{ and } i \ge i_l; \\
  \mbox{or}\quad     \mathfrak{m} &\notin \Ass_A\left( V_{(n,2i+l)} \right) \quad \mbox{for all }n \ge n_l \mbox{ and } i \ge i_l,
 \end{align*}
 which is equivalent to that
 \begin{align*}
  \mbox{either}\quad \Hom_A\left( k, V_{(n,2i+l)} \right) &\neq 0 \quad\mbox{for all }n \ge n_l \mbox{ and } i \ge i_l; \\
  \mbox{or}\quad     \Hom_A\left( k, V_{(n,2i+l)} \right) &= 0\quad\mbox{for all }n \ge n_l \mbox{ and } i \ge i_l.
 \end{align*}
 We show this in Lemma~\ref{lemma:well behaved polynomial}. 

\section{Asymptotic Stability of Complexities}
 
 Let $A$ be a local ring with residue field $k$. Let $M$ and $ N $ be finitely generated  $A$-module. The integer
 \[
  \beta_n^A(M) := \rank_k\left( \Ext_A^n(M,k) \right)
 \]
 is called the $n$th {\it Betti number}\index{Betti number} of $M$. It is equal to the rank of $F_n$ in a minimal free resolution
 $\mathbb{F}$ of $M$. The notion of complexity of a module was introduced by L. L. Avramov in \cite[Definition~(3.1)]{Avr89}. The
 {\it complexity}\index{complexity of a module} of $M$, denoted $\cx_A(M)$ is the smallest non-negative integer $b$ such that
 $\beta_n^A(M) \le a n^{b-1}$ for some real number $a > 0$ and for all sufficiently large $n$. If no such $b$ exists, then set
 $\cx_A(M) := \infty$.
 
 The complexity of a pair of modules $(M,N)$, introduced in \cite{AB00} by L. L. Avramov and R.-O. Buchweitz,
 \index{complexity of a pair of modules}is defined to be the number
 \[
  \cx_A(M,N) := \inf\left\{ b \in \mathbb{N} ~\middle|
    \begin{array}{c}
      \mu_A\left( \Ext_A^n(M,N) \right) \le a n^{b-1} \mbox{ for some} \\
      \mbox{ real number } a > 0 \mbox{ and for all } n \gg 0
    \end{array}
   \right\}, 
 \]
 where $\mu_A(D)$ denotes the minimal number of generators of an $A$-module $D$. Clearly, the complexity $\cx_A(M,N)$ measures
 `the size' of $\Ext_A^{*}(M,N)$. This notion encompasses the asymptotic invariant of $M$: its complexity $\cx_A(M) = \cx_A(M,k)$,
 measuring the polynomial rate of growth at infinity of its minimal free resolution.
 
 Let $A$ be a local complete intersection ring, and let $I$ be an ideal of $A$. Puthenpurakal (in \cite[Theorem 7.1]{Put13})
 proved that $\cx_A(M, I^n N)$\index{Puthenpurakal's result on complexity} is constant for all sufficiently large $n$.
 In Chapter~\ref{Chapter: Asymptotic Stability of Complexities} of this dissertation, we show that
 \begin{center}
  (C1) \hfill $\cx_A(M, N/I^n N)$\quad is independent of $n$ for all sufficiently large $n$. \hfill \;
 \end{center}
 To prove this result, we use the notion of support variety which was introduced by Avramov and Buchweitz
 in the same article \cite[2.1]{AB00}.

\section{Castelnuovo-Mumford Regularity of Powers of Several Ideals}
 
 Let $A$ be a standard $\mathbb{N}$-graded algebra, where {\it standard}\index{standard graded ring} means $A$ is generated
 by elements of $A_1$. Let $A_{+}$ be the ideal of $A$ generated by the homogeneous elements of positive degree. Let $M$ be a finitely generated 
 $\mathbb{N}$-graded $A$-module. For every integer $i \ge 0$, we denote the $i$th local cohomology module of $M$ with respect to
 $A_{+}$ by $H_{A_{+}}^i(M)$. The {\it Castelnuovo-Mumford regularity}\index{Castelnuovo-Mumford regularity} of $M$ is the invariant
 \[
  \reg(M) := \max\left\{ \eend\left( H_{A_{+}}^i(M) \right) + i : i \ge 0 \right\},
 \]
 where $\eend(D)$ denotes the maximal non-vanishing degree\index{end of an $\mathbb{N}$-graded module} of an $\mathbb{N}$-graded
 $A$-module $D$ with the convention $\eend(D) = - \infty$ if $D = 0$. It is a natural extension of the usual definition of the
 Castelnuovo-Mumford regularity in the case $A$ is a polynomial ring over a field.
 
 Let $S = k[X_1,\ldots,X_d]$ be a polynomial ring over a field $k$ with its usual grading, i.e., each $X_i$ has degree $1$. By the
 Hilbert's Syzygy Theorem\index{Hilbert's Syzygy Theorem}, every finitely generated $\mathbb{N}$-graded $S$-module $N$ has a
 finite graded minimal free resolution:
 \[
	  0 \longrightarrow F_p \longrightarrow \cdots \longrightarrow F_1 \longrightarrow F_0 \longrightarrow N \longrightarrow 0,
 \]
 where $F_i = \bigoplus_{j = 1}^{a_i} S(- b_{ij})$ for some integers $b_{ij}$. Then
 \[
	  \reg(N) = \max_{i,j}\{ b_{ij} - i \}.
 \]
 One can write it in terms of the maximal non-vanishing degree of Tor-modules:\index{Castelnuovo-Mumford regularity}
 \[
	  \reg(N) = \max\left\{ \eend\left( \Tor_i^S(N,k) \right) - i : i \ge 0 \right\}.
 \]
 
 There has been a surge of interest on the behaviour of the function $\reg(I^n)$, where $I$ is a homogeneous ideal in a polynomial
 ring over a field. A detailed survey of this topic can be found in Section~\ref{Survey on regularity} of
 Chapter~\ref{Review of Literature}. In \cite[Theorem~3.6]{Swa97}, I. Swanson proved that\index{Swanson's result} $\reg(I^n) \le kn$
 for all $n \ge 1$, where $k$ is some integer. Later, S. D. Cutkosky, J. Herzog and N. V. Trung \cite[Theorem~1.1]{CHT99}
 and\index{Cutkosky, Herzog and Trung}\index{Kodiyalam's result} V. Kodiyalam \cite[Theorem~5]{Kod00} independently showed that
 $\reg(I^n)$ can be expressed as a linear function of $n$ for all sufficiently large $n$.
 
 Recently, N. V. Trung and H.-J. Wang \cite[Theorem~3.2]{TW05} proved that if $A$ is a standard $\mathbb{N}$-graded ring, $I$ is a homogeneous ideal of $A$, and $M$ is a finitely generated $\mathbb{N}$-graded $A$-module,
 then $\reg(I^n M)$ is asymptotically a linear function of $n$.
 
 In this context, a natural question arises: What happens when we consider several ideals instead of just considering one ideal?
 More precisely, if $I_1,\ldots,I_t$ are homogeneous ideals of $A$, and $M$ is a finitely generated $\mathbb{N}$-graded $A$-module,
 then what will be the behaviour of $\reg(I_1^{n_1}\cdots I_t^{n_t} M)$ as a function of $(n_1,\ldots,n_t)$?

 Let $A = A_0[x_1,\ldots,x_d]$ be a standard $\mathbb{N}$-graded algebra over an Artinian local ring $(A_0,\mathfrak{m})$. In particular, $A$ can be the coordinate ring of a projective variety over a field with usual grading. Let $I_1,\ldots,I_t$ be homogeneous ideals of $A$, and let $M$ be a finitely generated $\mathbb{N}$-graded $A$-module. One of the main results of this
 dissertation is the following: There exist two integers $k_1$ and $k'_1$ such that
 \begin{center}
  (R1) \hfill $\reg\left(I_1^{n_1}\cdots I_t^{n_t} M\right) \le (n_1 + \cdots + n_t) k_1 + k'_1$\quad
  for all ~$n_1,\ldots,n_t \in \mathbb{N}$. \hfill \;
 \end{center}
 In Chapter~\ref{Chapter: Asymptotic linear bounds of Castelnuovo-Mumford regularity}, we prove this result in a quite general set-up.
 As a consequence, we also obtain the following: If $A_0$ is a field, then there exist two integers $k_2$ and $k'_2$ such that
 \begin{center}
  (R2) \hfill $\reg\left( \overline{I_1^{n_1}} \cdots \overline{I_t^{n_t}} M \right)
  \le (n_1 + \cdots + n_t) k_2 + k'_2$\quad for all ~$n_1,\ldots,n_t \in \mathbb{N}$, \hfill \;
 \end{center}
 where $\overline{I}$ denotes the integral closure of an ideal $I$ of $A$.
 
 The basic technique of the proof of our results is the use of $\mathbb{N}^{t+1}$-grading structures on
 \[
	  \bigoplus_{\underline{n} \in \mathbb{N}^t} \left( I_1^{n_1} \cdots I_t^{n_t} M \right) \quad\mbox{and}\quad
	  \bigoplus_{\underline{n} \in \mathbb{N}^t} \left( \overline{I_1^{n_1}} \cdots \overline{I_t^{n_t}} M \right).
 \]
 Let $R = A[I_1 T_1,\ldots, I_t T_t]$ be the Rees algebra of $I_1,\ldots,I_t$ over the graded ring $A$, and let
 $L = M[I_1 T_1,\ldots, I_t T_t]$ be the Rees module of $M$ with respect to the ideals $I_1,\ldots,I_t$. We give
 $\mathbb{N}^{t+1}$-grading structures on $R$ and $L$ by setting $(\underline{n},i)$th graded components of $R$ and $L$ as the $i$th
 graded components of the $\mathbb{N}$-graded $A$-modules $I_1^{n_1}\cdots I_t^{n_t} A$ and $I_1^{n_1}\cdots I_t^{n_t} M$
 respectively. Then clearly, $R$ is an $\mathbb{N}^{t+1}$-graded ring and $L$ is a finitely generated
 $\mathbb{N}^{t+1}$-graded $R$-module. Note that $R$ is not necessarily standard as an $\mathbb{N}^{t+1}$-graded ring. Also note that
 for every $\underline{n} \in \mathbb{N}^t$, we have
 \begin{align*}
	  R_{(\underline{n},\star)} &:= \bigoplus_{i \in \mathbb{N}} R_{(\underline{n},i)} = I_1^{n_1}\cdots I_t^{n_t} A \\
	  \mbox{and }\quad L_{(\underline{n},\star)} &:= \bigoplus_{i \in \mathbb{N}} L_{(\underline{n},i)} = I_1^{n_1}\cdots I_t^{n_t} M.
 \end{align*}
 Since $R = A[I_1 T_1,\ldots, I_t T_t] = R_{(\underline{0},\star)}[R_{(\underline{e}^1,\star)},\ldots,R_{(\underline{e}^t,\star)}]$,
 we may consider $R = \bigoplus_{\underline{n}\in \mathbb{N}^t}R_{(\underline{n},\star)}$ as a standard
 $\mathbb{N}^t$-graded ring.
 
 In a similar way as above, we can give an $\mathbb{N}^{t+1}$-graded $R$-module structure on
 \[
	  L := \bigoplus_{\underline{n} \in \mathbb{N}^t} \left( \overline{I_1^{n_1}} \cdots \overline{I_t^{n_t}} M \right).
 \]
 In this case also, $L$ is a finitely generated $\mathbb{N}^{t+1}$-graded $R$-module provided $A_0$ is a field; see
 Example~\ref{example 2 satisfying the hypothesis}. Keeping these two examples in mind, let us fix the following hypothesis for the
 rest of this section.
 
 \begin{hypothesis}\label{hypothesis: multigrading structure}
  Let
  \[
    R = \bigoplus_{(\underline{n},i)\in \mathbb{N}^{t+1}} R_{(\underline{n},i)}
  \]
  be an $\mathbb{N}^{t+1}$-graded ring, {\it which need not be standard}. Let
  \[
    L = \bigoplus_{(\underline{n},i)\in \mathbb{N}^{t+1}} L_{(\underline{n},i)}
  \]
  be a finitely generated $\mathbb{N}^{t+1}$-graded $R$-module. For each $\underline{n} \in \mathbb{N}^t$, we set
  \[ 
     R_{(\underline{n},\star)} := \bigoplus_{i \in \mathbb{N}} R_{(\underline{n},i)} \quad\mbox{ and }\quad
     L_{(\underline{n},\star)} := \bigoplus_{i \in \mathbb{N}} L_{(\underline{n},i)}.
  \]
  Also set $A := R_{(\underline{0},\star)}$. Suppose
  $R = \bigoplus_{\underline{n} \in \mathbb{N}^t} R_{(\underline{n},\star)}$ and $A = R_{(\underline{0},\star)}$
  are standard as $\mathbb{N}^t$-graded ring and $\mathbb{N}$-graded ring respectively. Assume that $A_0 = R_{(\underline{0},0)}$ is
  an Artinian local ring. Suppose $A = A_0[x_1,\ldots,x_d]$ for some $x_1,\ldots,x_d \in A_1$.
  Set $A_{+} := \langle x_1,\ldots,x_d \rangle$.
 \end{hypothesis}

 With the Hypothesis~\ref{hypothesis: multigrading structure}, we now give $\mathbb{N}^t$-grading structures on
 \[ 
   R = \bigoplus_{\underline{n} \in \mathbb{N}^t} R_{(\underline{n},\star)} \quad\mbox{and}\quad
   L = \bigoplus_{\underline{n} \in \mathbb{N}^t} L_{(\underline{n},\star)}
 \]
 in the obvious way, i.e., by setting $R_{(\underline{n},\star)}$ and $L_{(\underline{n},\star)}$ as the $\underline{n}$th graded
 components of $R$ and $L$ respectively. Then clearly, $R = \bigoplus_{\underline{n} \in \mathbb{N}^t} R_{(\underline{n},\star)}$ is
 a standard $\mathbb{N}^t$-graded ring and $L = \bigoplus_{\underline{n} \in \mathbb{N}^t} L_{(\underline{n},\star)}$ is a
 finitely generated $\mathbb{N}^t$-graded $R$-module. From now onwards, by $R$ and $L$, we mean $\mathbb{N}^t$-graded ring
  $\bigoplus_{\underline{n} \in \mathbb{N}^t} R_{(\underline{n},\star)}$ and $\mathbb{N}^t$-graded $R$-module
  $\bigoplus_{\underline{n} \in \mathbb{N}^t} L_{(\underline{n},\star)}$ (satisfying the
  Hypothesis~\ref{hypothesis: multigrading structure}) respectively. To prove our main results on regularity, it is now enough to
  show that there exist two integers $k$ and $k'$ such that
 \begin{center}
  (R3) \hfill $\reg\left( L_{(\underline{n},\star)} \right)
                   \le (n_1 + \cdots + n_t) k + k'$ \quad for all $~\underline{n} \in \mathbb{N}^t$. \hfill \;
 \end{center}
 
 We prove (R3) in several steps by using induction on an invariant which we introduce here. We call this invariant the saturated
 dimension of a multigraded module. Suppose $\mathscr{M} = \bigoplus_{\underline{n} \in \mathbb{N}^t} \mathscr{M}_{\underline{n}}$ is a
 finitely generated $\mathbb{N}^t$-graded module over a standard $\mathbb{N}^t$-graded ring
 $\mathscr{R} = \bigoplus_{\underline{n} \in \mathbb{N}^t} \mathscr{R}_{\underline{n}}$. Then we prove that there exists
 $\underline{v} \in \mathbb{N}^t$ such that
 \begin{align*}
  \ann_{\mathscr{R}_0}(\mathscr{M}_{\underline{n}}) &= \ann_{\mathscr{R}_0}(\mathscr{M}_{\underline{v}})
  \quad\mbox{for all } \underline{n} \ge \underline{v},\\
  \mbox{and hence }~ \dim_{\mathscr{R}_0}(\mathscr{M}_{\underline{n}}) &= \dim_{\mathscr{R}_0}(\mathscr{M}_{\underline{v}})
  \quad\mbox{for all } \underline{n} \ge \underline{v};
 \end{align*}
 see Lemma~\ref{lemma: annihilator stability for multigraded modules}. We call such a point $\underline{v} \in \mathbb{N}^t$ an
 annihilator stable point of $\mathscr{M}$. Then the saturated dimension of $\mathscr{M}$ is defined to be
 $s(\mathscr{M}) := \dim_{\mathscr{R}_0}(\mathscr{M}_{\underline{v}})$. We use induction on the saturated dimension $s(L)$ of
 $L= \bigoplus_{\underline{n} \in \mathbb{N}^t} L_{(\underline{n},\star)}$ in order to prove (R3).
 
 In the base case, i.e., in the case when $s(L) = 0$, we have $\dim_A\left( L_{(\underline{n},\star)} \right) = 0$ for all
 $\underline{n} \ge \underline{v}$, where $\underline{v}$ is an annihilator stable point of $L$. Then, in view of Grothendieck's Vanishing
 Theorem, we obtain that
 \begin{align*}
  & H_{A_{+}}^i( L_{(\underline{n},\star)} ) = 0 \quad\mbox{for all }~ i > 0 ~\mbox{ and }~\underline{n} \ge \underline{v},\\
  \mbox{which gives }~ & \reg(L_{(\underline{n},\star)}) = \max\left\{\mu : H_{A_{+}}^0(L_{(\underline{n},\star)})_{\mu} \neq 0\right\}
   \quad\mbox{for all }~\underline{n} \ge \underline{v}.
 \end{align*}
 In this situation, we show the linear boundedness result (in Theorem~\ref{theorem: bounds of regularity for saturated dimension 0}) by
 using the following fact: There exists a positive integer $k$ such that
 \[
  (A_{+})^k L_{(\underline{n},\star)} \cap H_{A_{+}}^0(L_{(\underline{n},\star)}) =0\quad\mbox{ for all }~\underline{n} \in\mathbb{N}^t;
 \]
 see Lemma~\ref{lemma: Artin-Rees}.

 The inductive step is shown in Theorem~\ref{theorem: bounds of regularity for saturated dimension positive} by using the following
 well-known result on regularity: For a finitely generated $\mathbb{N}$-graded $A$-module $N$, we have
 \[
  \reg(N) \le \max\{\reg(0 :_N x), \reg(N/xN) - l + 1\},
 \]
 where $x$ is a homogeneous element of $A$ with degree $l \ge 1$. In this step, we prove that there exist $\underline{u}\in\mathbb{N}^t$
 and an integer $k$ such that
 \[
  \reg(L_{(\underline{n},\star)}) < (n_1 + \cdots + n_t) k + k \quad\mbox{for all }~\underline{n} \ge \underline{u}.
 \]
 In particular, this shows that if $t = 1$, then there exist two integers $k$ and $k'$ such that
 \[
  \reg(L_{(n,\star)}) \le n k + k' \quad\mbox{for all }~n \in \mathbb{N}.
 \]
 
 Finally, in Theorem~\ref{theorem: bounds of regularity for multigraded module}, we prove (R3) by using induction on $t$.
 
\section{Characterizations of Regular Rings via Syzygy Modules}
 
 In the present section, $A$ always denotes a local ring with maximal ideal $\mathfrak{m}$ and residue field $k$.
 For every non-negative integer $n$, we denote the $n$th syzygy module of $k$ by $\Omega_n^A(k)$.
 
 One of the most influential results in commutative algebra is the result of Auslander, Buchsbaum and Serre: The local ring $A$
 is\index{Auslander-Buchsbaum-Serre's Criterion}\index{regular local ring} regular if and only if $\projdim_A(k)$ is finite.
 Note that $\projdim_A(k)$ is finite if and only if some syzygy module of $k$ is a
 free $A$-module. There are a number of characterizations of regular local rings in terms of syzygy
 modules of the residue field. In \cite[Corollary~1.3]{Dut89}, S. P. Dutta gave the following characterization of regular local rings. 
 
\begin{theorem}[Dutta]\label{theorem: Dutta}\index{Dutta's result}
 $A$ is regular if and only if $\Omega_n^A(k)$ has a non-zero free direct summand for some integer $n \ge 0$.
\end{theorem}
 
 Later, R. Takahashi generalized Dutta's result by giving a characterization of regular local rings via the existence of a
 semidualizing direct summand of some syzygy module of the residue field. Let us recall the definition of a semidualizing module.
 
\begin{definition}[\cite{Gol84}]\label{definition: semidualizing module}
 A finitely generated $A$-module $M$ is said to be a {\it semidualizing module}\index{semidualizing module} if the following hold:
 \begin{enumerate}
  \item[(i)] The natural homomorphism $A \longrightarrow \Hom_A(M,M)$ is an isomorphism.
  \item[(ii)] $\Ext_A^i(M,M) = 0$ for all $i \ge 1$.
 \end{enumerate}
\end{definition}

 Note that $A$ itself is a semidualizing $A$-module. So the following theorem generalizes the above result of Dutta.
 
\begin{theorem}{\rm \cite[Theorem~4.3]{Tak06}}\label{theorem: Takahashi; regular; semidualizing}\index{Takahashi's result}
 $A$ is regular if and only if $\Omega_n^A(k)$ has a semidualizing direct summand for some integer $n \ge 0$.
\end{theorem}

 If $A$ is a Cohen-Macaulay local ring with canonical module $\omega$, then $\omega$ is a semidualizing $A$-module. Therefore, as a
 corollary of Theorem~\ref{theorem: Takahashi; regular; semidualizing}, Takahashi obtained the following:

\begin{corollary}{\rm \cite[Corollary~4.4]{Tak06}}\label{corollary: Takahashi; regular; canonical module}
 Let $A$ be a Cohen-Macaulay local ring with canonical module $\omega$. Then $A$ is regular if and only if $\Omega_n^A(k)$ has a
 direct summand isomorphic to $\omega$ for some integer $n \ge 0$.
\end{corollary}
 
 Now recall that the canonical module (if exists) over a Cohen-Macaulay local ring has finite injective dimension. Also it is
 well-known that $A$ is regular if and only if $k$ has finite injective dimension. So, in this theme, a natural question arises
 that ``if $\Omega_n^A(k)$ has a non-zero direct summand of finite injective dimension for some integer $n \ge 0$, then is the ring
 $A$ regular?''. In the present study, we show that this question has an affirmative answer
 (see Theorem~\ref{theorem: characterization of RLR, injdim}).
 
 Kaplansky conjectured\index{Kaplansky's conjecture} that if some power of the maximal ideal of $A$ is non-zero and of finite
 projective dimension, then $A$ is regular. In \cite[Theorem~1.1]{LV68}, G. Levin and W. V. Vasconcelos proved this conjecture.
 In fact, their result is even stronger:

\begin{theorem}[Levin and Vasconcelos]\index{Levin and Vasconcelos's Theorem}
 If $M$ is a finitely generated $A$-module such that $\mathfrak{m} M$ is non-zero and of finite projective dimension {\rm (}or of
 finite injective dimension{\rm )}, then $A$ is regular.
\end{theorem}
 
 In \cite[Proposition~7]{Mar96}, A. Martsinkovsky generalized Dutta's result in the following direction. He also showed that the
 above result of Levin and Vasconcelos is a special case of the following theorem:
 
\begin{theorem}[Martsinkovsky]\label{theorem: Martsinkovsky}\index{Martsinkovsky's result}
 If a finite direct sum of syzygy modules of $k$ maps onto a non-zero $A$-module of finite projective dimension,
 then $A$ is regular.
\end{theorem}
 
 In this direction, we prove the following result which considerably strengthens
 Theorem~\ref{theorem: Takahashi; regular; semidualizing}. The proof presented here is very simple and elementary.
 
\begin{theorem}[See Corollary~\ref{corollary: RLR and surjection onto semidualizing}]
                    \label{Theorem I: surjection onto semidualizing}
 If a finite direct sum of syzygy modules of $k$ maps onto a semidualizing $A$-module, then $A$ is regular.
\end{theorem}
 
 Furthermore, we raise the following question:
 
\begin{question}\label{question: surjection onto finite injdim}
 If a finite direct sum of syzygy modules of $k$ maps onto a non-zero $A$-module of finite injective dimension,
 then is the ring $A$ regular?
\end{question}
 
 In Chapter~\ref{Chapter: Characterizations of Regular Local Rings via Syzygy Modules},
 we give a partial answer to this question as follows:
 
\begin{theorem}[See Corollary~\ref{corollary: RLR and surjection onto finite injdim}]
	       \label{Theorem II: surjection onto finite inj dim}
 If a finite direct sum of syzygy modules of $k$ maps onto a non-zero maximal Cohen-Macaulay $A$-module $L$ of finite
 injective dimension, then $A$ is regular.
\end{theorem}

 If $A$ is a Cohen-Macaulay local ring with canonical module $\omega$, then one can take $L = \omega$ in the above theorem.
 
 Theorems~\ref{Theorem I: surjection onto semidualizing} and \ref{Theorem II: surjection onto finite inj dim} are deduced as
 consequences of a more general result appeared below. Let $\mathcal{P}$ be a property of modules over local rings. We say that
 $\mathcal{P}$ is a $(*)$-property if $\mathcal{P}$ satisfies the following:\index{$(*)$-property of modules}
 \begin{enumerate}[(i)]
  \item An $A$-module $M$ satisfies $\mathcal{P}$ implies that the $A/(x)$-module $M/xM$ satisfies $\mathcal{P}$, where $x \in A$ is an
        $A$-regular element.
  \item An $A$-module $M$ satisfies $\mathcal{P}$ and $\depth(A) = 0$ together imply that $\ann_A(M) = 0$.
 \end{enumerate}
 It can be noticed that the properties `semidualizing modules' and `non-zero maximal Cohen-Macaulay modules of finite injective
 dimension' are two examples of $(*)$-properties; see Examples~\ref{example: star property: semidualizing}
 and \ref{example: star property: MCM, finite injdim}. Therefore we obtain Theorems~\ref{Theorem I: surjection onto semidualizing}
 and \ref{Theorem II: surjection onto finite inj dim} as corollaries of the following general result. Here we denote a finite
 collection of non-negative integers by $\Lambda$.
 
\begin{theorem}[See Theorem~\ref{theorem: RLR and surjection onto star module}]
		\label{theorem: Intro: RLR and surjection onto star module}
 Assume that $\mathcal{P}$ is a $(*)$-property. Let
 \[
   f : \bigoplus_{n \in \Lambda} {\left( \Omega_n^A(k) \right)}^{j_n} \longrightarrow L
   \quad\quad \mbox{$(j_n \ge 1$ for each $n \in \Lambda)$}
 \]
 be a surjective $A$-module homomorphism, where $L$ {\rm(}$\neq 0${\rm)} satisfies $\mathcal{P}$. Then $A$ is regular.
\end{theorem}
 
 We establish a relationship between the socle of the ring and the annihilator of the syzygy modules: If $A \neq k$
 {\rm (}i.e., if $\mathfrak{m} \neq 0${\rm )}, then
 \[
  \Soc(A) \subseteq \ann_A\left( \Omega_n^A(k) \right) \quad \mbox{for all } n \ge 0;
 \]
 see Lemma~\ref{lemma: socle, syzygy}. We use this fact in order to prove
 Theorem~\ref{theorem: Intro: RLR and surjection onto star module} when $\depth(A) = 0$. Considering this as the base case, the
 general case follows from an inductive argument; see the proof of Theorem~\ref{theorem: RLR and surjection onto star module}.
 
 We obtain one new characterization of regular local rings. It follows from Dutta's result (Theorem~\ref{theorem: Dutta})
 that $A$ is regular if and only if some syzygy module of $k$ has a non-zero direct summand of finite projective dimension.
 We prove the following counterpart for injective dimension.
 
\begin{theorem}[See Theorem~\ref{theorem: characterization of RLR, injdim}]
 $A$ is regular if and only if some syzygy module of $k$ has a non-zero direct summand of finite injective dimension.
\end{theorem}

 Moreover, this result has a dual companion which says that $A$ is regular if and only if some cosyzygy module of $k$ has a non-zero
 finitely generated direct summand of finite projective dimension; see Corollary~\ref{corollary: dual result, cosyzygy}.
 
 Till now we have considered surjective homomorphisms from a finite direct sum of syzygy modules of $k$ onto a `special module'.
 One may ask ``what happens if we consider injective homomorphisms from a `special module' to a finite direct sum of syzygy modules
 of $k$?". More precisely, if
 \[
  f : L \longrightarrow \bigoplus_{n \in \Lambda} {\left( \Omega_n^A(k) \right)}^{j_n}
 \]
 is an injective $A$-module homomorphism, where $L$ is non-zero and of finite projective dimension (or of finite injective dimension),
 then is the ring $A$ regular? We give an example which shows that $A$ is not necessarily a regular local ring in this situation;
 see Example~\ref{example: counter, injective map}.

\newpage
\thispagestyle{empty}
\cleardoublepage

\chapter{Review of Literature}\label{Review of Literature}
 
 The literature survey in this chapter concentrates on the following subjects: (1) asymptotic prime divisors related to derived
 functors Ext and Tor, (2) Castelnuovo-Mumford regularity of powers of ideals, and (3) characterizations of regular local rings via
 syzygy modules of the residue field.

\section{Asymptotic Prime Divisors Related to Derived Functors}\label{Survey on asymptotic prime divisors}
 
 Throughout this section, let $A$ be a ring. Let $I$ be an ideal of $A$ and $M$ be a finitely generated $A$-module.
 M. Brodmann \cite{Bro79} proved that the set of associated prime ideals $\Ass_A(M/I^n M)$ is independent of $n$ for all sufficiently
 large $n$\index{Brodmann's result}. He deduced this result by proving that $\Ass_A(I^n M/I^{n+1} M)$ is independent of $n$ for all
 large $n$.
 
 Thereafter, in \cite[Theorem~1]{MS93}, L. Melkersson and P. Schenzel generalized Brodmann's result by showing that for every
 fixed $i \ge 0$, the sets\index{Melkersson and Schenzel's result}
 \[
  \Ass_A\left( \Tor_i^A(M, I^n/I^{n+1}) \right) \quad\mbox{and}\quad \Ass_A\left( \Tor_i^A(M, A/I^n) \right)
 \]
 are independent of $n$ for all sufficiently large $n$. By a similar argument, one obtains that for a given $i \ge 0$, the set
 \[
  \Ass_A\left( \Ext_A^i(M, I^n/I^{n+1}) \right)
 \]
 is independent of $n$ for all large $n$. Later, D. Katz and E. West proved the above results in a more general way
 \cite[Corollary~3.5]{KW04}; if $N$ is a finitely generated $A$-module, then for every fixed $i \ge 0$,
 the sets\index{Katz and West's result}
 \[
  \Ass_A\left( \Tor_i^A(M, N/I^n N) \right) \quad\mbox{and}\quad \Ass_A\left( \Ext_A^i(M, N/I^n N) \right)
 \]
 are stable for all sufficiently large $n$. So, in particular, for every fixed $i \ge 0$, the sets
 \[
  \bigcup_{n\ge 1}\Ass_A\left(\Tor_i^A(M,N/I^n N)\right) \quad\mbox{and}\quad \bigcup_{n\ge 1}\Ass_A\left(\Ext_A^i(M,N/I^n N)\right)
 \]
 are finite. However, for a given $i \ge 0$, the set
 \[
  \bigcup_{n \ge 1} \Ass_A\left( \Ext_A^i(A/I^n, M) \right)
 \]
 need not be finite, which follows from the fact that the set of associated prime ideals of the $i$th
 {\it local cohomology module}\index{local cohomology module}
 \[
  H_I^i(M) \cong \varinjlim_{n \in \mathbb{N}} \Ext_A^i(A/I^n, M)
 \]
 need not be finite, due to an example of A. K. Singh \cite[Section~4]{Sin00}\index{Singh's example}.
 
 Recently, T. J. Puthenpurakal \cite[Theorem~5.1]{Put13} proved that if $A$ is a local complete intersection ring and
 $\mathcal{N} = \bigoplus_{n \ge 0} N_n$ is a finitely generated graded module over the Rees ring $\mathscr{R}(I)$, then
 \[
  \bigcup_{n \ge 0} \bigcup_{i \ge 0} \Ass_A \left( \Ext_A^i(M, N_n) \right)
 \]
 is a finite set. Moreover, he proved that there exist $n_0, i_0 \ge 0$ such that
 \begin{align*}
  \Ass_A\left(\Ext_A^{2i}(M,N_n)\right) &= \Ass_A\left(\Ext_A^{2 i_0}(M,N_{n_0})\right), \\
  \Ass_A\left(\Ext_A^{2i+1}(M,N_n)\right) &= \Ass_A\left(\Ext_A^{2 i_0 + 1}(M,N_{n_0})\right)
 \end{align*}
 for all $n \ge n_0$ and $i \ge i_0$. In particular, $\mathcal{N}$ can be taken as
 \[
  \bigoplus_{n \ge 0} (I^n N) \quad\mbox{ or }\quad \bigoplus_{n \ge 0} (I^n N/I^{n+1} N).
 \]
 He showed these results by proving the finite generation of a family of Ext-modules: If $A = Q/{\bf f}$, where $Q$ is a local
 ring and ${\bf f} = f_1,\ldots,f_c$ is a $Q$-regular sequence, and if $\projdim_Q(M)$ is finite, then 
 \[
  \mathscr{E}(\mathcal{N}) := \bigoplus_{n \ge 0} \bigoplus_{i \ge 0} \Ext_A^i(M,N_n)
 \]
 is a finitely generated bigraded $\mathscr{S} = \mathscr{R}(I)[t_1,\ldots,t_c]$-module, where $t_j$ are the cohomology operators
 over $A$.
 
 In the present study, we prove that if $A$ is a local complete intersection ring, then the set of associate primes
 \[
  \bigcup_{n \ge 1} \bigcup_{i \ge 0} \Ass_A \left( \Ext_A^i(M, N/I^n N) \right)
 \]
 is finite. Moreover, there are non-negative integers $n_0$ and $i_0$ such that for all $n \ge n_0$ and $i \ge i_0$, the set
 $\Ass_A\left( \Ext_A^i(M, N/I^n N) \right)$ depends only on whether $i$ is even or odd; see
 Chapter~\ref{Chapter: Asymptotic Prime Divisors over Complete Intersection Rings}.
 
\section{Castelnuovo-Mumford Regularity of Powers of Ideals}\label{Survey on regularity}

 Let $I$ be a homogeneous ideal of a polynomial ring $S = K[X_1,\ldots,X_d]$ over a field $K$ with usual grading.
 In \cite[Proposition~1]{BEL91}, A. Bertram, L. Ein and R. Lazarsfeld have initiated the study of the Castelnuovo-Mumford regularity
 of $I^n$ as a function of $n$ by proving that if $I$ is the defining ideal of a smooth complex projective variety, then
 $\reg(I^n)$ is bounded by a linear function of $n$\index{Bertram, Ein and Lazarsfeld's result}.

 Thereafter, A. V. Geramita, A. Gimigliano and Y. Pitteloud \cite[Theorem~1.1]{GGP95} and K. A. Chandler \cite[Theorem~1]{Cha97}
 proved that\index{Geramita, Gimigliano and Pitteloud}\index{Chandler's result}
 \[
  \mbox{if $~\dim(S/I) \le 1$, then $~\reg(I^n) \le n \cdot \reg(I)~$ for all $n \ge 1$.}
 \]
 This result does not hold true for higher dimension. A first counter example was given by Terai in characteristic different from $2$.
 Later, B. Sturmfels \cite[Section~1]{Stu00}\index{Sturmfels's example} exhibited a monomial ideal $J$ with $8$ generators for which
 \[
  \reg(J^2) > 2 \reg(J)
 \]
 in any characteristic. However, in arbitrary dimension, I. Swanson (\cite[Theorem~3.6]{Swa97}) first proved that for a homogeneous
 ideal $I$, there exists an integer $k$ such that\index{Swanson's result}
 \[
  \reg(I^n) \le k n ~\mbox{ for all } n \ge 1.
 \]
 
 Later, S. D. Cutkosky, J. Herzog and N. V. Trung \cite[Theorem~1.1]{CHT99} and V. Kodiyalam \cite[Theorem~5]{Kod00}
 independently\index{Cutkosky, Herzog and Trung}\index{Kodiyalam's result} proved that $\reg(I^n)$ can be expressed as a linear function of $n$ for all sufficiently large $n$. Recently,
 in \cite[Theorem~3.2]{TW05}, N. V. Trung and H.-J. Wang proved this result in a more general way. Let $A$ be a standard
 graded ring, and let $M$ be a finitely generated graded $A$-module. Then they showed that for a homogeneous ideal $I$ of $A$, there exists
 an integer $e \ge \varepsilon(M)$ such that\index{Trung and Wang's result}
 \[
  \reg(I^n M) = \rho_M(I) \cdot n + e \quad\mbox{for all } n \gg 1,
 \]
 where the invariants $\rho_M(I)$ and $\varepsilon(M)$ are defined as follows. The number $\varepsilon(M)$ denotes the smallest degree
 of the (non-zero) homogeneous elements of $M$. A homogeneous ideal $J \subseteq I$ is said to be an $M$-reduction of $I$ if
 $I^{n+1} M = J I^n M$ for some $n \ge 0$\index{reduction ideal with respect to a module}. Define
 \[
	  d(J) := \min\{ d : J \mbox{ can be generated by homogeneous elements of degree } \le d \}.
 \]
 The invariant $\rho_M(I)$ is defined to be the number
 \[
	  \rho_M(I) := \min\{ d(J) : J \mbox{ is an $M$-reduction of } I \}.
 \]
 
 In the present study, we deal with several ideals instead of just considering one ideal. Let $A = A_0[x_1,\ldots,x_d]$ be a standard graded algebra over an Artinian local ring $A_0$. Let $I_1,\ldots,I_t$ be homogeneous ideals of $A$, and $M$ be a finitely generated graded $A$-module. In Chapter~\ref{Chapter: Asymptotic linear bounds of Castelnuovo-Mumford regularity}, we prove that there
 exist two integers $k$ and $k'$ such that
 \[
	  \reg(I_1^{n_1}\cdots I_t^{n_t} M) \le (n_1 + \cdots + n_t) k + k' \quad \mbox{for all } n_1,\ldots,n_t \in \mathbb{N}.
 \]
 
\section{Characterizations of Regular Rings via Syzygy Modules}\label{Survey on characterizations of rings via syzygy modules}
 
 Throughout this section, let $A$ denote a local ring with maximal ideal $\mathfrak{m}$ and residue field $k$.
 Let $\Omega_n^A(k)$ be the $n$th syzygy module of $k$. In \cite[Corollary~1.3]{Dut89}, S. P. Dutta gave the following
 characterization of regular local rings. 
 
\begin{theorem}[Dutta]\label{theorem: RoL: Dutta}\index{Dutta's result}
 $A$ is regular if and only if $\Omega_n^A(k)$ has a non-zero free direct summand for some integer $n \ge 0$.
\end{theorem}

 Later, R. Takahashi generalized Dutta's result by giving a characterization of regular local rings in terms of the existence of a
 semidualizing direct summand of some syzygy module of the residue field. Note that $A$ itself is a semidualizing $A$-module; see
 Definition~\ref{definition: semidualizing module}. So the following theorem generalizes the above result of Dutta.

\begin{theorem}{\rm \cite[Theorem~4.3]{Tak06}}\index{Takahashi's result}
 $A$ is regular if and only if $\Omega_n^A(k)$ has a semidualizing direct summand for some integer $n \ge 0$.
\end{theorem}

 If $A$ is a Cohen-Macaulay local ring with canonical module $\omega$, then $\omega$ is a semidualizing $A$-module.
 Therefore, as an application of the above theorem, Takahashi obtained the following:

\begin{corollary}{\rm \cite[Corollary~4.4]{Tak06}}
 Let $A$ be a Cohen-Macaulay local ring with canonical module $\omega$. Then $A$ is regular if and only if $\Omega_n^A(k)$ has a
 direct summand isomorphic to $\omega$ for some integer $n \ge 0$.
\end{corollary}
 
 Kaplansky conjectured\index{Kaplansky's conjecture} that if some power of the maximal ideal of $A$ is non-zero and of finite
 projective dimension, then $A$ is regular. In \cite[Theorem~1.1]{LV68}, G. Levin and W. V. Vasconcelos proved this conjecture.
 In fact, their result is even stronger:

\begin{theorem}[Levin and Vasconcelos]\label{theorem: Levin and Vasconcelos}\index{Levin and Vasconcelos's Theorem}
 If $M$ is a finitely generated $A$-module such that $\mathfrak{m} M$ is non-zero and of finite projective dimension {\rm (}or of
 finite injective dimension{\rm )}, then $A$ is regular.
\end{theorem}
 
 Later, A. Martsinkovsky generalized Dutta's result in the following direction. He also showed that the
 above result of Levin and Vasconcelos is a special case of the following theorem. We denote a finite
 collection of non-negative integers by $\Lambda$.
 
\begin{theorem}{\rm \cite[Proposition~7]{Mar96}}\label{theorem: RoL: Martsinkovsky}\index{Martsinkovsky's result}
 Let
 \[
  f : \bigoplus_{n \in \Lambda} \left( \Omega_n^A(k) \right)^{j_n} \longrightarrow L \quad \mbox{$(j_n \ge 1$ for each $n \in \Lambda)$}
 \]
 be a surjective $A$-module homomorphism, where $L$ is non-zero and of finite projective dimension. Then $A$ is regular.
\end{theorem}
 
 Thereafter, L. L. Avramov proved a much more stronger result than the above one.
 
\begin{theorem}{\rm \cite[Corollary~9]{Avr96}}\label{theorem: Avramov, extremal complexity}\index{Avramov's result on syzygy}
 Each non-zero homomorphic image $L$ of a finite direct sum of syzygy modules of $k$ has maximal complexity, i.e.,
 $\cx_A(L) = \cx_A(k)$.
\end{theorem}

 Recall that the complexity of a finitely generated $A$-module $M$ is defined to be the number\index{complexity of a module}
 \[
  \cx_A(M) := \inf\left\{ b \in \mathbb{N} ~\middle|~ 
  \limsup_{n \to \infty} \left( \dfrac{\rank_k\left(\Ext_A^n(M,k)\right)}{n^{b-1}} \right) < \infty \right\}.
 \]
 It can be noted that $\cx_A(M) \ge 0$ with equality if and only if $\projdim_A(M)$ is finite. Therefore one obtains
 Theorem~\ref{theorem: RoL: Martsinkovsky} as a consequence of Theorem~\ref{theorem: Avramov, extremal complexity}.
 
 In Chapter~\ref{Chapter: Characterizations of Regular Local Rings via Syzygy Modules}, we prove a few variations of the above results.
 Suppose $L$ is a non-zero homomorphic image of a finite direct sum of syzygy modules of $k$. We prove that $A$ is regular if $L$ is
 either `a semidualizing module' or `a maximal Cohen-Macaulay module of finite injective dimension'.
 We also obtain one new characterization of regular local rings. We show that $A$ is regular if and only if some syzygy module of
 $k$ has a non-zero direct summand of finite injective dimension. 
 
\newpage
\thispagestyle{empty}
\cleardoublepage

\chapter{Asymptotic Prime Divisors over Complete Intersection Rings}\label{Chapter: Asymptotic Prime Divisors over Complete Intersection Rings}
                                                    
 Assume $A$ is either a local complete intersection ring or a geometric locally complete intersection ring. Throughout this chapter, let $M$ and $N$ be finitely generated $A$-modules, and let $I$ be an ideal of $A$. The main goal of this chapter is to show that the set
 \[
  \bigcup_{n \ge 1} \bigcup_{i \ge 0} \Ass_A\left( \Ext_A^i(M, N/I^n N) \right) \quad \mbox{is finite}.
 \]
 Moreover, we analyze the asymptotic behaviour of the sets
 \begin{center}
  $\Ass_A\left( \Ext_A^i(M, N/I^n N) \right)$\quad if $n$ and $i$ both tend to $\infty$.
 \end{center}
 We prove that there are non-negative integers $n_0$ and $i_0$ such that for all $n \ge n_0$ and $i \ge i_0$, we have that
 \begin{align*}
  \Ass_A\left(\Ext_A^{2i}(M,N/I^nN)\right) &= \Ass_A\left(\Ext_A^{2 i_0}(M,N/I^{n_0}N)\right), \\
  \Ass_A\left(\Ext_A^{2i+1}(M,N/I^nN)\right) &= \Ass_A\left(\Ext_A^{2 i_0 + 1}(M,N/I^{n_0}N)\right).
 \end{align*}
 
 Here we describe in brief the contents of this chapter. In Section~\ref{Module structure}, we give some graded module structures
 which we use in order to prove our main results of this chapter. The finiteness results on asymptotic prime divisors are proved
 in Section~\ref{Asymptotic associated primes: Finiteness}; while the stability results are shown in
 Section~\ref{Asymptotic associated primes: Stability}. Finally, in Section~\ref{Asymptotic associated primes: The geometric case},
 we prove the analogous results on associate primes for complete intersection rings which arise in algebraic geometry.
 
 We refer the reader to \cite[\S6]{Mat86} for all the basic results on associate primes which we use in this chapter.
 
\section{Module Structures on Ext}\label{Module structure}

 In this section, we give the graded module structures which we are going to use in order to prove our main results.
 
 Let $Q$ be a ring, and let ${\bf f} = f_1,\ldots,f_c$ be a $Q$-regular sequence. Set $A := Q/({\bf f})$.
 Let $M$ and $D$ be finitely generated $A$-modules.
 
 \begin{para}[\textbf{Eisenbud Operators and Total Ext-module}]\label{para:module structure 1}
 Let
 \[
  \mathbb{F} : \quad \cdots \rightarrow F_n \rightarrow \cdots \rightarrow F_1\rightarrow F_0 \rightarrow 0
 \]
 be a projective resolution of $M$ by finitely generated free $A$-modules. Let
 \[
  t'_j : \mathbb{F}(+2) \longrightarrow \mathbb{F}, \quad 1 \le j \le c
 \]
 be the {\it Eisenbud operators}\index{Eisenbud operators} defined by ${\bf f} = f_1,\ldots,f_c$ (see \cite[Section~1]{Eis80}).
 In view of \cite[Corollary~1.4]{Eis80}, the chain maps $t'_j$ are determined uniquely up to
 homotopy\index{homotopy equivalence of chain maps}. In particular, they induce well-defined maps
 \[
  t_j : \Ext_A^i(M,D) \longrightarrow \Ext_A^{i+2}(M,D),
 \]
 (for all $i$ and $1 \le j \le c$), on the cohomology of $\Hom_A(\mathbb{F},D)$. In \cite[Corollary~1.5]{Eis80}, it is shown that
 the chain maps $t'_j$ ($j = 1,\ldots,c$) commute up to homotopy. Thus\index{graded module structure on Ext}
 \[
  \Ext_A^{\star}(M,D) := \bigoplus_{i \ge 0} \Ext_A^i(M,D)
 \]
 turns into a graded $\mathscr{T} := A[t_1,\ldots,t_c]$-module,\index{polynomial rings in cohomology (Eisenbud) operators}
 where $\mathscr{T}$ is the graded polynomial ring  over $A$ in the {\it cohomology operators}\index{cohomology operators} $t_j$
 defined by ${\bf f}$ with $\deg(t_j) = 2$ for all $1 \le j \le c$. We call $\Ext_A^{\star}(M,D)$ the {\it total Ext-module}
 \index{total Ext-module} of $M$ and $D$. These structures depend only on ${\bf f}$, are natural in both module arguments and
 commute with the connecting maps induced by short exact sequences.
 \end{para}
 
 \begin{para}[\textbf{Gulliksen's Finiteness Theorem}]\label{para:Gulliksen}
  T. H. Gulliksen\index{Gulliksen's Finiteness Theorem} \cite[Theorem~3.1]{Gul74} proved that if either $\projdim_Q(M)$ is finite or
  $\injdim_Q(D)$ is finite, then $\Ext_A^{\star}(M,D)$ is a finitely generated
  graded $\mathscr{T} = A[t_1,\ldots,t_c]$-module; see also \cite[Theorem~(2.1)]{Avr89}.
 \end{para}
 
 \begin{para}[\textbf{Bigraded Module Structure on Ext}]\label{para:module structure 2}
 Let $I$ be an ideal of $A$. Let $\mathscr{R}(I)$ be the Rees ring\index{Rees ring associated to an ideal}
 $\bigoplus_{n \ge 0} I^n X^n$ associated to $I$. We consider $\mathscr{R}(I)$ as a subring of the polynomial ring $A[X]$. Let
 $\mathcal{N} = \bigoplus_{n \ge 0} N_n$ be a graded $\mathscr{R}(I)$-module. Let $u\in\mathscr{R}(I)$ be a homogeneous element of
 degree $s$. Consider the $A$-linear maps given by multiplication with $u$:
 \[
  N_n \stackrel{u\cdot}{\longrightarrow} N_{n+s} \quad \mbox{for all } n.
 \]
 By applying $\Hom_A(\mathbb{F},-)$ on the above maps and using the naturality of the Eisenbud operators $t'_j$, we have the
 following commutative diagram of cochain complexes:
 \[
  \xymatrixrowsep{10mm} \xymatrixcolsep{14mm}
   \xymatrix{
   \Hom_A(\mathbb{F},N_n) \ar[d]^{u} \ar[r]^{t'_j} &\Hom_A(\mathbb{F}(+2),N_n) \ar[d]^{u} \\
   \Hom_A(\mathbb{F},N_{n+s}) \ar[r]^{t'_j} &\Hom_A(\mathbb{F}(+2),N_{n+s}).
   }
 \]
 Now, taking cohomology, we obtain the following commutative diagram of $A$-modules:
 \[
  \xymatrixrowsep{10mm} \xymatrixcolsep{14mm}
   \xymatrix{
   \Ext_A^i(M,N_n) \ar[d]^{u} \ar[r]^{t_j} & \Ext_A^{i+2}(M,N_n) \ar[d]^{u} \\
   \Ext_A^i(M,N_{n+s}) \ar[r]^{t_j} & \Ext_A^{i+2}(M,N_{n+s})
   }
 \]
 for all $n,i$ and $1 \le j \le c$. Thus\index{bigraded module structure on Ext}
 \[
   \mathscr{E}(\mathcal{N}) := \bigoplus_{n \ge 0} \bigoplus_{i \ge 0}\Ext_A^i(M,N_n)
 \]
 turns into a bigraded $\mathscr{S} := \mathscr{R}(I)[t_1,\ldots,t_c]$-module, where we set $\deg(t_j) = (0,2)$ for all
 $1 \le j \le c$ and $\deg(u X^s) = (s,0)$ for all $u \in I^s$, $s \ge 0$.\index{polynomial rings in cohomology (Eisenbud) operators}
\end{para}
 
\begin{para}[\textbf{Bigraded Module Structure on Ext}]\label{para:module structure 3}
 Suppose $N$ is a finitely generated $A$-module. Set $\mathcal{L} := \bigoplus_{n \ge 0}(N/I^{n+1}N)$. Note that
 $\mathscr{R}(I,N) = \bigoplus_{n \ge 0} I^n N$ and $N[X] = N \otimes_A A[X]$ are graded modules over $\mathscr{R}(I)$
 and $A[X]$ respectively. Since $\mathscr{R}(I)$ is a graded subring of $A[X]$, we set that $N[X]$ is a graded
 $\mathscr{R}(I)$-module. Therefore $\mathcal{L}$ is a graded $\mathscr{R}(I)$-module, where the graded structure is
 induced by the sequence
 \[
  0 \longrightarrow \mathscr{R}(I,N) \longrightarrow N[X] \longrightarrow \mathcal{L}(-1) \longrightarrow 0.
 \]
 Therefore, by the observations made in Section~\ref{para:module structure 2}, we have that
 \[
  \mathscr{E}(\mathcal{L}) = \bigoplus_{n \ge 0} \bigoplus_{i \ge 0} \Ext_A^i(M,N/I^{n+1}N)
 \]
 is a bigraded module over $\mathscr{S}=\mathscr{R}(I)[t_1,\ldots,t_c]$.
\end{para}
 
 One of the main ingredients we use in this chapter is the following finiteness result, due to T. J. Puthenpurakal
 \cite[Theorem~1.1]{Put13}.
 
\begin{theorem}[Puthenpurakal]\label{theorem:finitely generated}\index{Puthenpurakal's Finiteness Theorem}
 Let $Q$ be a ring of finite Krull dimension, and let ${\bf f} = f_1,\ldots,f_c$ be a $Q$-regular sequence. Set $A :=
 Q/({\bf f})$. Let $M$ be a finitely generated $A$-module, where $\projdim_Q(M)$ is finite. Let $I$ be an ideal of $A$, and let
 $\mathcal{N} = \bigoplus_{n \ge 0} N_n$ be a finitely generated graded $\mathscr{R}(I)$-module. Then 
 \[
	  \mathscr{E}(\mathcal{N}) := \bigoplus_{n \ge 0} \bigoplus_{i \ge 0} \Ext_A^i(M,N_n)
 \]
 is a finitely generated bigraded $\mathscr{S}=\mathscr{R}(I)[t_1,\ldots,t_c]$-module.
\end{theorem}

We first prove the main results of this chapter for a ring $A$ which is of the form $Q/({\bf f})$, where $Q$ is a regular local ring
and ${\bf f} = f_1,\ldots,f_c$ is a $Q$-regular sequence. Then we deduce the main results for a local complete intersection ring with
the help of the following well-known lemma:

\begin{lemma}\label{lemma:associated}\index{completion and associate primes}
 Let $(A,\mathfrak{m})$ be a local ring, and let $\widehat{A}$ be the $\mathfrak{m}$-adic completion of $A$.
 Let $D$ be a finitely generated $A$-module. Then
 \[
  \Ass_A(D) = \left\{ \mathfrak{q} \cap A : \mathfrak{q} \in \Ass_{\widehat{A}}\left(D \otimes_A \widehat{A}\right) \right\}.
 \]
\end{lemma}

\section{Asymptotic Associate Primes: Finiteness}\label{Asymptotic associated primes: Finiteness}

 In this section, we prove the announced finiteness result for the set of associated prime ideals of the family of Ext-modules
 $\Ext_A^i(M, N/I^n N)$, $(n, i \ge 0)$, where $M$ and $N$ are finitely generated modules over a local complete intersection ring $A$,
 and $I \subseteq A$ is an ideal; see Corollary~\ref{corollary:finiteness}.
 
\begin{theorem}\label{theorem:Q mod f finiteness}\index{finiteness of associate primes}
 Let $Q$ be a ring of finite Krull dimension, and let ${\bf f} = f_1,\ldots,f_c$ be a $Q$-regular sequence. Set $A :=
 Q/({\bf f})$. Let $M$ and $N$ be finitely generated $A$-modules, where $\projdim_Q(M)$ is finite, and let $I$ be an ideal of $A$. Then the set
 \[
  \bigcup_{n \ge 1} \bigcup_{i \ge 0} \Ass_A \left( \Ext_A^i(M,N/I^n N) \right) \quad \mbox{is finite}.
 \]
\end{theorem}

\begin{proof}
 For every fixed $n \ge 0$, we consider the short exact sequence of $A$-modules:
 \[ 
	   0 \longrightarrow I^nN/I^{n+1}N \longrightarrow N/I^{n+1}N \longrightarrow N/I^nN \longrightarrow 0.
 \]
 Taking direct sum over $n \ge 0$ and setting
 \[
	  \mathcal{L} := \bigoplus_{n \ge 0}(N/I^{n+1}N),
 \]
 we obtain the following short exact sequence of graded $\mathscr{R}(I)$-modules:
 \[
	   0 \longrightarrow \gr_I(N) \longrightarrow \mathcal{L} \longrightarrow \mathcal{L}(-1) \longrightarrow 0,
 \]
 which induces an exact sequence of graded $\mathscr{R}(I)$-modules for each $i \ge 0$:
 \[ 
	   \Ext_A^i(M,\gr_I(N)) \longrightarrow \Ext_A^i(M,\mathcal{L}) \longrightarrow \Ext_A^i(M,\mathcal{L}(-1)).
 \]
 Taking direct sum over $i \ge 0$ and using the naturality of the cohomology operators $t_j$, we get the following exact sequence of
 bigraded $\mathscr{S} = \mathscr{R}(I)[t_1,\ldots,t_c]$-modules:
 \[
	  \bigoplus_{n, i \ge 0} \Ext_A^i\left(M,\dfrac{I^nN}{I^{n+1}N}\right) \stackrel{\Phi}{\longrightarrow} \bigoplus_{n, i \ge 0} V_{(n,i)} \stackrel{\Psi}{\longrightarrow}  \bigoplus_{n, i \ge 0} V_{(n-1,i)},
 \]
 where
 \[
	  V_{(n,i)} := \Ext_A^i(M,N/I^{n+1}N)
 \]
 for each $n \ge -1$ and $i \ge 0$. Now we set
 \[
   U = \bigoplus_{n, i \ge 0} U_{(n,i)} := \Image(\Phi).
 \]
 Then, for each $n, i \ge 0$, considering the exact sequence of $A$-modules:
 \[
	  0 \rightarrow U_{(n,i)} \rightarrow V_{(n,i)} \rightarrow V_{(n-1,i)},
 \]
 we have 
 \begin{align*}
	  \Ass_A\left( V_{(n,i)} \right) & \subseteq \Ass_A\left( U_{(n,i)} \right) \cup \Ass_A\left( V_{(n-1,i)} \right) \\
				  &\subseteq \Ass_A\left( U_{(n,i)} \right) \cup \Ass_A\left( U_{(n-1,i)} \right)
					\cup \Ass_A\left( V_{(n-2,i)} \right) \\
				  & ~~\vdots \\
				  &\subseteq \bigcup_{0\le j\le n}\Ass_A\left( U_{(j,i)} \right) \quad
					\mbox{[as $\Ass_A\left( V_{(-1,i)} \right) = \phi$ for each $i \ge 0$]}.
 \end{align*}
 Taking union over $n, i \ge 0$, we obtain that
 \begin{equation}\label{VU containment}
  \bigcup_{n, i \ge 0}\Ass_A\left( V_{(n,i)} \right) \subseteq \bigcup_{n, i \ge 0}\Ass_A\left( U_{(n,i)} \right).
 \end{equation}
 
 Since $\gr_I(N)$ is a finitely generated graded $\mathscr{R}(I)$-module, by Theorem~\ref{theorem:finitely generated},
 \[ 
  \bigoplus_{n, i \ge 0} \Ext_A^i\left(M, \dfrac{I^n N}{I^{n+1} N}\right)
 \]
 is a finitely generated bigraded $\mathscr{S}$-module, and hence $U$ is a finitely generated bigraded $\mathscr{S}$-module.
 Therefore, in view of \cite[Lemma~3.2]{Wes04}, we obtain that
 \begin{equation}\label{U finite}
  \bigcup_{n, i \ge 0}\Ass_A\left( U_{(n,i)} \right)\quad\mbox{ is a finite set.}
 \end{equation}
 Now the result follows from \eqref{VU containment} and \eqref{U finite}.
\end{proof}

As an immediate corollary, we obtain the following desired result.

\begin{corollary}\label{corollary:finiteness}\index{finiteness of associate primes}
 Let $(A,\mathfrak{m})$ be a local complete intersection ring. Let $M$ and $N$ be finitely generated $A$-modules, and let
 $I$ be an ideal of $A$. Then the set
 \[
  \bigcup_{n \ge 1} \bigcup_{i \ge 0}\Ass_A\left(\Ext_A^i(M,N/I^n N)\right) \quad \mbox{is finite}.
 \]
\end{corollary}

\begin{proof}
 Since $A$ is a local complete intersection ring, $\widehat{A} = Q/({\bf f})$, where $Q$ is a regular local ring
 and ${\bf f} = f_1,\ldots,f_c$ is a $Q$-regular sequence. Since $Q$ is a regular local ring, $\projdim_Q(M)$ is finite.
 Then, by applying Theorem~\ref{theorem:Q mod f finiteness} for the ring $\widehat{A}$, we have that
 \[
   \bigcup_{n, i \ge 0} \Ass_{\widehat{A}}\left(\Ext_A^i(M,N/I^nN)\otimes_A\widehat{A}\right) = 
   \bigcup_{n, i \ge 0} \Ass_{\widehat{A}}
   \left(\Ext_{\widehat{A}}^i\left(\widehat{M},\widehat{N}/{(I\widehat{A})}^n\widehat{N}\right)\right)
 \]
 is a finite set, and hence the result follows from Lemma~\ref{lemma:associated}. 
\end{proof}

\section{Asymptotic Associate Primes: Stability}\label{Asymptotic associated primes: Stability}
 
 In the present section, we analyze the asymptotic behaviour of the sets of associated prime ideals of Ext-modules
 $\Ext_A^i(M, N/I^n N)$, $(n, i \ge 0)$, where $M$ and $N$ are finitely generated modules over a local complete intersection
 ring $A$, and $I \subseteq A$ is an ideal (see Corollary~\ref{corollary:stability}). We first prove the following theorem:
 
\begin{theorem}\label{theorem:Q mod f stability}\index{stability of associate primes}
 Let $Q$ be a ring of finite Krull dimension, and let ${\bf f} = f_1,\ldots,f_c$ be a $Q$-regular sequence. Set $A :=
 Q/({\bf f})$. Let $M$ and $N$ be finitely generated $A$-modules, where $\projdim_Q(M)$ is finite, and let $I$ be an ideal of $A$.
 Then there exist two non-negative integers $n_0, i_0$ such that for all $n \ge n_0$ and $i \ge i_0$, we have that
 \begin{align*}
  \Ass_A\left(\Ext_A^{2i}(M,N/I^nN)\right) &= \Ass_A\left(\Ext_A^{2 i_0}(M,N/I^{n_0}N)\right), \\
  \Ass_A\left(\Ext_A^{2i+1}(M,N/I^nN)\right) &= \Ass_A\left(\Ext_A^{2 i_0 + 1}(M,N/I^{n_0}N)\right).
 \end{align*}
\end{theorem}

 We give the following example which shows that two sets of stable values of associate primes can occur.
 
 \begin{example}\label{example of two sets of stable values of associate primes}\index{example on stability of associate primes}
  Let $Q = k[[u,x]]$ be a ring of formal power series in two indeterminates over a field $k$. We set $A := Q/(ux)$,
  $M = N := Q/(u)$ and $I = 0$. Clearly, $A$ is a local complete intersection ring, and $M$, $N$ are $A$-modules.
  Then, for each $i \ge 1$, we have that
  \[
   \Ext_A^{2i-1}(M,N) = 0 \quad\mbox{and}\quad \Ext_A^{2i}(M,N) \cong k;
  \]
  see \cite[Example~4.3]{AB00}. So, in this example, we see that
  \begin{align*}
   & \Ass_A\left(\Ext_A^{2i-1}(M,N/I^n N)\right) = \phi \quad\mbox{and} \\
   & \Ass_A\left(\Ext_A^{2i}(M,N/I^n N)\right) = \Ass_A(k) ~(\neq \phi)
  \end{align*}
  for all positive integers $n$ and $i$.
 \end{example}
 
 Now we prove Theorem~\ref{theorem:Q mod f stability}. To prove this, we assume the following lemma which we prove at the end of
 this section.
 
\begin{lemma}\label{lemma:well behaved polynomial}\index{Hilbert function}
 Let $(Q,\mathfrak{n})$ be a local ring with residue field $k$, and let ${\bf f} = f_1,\ldots,f_c$ be a $Q$-regular
 sequence. Set $A := Q/({\bf f})$. Let $M$ and $N$ be finitely generated $A$-modules, where $\projdim_Q(M)$ is finite, and let $I$ be
 an ideal of $A$. Then, for every fixed $l = 0, 1$, we have that
 \begin{align*}
  \mbox{either}\quad & \Hom_A\left(k,\Ext_A^{2i + l}(M,N/I^nN)\right) \neq 0 \quad \mbox{for all }~ n, i \gg 0;\\
  \mbox{or}\quad     & \Hom_A\left(k,\Ext_A^{2i + l}(M,N/I^nN)\right) = 0 \quad \mbox{for all }~ n, i \gg 0.
 \end{align*}
\end{lemma}
 
\begin{proof}[Proof of Theorem~\ref{theorem:Q mod f stability}]
 By virtue of Theorem~\ref{theorem:Q mod f finiteness}, we may assume that
 \[
  \bigcup_{n \ge 1} \bigcup_{i \ge 0}\Ass_A\left(\Ext_A^i(M,N/I^n N)\right) = \{\mathfrak{p}_1,\mathfrak{p}_2,\ldots,\mathfrak{p}_l\}.
 \]
 Set $V_{(n,i)} := \Ext_A^i(M,N/I^n N)$ for each $n, i \ge 0$, and $V := \bigoplus_{n, i \ge 0} V_{(n,i)}$.
 
 We first prove that there exist some $n', i' \ge 0$ such that
 \begin{equation}\label{equation:even stability}
  \Ass_A\left( V_{(n,2i)} \right) = \Ass_A\left( V_{(n',2i')} \right) \quad \mbox{for all }n \ge n' \mbox{ and } i \ge i'.
 \end{equation}
 To prove the claim \eqref{equation:even stability}, it is enough to prove that for each $\mathfrak{p}_j$, where $1 \le j \le l$,
 there exist some $n_{j_0}, i_{j_0} \ge 0$ such that exactly one of the following alternatives must hold:
 \begin{align*}
  \mbox{either}\quad \mathfrak{p}_j &\in \Ass_A\left( V_{(n,2i)} \right)\quad\mbox{for all }n \ge n_{j_0} \mbox{ and } i \ge i_{j_0}; \\
  \mbox{or}\quad     \mathfrak{p}_j &\notin \Ass_A\left( V_{(n,2i)} \right)\quad\mbox{for all }n \ge n_{j_0} \mbox{ and } i \ge i_{j_0}.
 \end{align*}
 Localizing at $\mathfrak{p}_j$, and replacing $A_{\mathfrak{p}_j}$ by $A$ and $\mathfrak{p}_j A_{\mathfrak{p}_j}$ by $\mathfrak{m}$,
 it is now enough to prove that there exist some $n', i' \ge 0$ such that
 \begin{align}
  \mbox{either}\quad \mathfrak{m} &\in \Ass_A\left( V_{(n,2i)} \right) \quad \mbox{for all } n \ge n' \mbox{ and } i \ge i'; \label{equation:m in} \\
  \mbox{or}\quad     \mathfrak{m} &\notin \Ass_A\left( V_{(n,2i)} \right)\quad\mbox{for all }n \ge n' \mbox{ and } i \ge i'. \label{equation:m not in}
 \end{align}
 But, in view of Lemma~\ref{lemma:well behaved polynomial}, we get that there exist some $n', i' \ge 0$ such that
 \begin{align*}
  \mbox{either}\quad \Hom_A\left( k, V_{(n,2i)} \right) &\neq 0 \quad\mbox{for all }n \ge n' \mbox{ and } i \ge i'; \\
  \mbox{or}\quad     \Hom_A\left( k, V_{(n,2i)} \right) &= 0\quad\mbox{for all }n \ge n' \mbox{ and } i \ge i',
 \end{align*}
 which is equivalent to that either \eqref{equation:m in} is true, or \eqref{equation:m not in} is true.
 
 Applying a similar procedure as in the even case, we obtain that there exist some $n'', i'' \ge 0$ such that
 \begin{equation*}
  \Ass_A\left( V_{(n,2i+1)} \right) = \Ass_A\left( V_{(n'',2i''+1)} \right) \quad \mbox{for all }n \ge n'' \mbox{ and } i \ge i''.
 \end{equation*}
 Now $(n_0,i_0) := \max\{(n',i'),(n'',i'')\}$ satisfies the required result of the theorem.
\end{proof}

An immediate corollary of the Theorem~\ref{theorem:Q mod f stability} is the following:

\begin{corollary}\label{corollary:stability}\index{stability of associate primes}
 Let $(A,\mathfrak{m})$ be a local complete intersection ring. Let $M$ and $N$ be finitely generated $A$-modules, and let $I$ be
 an ideal of $A$. Then there exist two non-negative integers $n_0$ and $i_0$ such that for all $n \ge n_0$ and $i \ge i_0$, we have that
 \begin{align*}
  \Ass_A\left(\Ext_A^{2i}(M,N/I^nN)\right) &= \Ass_A\left(\Ext_A^{2 i_0}(M,N/I^{n_0}N)\right), \\
  \Ass_A\left(\Ext_A^{2i+1}(M,N/I^nN)\right) &= \Ass_A\left(\Ext_A^{2 i_0 + 1}(M,N/I^{n_0}N)\right).
 \end{align*}
\end{corollary}

\begin{proof}
 Assume $\widehat{A} = Q/({\bf f})$, where $Q$ is a regular local ring and ${\bf f}=f_1,\ldots,f_c$ is a $Q$-regular
 sequence. Then, by applying Theorem~\ref{theorem:Q mod f stability} for the ring $\widehat{A}$, we see that there exist $n_0, i_0
 \ge 0$ such that for all $n \ge n_0$ and $i \ge i_0$, we have that
 \begin{align*}
  \Ass_{\widehat{A}}\left(\Ext_A^{2i}(M,N/I^nN)\otimes_A\widehat{A}\right) & = 
  \Ass_{\widehat{A}}\left(\Ext_A^{2i_0}(M,N/I^{n_0}N)\otimes_A\widehat{A}\right), \\
  \Ass_{\widehat{A}}\left(\Ext_A^{2i+1}(M,N/I^nN)\otimes_A\widehat{A}\right) & = 
  \Ass_{\widehat{A}}\left(\Ext_A^{2i_0+1}(M,N/I^{n_0}N)\otimes_A\widehat{A}\right).
 \end{align*}
 The result now follows from Lemma~\ref{lemma:associated}.
\end{proof}

 We need the following result to prove Lemma~\ref{lemma:well behaved polynomial}.
 
\begin{lemma}\label{lemma:polynomial}\index{lemma on Hilbert function}\index{Hilbert function}
 Let $(Q,\mathfrak{n})$ be a local ring with residue field $k$, and let ${\bf f} = f_1,\ldots,f_c$ be a $Q$-regular
 sequence. Set $A := Q/({\bf f})$. Let $M$ and $N$ be finitely generated $A$-modules, where $\projdim_Q(M)$ is finite, and let $I$ be
 an ideal of $A$. Then
 \[
   \lambda_A\left(\Hom_A\left(k,\Ext_A^{2i}(M,N/I^nN)\right)\right) ~\mbox{ and }~
   \lambda_A\left(\Hom_A\left(k,\Ext_A^{2i+1}(M,N/I^nN)\right)\right)
 \]
 are given by polynomials in $n,i$ with rational coefficients for all sufficiently large $(n,i)$.
\end{lemma}

 Here we need to use the Hilbert-Serre Theorem for standard bigraded rings. Let us recall the definition of standard bigraded rings.
 
 \begin{definition}\label{definition: standard bigraded ring}\index{standard bigraded ring}
  A bigraded ring $R = \bigoplus_{(n,i) \in \mathbb{N}^2} R_{(n,i)}$ is said to be a {\it standard bigraded ring} if there exist
  $a_1,\ldots,a_r \in R_{(1,0)}$ and $b_1,\ldots,b_s \in R_{(0,1)}$ such that
  \[
   R = R_{(0,0)}[a_1,\ldots,a_r,b_1,\ldots,b_s].
  \]
 \end{definition}
 
 Let us now recall the Hilbert-Serre Theorem for standard bigraded rings.

\begin{theorem}[Hilbert-Serre]{\rm \cite[Theorem~2.1.7]{Rob98}}\label{theorem: Hilbert-Serre}\index{Hilbert-Serre Theorem}
 Let $R = \bigoplus_{(n,i) \in \mathbb{N}^2} R_{(n,i)}$ be a standard bigraded ring, where $R_{(0,0)}$ is an Artinian ring. Let
 $L = \bigoplus_{(n,i) \in \mathbb{N}^2} L_{(n,i)}$ be a finitely generated bigraded $R$-module. Then there is a polynomial
 $P(z,w) \in \mathbb{Q}[z,w]$ and an element $(n_0,i_0) \in \mathbb{N}^2$ such that
 \[
  \lambda_{R_{(0,0)}}\left( L_{(n,i)} \right) = P(n,i) \quad \mbox{for all }~ (n,i) \ge (n_0,i_0).
 \]
\end{theorem}

 As a corollary of this theorem, one obtains the following:
 
\begin{corollary}[Hilbert-Serre]\label{corollary: Hilbert-Serre}
 Let $R = \bigoplus_{(n,i) \in \mathbb{N}^2} R_{(n,i)}$ be a standard bigraded ring. Let
 $L = \bigoplus_{(n,i)\in\mathbb{N}^2}L_{(n,i)}$ be a finitely generated bigraded $R$-module, where
 $\lambda_{R_{(0,0)}}\left( L_{(n,i)} \right)$ is finite for every $(n,i) \in \mathbb{N}^2$. Then there is
 a polynomial $P(z,w) \in \mathbb{Q}[z,w]$ and an element $(n_0,i_0) \in \mathbb{N}^2$ such that
 \[
  \lambda_{R_{(0,0)}}\left( L_{(n,i)} \right) = P(n,i) \quad \mbox{for all }~ (n,i) \ge (n_0,i_0).
 \]
\end{corollary}

\begin{proof}
 We set $S := R/\ann_R(L)$. Since $R$ is a standard bigraded ring, it can be observed that $S$ is also a standard bigraded ring. Note
 that $\lambda_{R_{(0,0)}} \left( L_{(n,i)} \right) = \lambda_{S_{(0,0)}} \left( L_{(n,i)} \right)$ for all $(n,i) \in \mathbb{N}^2$.
 Therefore, by virtue of Theorem~\ref{theorem: Hilbert-Serre}, it is enough to show that $S_{(0,0)}$ is Artinian.
 
 Assume that $L$ is generated (as an $R$-module) by a collection $\{ m_1,\ldots,m_l \}$ of homogeneous elements of $L$.
 Let $\deg(m_j) = (a_j,b_j)$ for all $1 \le j \le l$. We now consider the following map:
 \begin{align*}
  \varphi : R & \longrightarrow \bigoplus_{j = 1}^l L(a_j,b_j)\\
            r & \longmapsto (r m_1, \ldots, r m_l) \quad \mbox{for all }~ r \in R.
 \end{align*}
 Clearly, $\varphi$ is a homogeneous $R$-module homomorphism, where $\Ker(\varphi) = \ann_R(L)$. So we have an embedding of graded
 $R$-modules:
 \[
   \xymatrixrowsep{4mm} \xymatrixcolsep{10mm}
   \xymatrix{\displaystyle S ~~~\ar@{^{(}->}[r] & \bigoplus_{j = 1}^l L(a_j,b_j),}
 \]
 which yields an embedding of $R_{(0,0)}$-modules:
 \[
   \xymatrixrowsep{4mm} \xymatrixcolsep{10mm}
   \xymatrix{\displaystyle S_{(0,0)} ~~~\ar@{^{(}->}[r] & \bigoplus_{j = 1}^l L_{(a_j,b_j)}.}
 \]
 Therefore $S_{(0,0)}$ has finite length, and hence it is Artinian, which completes the proof of the corollary.
\end{proof}

 The following result is well-known. But we give a proof for the reader's convenience.
 
\begin{theorem}\label{theorem: Hilbert-Serre, Even-Odd}
 Let $A$ be a ring and $I$ an ideal of $A$. Suppose that $\mathscr{R}(I)$ is the Rees ring of $I$.
 Let $\mathscr{S} = \mathscr{R}(I)[t_1,\ldots,t_c]$ with $\deg(t_j) = (0,2)$ for all $1 \le j \le c$ and $\deg(I^s) = (s,0)$ for all
 $s \ge 0$. Let $L = \bigoplus_{(n,i)\in\mathbb{N}^2}L_{(n,i)}$ be a finitely generated bigraded $\mathscr{S}$-module, where
 $\lambda_A\left( L_{(n,i)} \right)$ is finite for every $(n,i) \in \mathbb{N}^2$. Set
 \begin{align*}
  f_1(n,i) & := \lambda_A\left( L_{(n,2i)} \right) \quad \mbox{for all }~ (n,i) \in \mathbb{N}^2,\\
  f_2(n,i) & := \lambda_A\left( L_{(n,2i+1)} \right) \quad \mbox{for all }~ (n,i) \in \mathbb{N}^2.
 \end{align*}
 Then there are polynomials $P_1(z,w), P_2(z,w) \in \mathbb{Q}[z,w]$ such that
 \begin{align*}
  f_1(n,i) & = P_1(n,i) \quad \mbox{for all }~ n,i \gg 0,\\
  f_2(n,i) & = P_2(n,i) \quad \mbox{for all }~ n,i \gg 0.
 \end{align*}
\end{theorem}

\begin{proof}
 We set
 \[
  L^{\mbox{even}} := \bigoplus_{(n,i) \in \mathbb{N}^2} L_{(n,2i)} \quad \mbox{ and } \quad
  L^{\mbox{odd}} := \bigoplus_{(n,i) \in \mathbb{N}^2} L_{(n,2i+1)}.
 \]
 In view of $\mathscr{S} = \mathscr{R}(I)[t_1,\ldots,t_c]$, we build a new standard bigraded ring
 \[
  \mathscr{S}' := \mathscr{R}(I)[u_1,\ldots,u_c],
 \]
 where $\deg(u_j) := (0,1)$ for all $1 \le j \le c$ and $\deg(I^s) := (s,0)$ for all $s \ge 0$. We give $\mathbb{N}^2$-graded
 structures on $L^{\mbox{even}}$ and $L^{\mbox{odd}}$ by setting the $(n,i)$th components of $L^{\mbox{even}}$ and $L^{\mbox{odd}}$
 as $L_{(n,2i)}$ and $L_{(n,2i+1)}$ respectively. Now we define the action of $\mathscr{S}' := \mathscr{R}(I)[u_1,\ldots,u_c]$ on
 $L^{\mbox{even}}$ and $L^{\mbox{odd}}$ as follows: Elements of $\mathscr{R}(I)$ act on $L^{\mbox{even}}$ and $L^{\mbox{odd}}$ as
 before; while the action of $u_j$ ($1 \le j \le c$) is defined by
 \[
  u_j \cdot m := t_j \cdot m \quad \mbox{for all }~ m \in L^{\mbox{even}}  ~~\left(\mbox{resp. }L^{\mbox{odd}} \right).
 \]
 It can be noticed that for every $1 \le j \le c$, we have
 \begin{align*}
  u_j \left( L_{(n,2i)} \right) \subseteq L_{(n,2(i+1))}, & \mbox{ i.e., } u_j \left( L^{\mbox{even}}_{(n,i)} \right)
    \subseteq L^{\mbox{even}}_{(n,i+1)} \quad \mbox{for all }~ (n,i) \in \mathbb{N}^2;\\
  u_j \left( L_{(n,2i+1)} \right) \subseteq L_{(n,2(i+1)+1)}, & \mbox{ i.e., } u_j \left( L^{\mbox{odd}}_{(n,i)} \right)
    \subseteq L^{\mbox{odd}}_{(n,i+1)} \quad \mbox{for all }~ (n,i) \in \mathbb{N}^2.
 \end{align*}
 In this way, we obtain that $L^{\mbox{even}}$ and $L^{\mbox{odd}}$ are bigraded modules over the standard bigraded ring
 $\mathscr{S}' := \mathscr{R}(I)[u_1,\ldots,u_c]$. Note that
 \begin{align*}
  \lambda_A\left( L^{\mbox{even}}_{(n,i)} \right) & = \lambda_A\left( L_{(n,2i)} \right) < \infty \quad \mbox{for all }~
  (n,i) \in \mathbb{N}^2;\\
  \lambda_A\left( L^{\mbox{odd}}_{(n,i)} \right) & = \lambda_A\left( L_{(n,2i+1)} \right) < \infty \quad \mbox{for all }~
  (n,i) \in \mathbb{N}^2.
 \end{align*}
 So, in view of Corollary~\ref{corollary: Hilbert-Serre}, it is now enough to show that $L^{\mbox{even}}$ and $L^{\mbox{odd}}$ are
 finitely generated as $\mathscr{S}' := \mathscr{R}(I)[u_1,\ldots,u_c]$-modules. We only show that $L^{\mbox{even}}$ is finitely
 generated as $\mathscr{S}'$-module. $L^{\mbox{odd}}$ can be shown to be finitely generated $\mathscr{S}'$-module in a similar manner.
 
 Let $L$ be generated (as an $\mathscr{S}$-module) by a collection $\{ m_1,m_2,\ldots,m_r,m'_1,m'_2,\ldots,m'_s \}$ of homogeneous
 elements of $L$, where
 \begin{align*}
  & m_j \in L_{(n_j,2i_j)} \quad \mbox{for some }~ (n_j,i_j) \in \mathbb{N}^2, 1 \le j \le r;\\
  & m'_l \in L_{(n'_l, 2i'_l + 1)} \quad \mbox{for some }~ (n'_l,i'_l) \in \mathbb{N}^2, 1 \le l \le s.
 \end{align*}
 We claim that $L^{\mbox{even}}$ is generated by $\{ m_1,m_2,\ldots,m_r \}$ as an $\mathscr{S}'$-module. To prove this claim,
 we consider an arbitrary homogeneous element $m$ of $L^{\mbox{even}}$. Then $m \in L_{(n,2i)}$ for some $(n,i) \in \mathbb{N}^2$.
 Since $L$ is generated by $\{ m_1,\ldots,m_r,m'_1,\ldots,m'_s \}$, the element $m$ can be written as an
 $\mathscr{S}$-linear combination as follows:
 \begin{equation}\label{theorem: Hilbert-Serre, Even-Odd: equation 1}
  m = \alpha_1 m_1 + \cdots + \alpha_r m_r + \beta_1 m'_1 + \cdots + \beta_s m'_s
 \end{equation}
 for some homogeneous elements $\alpha_1,\ldots,\alpha_r,\beta_1,\ldots,\beta_s$ of $\mathscr{S} = \mathscr{R}(I)[t_1,\ldots,t_c]$,
 where $\alpha_j m_j$ and $\beta_l m'_l$ are in $L_{(n,2i)}$. Note that any
 homogeneous element $\alpha$ of $\mathscr{S}$ can be written as a finite sum of homogeneous elements (with same degree) of the
 following type:
 \[
  a\; t_1^{l_1} t_2^{l_2} \cdots t_c^{l_c} \quad \mbox{for some } a \in I^p ~(p \ge 0) \mbox{ and for some } l_1,l_2,\ldots,l_c \ge 0.
 \]
 Observe that
 \[
  \left( a\; t_1^{l_1} \cdots t_c^{l_c} \right) L_{(n',i')} \subseteq L_{(n'+p,i'+2(l_1+\cdots+l_c))} \quad \mbox{for all }
  (n',i') \in \mathbb{N}^2.
 \]
 Thus, for any homogeneous element $\alpha$ of $\mathscr{S}$, we obtain that
 \[
  \alpha L_{(n',i')} \subseteq L_{(n'+p,i'+2q)} \quad \mbox{for some }~ p, q \ge 0.
 \]
 Therefore, for every $1 \le l \le s$, since $m'_l \in L_{(n'_l, 2i'_l + 1)}$, we get that
 \[
  \beta_l m'_l \in L_{(n'_l + p_l, 2(i'_l + q_l) + 1)} \quad \mbox{for some }~ p_l, q_l \ge 0.
 \]
 But, in view of \eqref{theorem: Hilbert-Serre, Even-Odd: equation 1}, we have that $\beta_l m'_l \in L_{(n,2i)}$. Therefore
 $\beta_l m'_l$ must be zero for all $1 \le l \le s$, and hence $m = \alpha_1 m_1 + \cdots + \alpha_r m_r$. Replacing $t_j$ by
 $u_j$ in $\alpha_1,\ldots,\alpha_r$, we obtain that $m$ can be written as an $\mathscr{S}' := \mathscr{R}(I)[u_1,\ldots,u_c]$-linear
 combination of $m_1,\ldots,m_r$. Thus $L^{\mbox{even}}$ is generated by $\{ m_1,m_2,\ldots,m_r \}$ as an $\mathscr{S}'$-module,
 which completes the proof of the theorem.
\end{proof}
 
 In a similar way as above, one can prove the following well-known result for single graded case:
 
\begin{theorem}\label{theorem: Hilbert-Serre: even-odd: single grade case}
 Let $A$ be a ring. Let $\mathscr{T} = A[t_1,\ldots,t_c]$ with $\deg(t_j) = 2$ for all $1 \le j \le c$.
 Let $L = \bigoplus_{i \in \mathbb{N}} L_i$ be a finitely generated graded $\mathscr{T}$-module, where
 $\lambda_A( L_i )$ is finite for every $i \in \mathbb{N}$. Then there are polynomials $P_1(z), P_2(z) \in \mathbb{Q}[z]$ such that
 \begin{align*}
  \lambda_A\left( L_{2i} \right) & = P_1(i) \quad \mbox{for all }~ i \gg 0,\\
  \lambda_A\left( L_{2i+1} \right) & = P_2(i) \quad \mbox{for all }~ i \gg 0.
 \end{align*}
\end{theorem}

 Here we give:
 
\begin{proof}[Proof of Lemma~\ref{lemma:polynomial}]
 We set $V_{(n,i)} := \Ext_A^i(M,N/I^n N)$ for every $n, i \ge 0$, and $V := \bigoplus_{n, i \ge 0} V_{(n,i)}$. We only prove that the
 length $\lambda_A\left( \Hom_A\left( k, V_{(n,2i)} \right) \right)$ is given by a polynomial in $n,i$ with rational coefficients for
 all sufficiently large $(n,i)$. For $\lambda_A\left( \Hom_A\left( k, V_{(n,2i + 1)} \right) \right)$, the proof goes through exactly
 the same way.
 
 For every fixed $n \ge 0$, we consider the short exact sequence of $A$-modules:
 \[ 
  0 \longrightarrow I^n N \longrightarrow N \longrightarrow N/I^n N \longrightarrow 0,
 \]
 which induces an exact sequence of $A$-modules for each $n,i$:
 \[
  \Ext_A^i(M,I^n N) \longrightarrow \Ext_A^i(M,N) \longrightarrow \Ext_A^i(M,N/I^n N) \longrightarrow \Ext_A^{i+1}(M,I^n N).
 \]
 Taking direct sum over $n,i$ and using the naturality of the cohomology operators $t_j$, we obtain an exact sequence of bigraded
 $\mathscr{S} = \mathscr{R}(I)[t_1,\ldots,t_c]$-modules:
 \[
  U \longrightarrow T \longrightarrow V \longrightarrow U(0,1),
 \]
 where
 \begin{align*}
  U &= \bigoplus_{n, i \ge 0}U_{(n,i)} := \bigoplus_{n, i \ge 0}\Ext_A^i(M,I^n N), \\
  T &= \bigoplus_{n, i \ge 0}T_{(n,i)} := \bigoplus_{n, i \ge 0}\Ext_A^i(M,N), \\
  V &= \bigoplus_{n, i \ge 0}V_{(n,i)} := \bigoplus_{n, i \ge 0}\Ext_A^i(M,N/I^n N), \quad\mbox{and}
 \end{align*}
 $U(0,1)$ is same as $U$ but the grading is twisted by $(0,1)$. Setting
 \[
  X := \Image(U \to T), ~ Y := \Image(T \to V) ~\mbox{ and }~ Z := \Image\big(V \to U(0,1)\big),
 \]
 we have the following commutative diagram of exact sequences of bigraded $\mathscr{S}$-modules:
 \[
  \xymatrixrowsep{2mm} \xymatrixcolsep{5mm}
  \xymatrix{
	    U \ar[rd] \ar[rr] && T \ar[rd] \ar[rr] && V \ar[rd] \ar[rr] && U(0,1) \\
	    & X \ar[ru] \ar[rd] && Y \ar[ru] \ar[rd] && Z \ar[ru] \ar[rd]     \\
	    0 \ar[ru] && 0 \ar[ru] && 0 \ar[ru] && 0, }
 \]
 which gives the following short exact sequences of bigraded $\mathscr{S}$-modules:
 \[
  0 \rightarrow X \rightarrow T \rightarrow Y \rightarrow 0 \quad \mbox{and} \quad
  0 \rightarrow Y \rightarrow V \rightarrow Z \rightarrow 0.
 \]
 Now applying $\Hom_A(k,-)$ to these short exact sequences, we get the following exact sequences of bigraded $\mathscr{S}$-modules:
 \begin{align}
  &0 \longrightarrow \Hom_A(k,X) \longrightarrow \Hom_A(k,T) \longrightarrow \Hom_A(k,Y) \longrightarrow C \longrightarrow 0,
													    \label{equation:es 1} \\
  &0 \longrightarrow \Hom_A(k,Y) \longrightarrow \Hom_A(k,V) \longrightarrow D \longrightarrow 0,        \label{equation:es 2}
 \end{align}
 where
 \begin{align*}
  C & := \Image\left( \Hom_A(k,Y) \longrightarrow \Ext_A^1(k,X) \right) \quad\mbox{and} \\
  D & := \Image\big( \Hom_A(k,V) \longrightarrow \Hom_A(k,Z) \big).
 \end{align*}

 By virtue of Theorem~\ref{theorem:finitely generated}, we have that
 \[
  U = \bigoplus_{n \ge 0} \bigoplus_{i \ge 0} \Ext_A^i(M, I^n N)
 \]
 is a finitely generated bigraded
 $\mathscr{S}$-module, and hence $X = \Image(U \rightarrow T)$ is so. Therefore $\Hom_A(k,X)$ and $\Ext_A^1(k,X)$ are finitely
 generated bigraded $\mathscr{S}$-modules. Being an $\mathscr{S}$-submodule of $\Ext_A^1(k,X)$, the bigraded $\mathscr{S}$-module
 $C$ is also finitely generated. Since $\Hom_A(k,X)$ and $\Ext_A^1(k,X)$ are annihilated by the maximal ideal of
 $A$, $\Hom_A\left( k, X_{(n,i)} \right)$ and $C_{(n,i)}$ both are finitely generated $k$-modules, and hence they have finite length as $A$-modules
 for each $n, i \ge 0$. Therefore, by applying Theorem~\ref{theorem: Hilbert-Serre, Even-Odd} to the bigraded
 $\mathscr{S} = \mathscr{R}(I)[t_1,\ldots,t_c]$-modules $\Hom_A(k,X)$ and $C$ (where $\deg(t_j) = (0,2)$ for all $j = 1,\ldots, c$, and
 $\deg(I^s) = (s,0)$ for all $s \ge 0$), we obtain that
 \begin{align}
  (n,i)                 & \longmapsto \lambda_A\left( \Hom_A\left( k, X_{(n,2i)} \right) \right)   \label{equation:series X} \\
  \mbox{and}\quad (n,i) & \longmapsto \lambda_A\left( C_{(n,2i)} \right)             \label{equation:series C}
 \end{align}
 are given by polynomials in $n,i$ with rational coefficients for all sufficiently large $(n,i)$.
 
 For every fixed $n \ge 0$, we have that $\bigoplus_{i \ge 0} T_{(n,i)} = \bigoplus_{i \ge 0} \Ext_A^i(M,N)$
 is a finitely generated graded $A[t_1,\ldots,t_c]$-module (where $\deg(t_j) = 2$ for all $j = 1,\ldots, c$), and hence
 $\Hom_A\left(k,\bigoplus_{i\ge 0}T_{(n,i)}\right)$ is also so. By a similar argument as before, we have that
 $\lambda_A\left( \Hom_A\left( k, T_{(n,i)} \right) \right)$ is finite for every $i \ge 0$. Therefore, in view of
 Theorem~\ref{theorem: Hilbert-Serre: even-odd: single grade case}, for each $n \ge 0$, we obtain that
 \begin{align*}
  i \longmapsto \lambda_A\left( \Hom_A\left( k, T_{(n,2i)} \right) \right)
 \end{align*}
 is given by a polynomial in $i$ for all sufficiently large $i$.
 Note that these polynomials are the same polynomial for all $n \ge 0$. Therefore
 \begin{equation}\label{equation:series T}
  (n,i) \longmapsto \lambda_A\left( \Hom_A\left( k, T_{(n,2i)} \right) \right)
 \end{equation}
 is asymptotically given by a polynomial in $n,i$ (which is independent of $n$).
 
 Now considering \eqref{equation:es 1}, we have an exact sequence of $A$-modules:
 \[
  0 \longrightarrow \Hom_A\left( k, X_{(n,i)} \right) \longrightarrow \Hom_A\left( k, T_{(n,i)} \right)
  \longrightarrow \Hom_A\left( k, Y_{(n,i)} \right) \longrightarrow C_{(n,i)} \longrightarrow 0
 \]
 for each $n, i \ge 0$. So, the additivity of the length function\index{additivity of the length function} gives
 \[
  \lambda_A\left( \Hom_A\left( k, X_{(n,2i)} \right) \right) - \lambda_A\left( \Hom_A\left( k, T_{(n,2i)} \right) \right) +
  \lambda_A\left( \Hom_A\left( k, Y_{(n,2i)} \right) \right) - \lambda_A\left( C_{(n,2i)} \right) = 0
 \]
 for each $n, i \ge 0$. So, in view of \eqref{equation:series X}, \eqref{equation:series C} and \eqref{equation:series T},
 we obtain that
 \begin{equation}\label{equation:series Y}
  (n,i) \longmapsto \lambda_A\left( \Hom_A\left( k, Y_{(n,2i)} \right) \right)
 \end{equation}
 is asymptotically given by a polynomial in $n,i$.
 
 Recall that by Theorem~\ref{theorem:finitely generated}, $U$ is a finitely generated bigraded $\mathscr{S}$-module. Now observe that
 $Z$ is a bigraded submodule of $U(0,1)$. Therefore $Z$ is a finitely generated bigraded
 $\mathscr{S} = \mathscr{R}(I)[t_1,\ldots,t_c]$-module, hence $\Hom_A(k,Z)$ is so, and hence
 \[
  D = \Image\big( \Hom_A(k,V) \longrightarrow \Hom_A(k,Z) \big)
 \]
 is also so. Observe that $\lambda_A\left( D_{(n,i)} \right)$ is finite for each $n, i \ge 0$.
 Therefore, once again by applying Theorem~\ref{theorem: Hilbert-Serre, Even-Odd}, we have that
 \begin{equation}\label{equation:series D}
  (n,i) \longmapsto \lambda_A\left( D_{(n,2i)} \right)
 \end{equation}
 is asymptotically given by a polynomial in $n,i$.
 
 Now considering \eqref{equation:es 2}, we have an exact sequence of $A$-modules:
 \[
  0 \longrightarrow \Hom_A\left( k, Y_{(n,i)} \right) \longrightarrow \Hom_A\left( k, V_{(n,i)} \right) \longrightarrow D_{(n,i)} \longrightarrow 0
 \]
 for each $n, i \ge 0$, which gives
 \[\lambda_A\left( \Hom_A\left( k, Y_{(n,2i)} \right) \right) - \lambda_A\left( \Hom_A\left( k, V_{(n,2i)} \right) \right) + \lambda_A\left( D_{(n,2i)} \right) = 0.\]
 Therefore, in view of \eqref{equation:series Y} and \eqref{equation:series D}, we obtain that
 \begin{equation*}
  (n,i) \longmapsto \lambda_A\left( \Hom_A\left( k, V_{(n,2i)} \right) \right)
 \end{equation*}
 is asymptotically given by a polynomial in $n,i$ with rational coefficients, which completes the proof of the lemma.
\end{proof}

 To prove Lemma~\ref{lemma:well behaved polynomial}, we also need the following result:
 
 \begin{lemma}\label{lemma: West Proposition 5.1}
  Let $A$ be a ring and $I$ an ideal of $A$. Let $\mathscr{R}(I)$ be the Rees ring of $I$.
  Let $\mathscr{S} = \mathscr{R}(I)[t_1,\ldots,t_c]$, where $\deg(t_j) = (0,2)$ for all $j = 1,\ldots,c$ and $\deg(I^s) = (s,0)$
  for all $s \ge 0$. Let $L = \bigoplus_{(n,i) \in \mathbb{N}^2} L_{(n,i)}$ be a finitely generated bigraded $\mathscr{S}$-module.
  Then, for every fixed $l = 0, 1$, we have that
   \begin{align*}
    \mbox{either} \quad & L_{(n,2i+l)} \neq 0 ~\mbox{ for all }~ n, i \gg 0;\\
    \mbox{or} \quad     & L_{(n,2i+l)} = 0 ~\mbox{ for all }~ n, i \gg 0.
   \end{align*}
 \end{lemma}
 
 \begin{proof}
  In view of \cite[Proposition~5.1]{Wes04}, we get that there is $(n_0,i_0) \in \mathbb{N}^2$ such that
  \begin{align*}
   \Ass_A\left( L_{(n,2i)} \right) & = \Ass_A\left( L_{(n_0,2i_0)} \right) \quad \mbox{for all }~ (n,i) \ge (n_0,i_0);\\
   \Ass_A\left( L_{(n,2i+1)} \right) & = \Ass_A\left( L_{(n_0,2i_0 + 1)} \right) \quad \mbox{for all }~ (n,i) \ge (n_0,i_0).
  \end{align*}
  It is well-known that an $A$-module $M$ is zero if and only if $\Ass_A(M) = \phi$. Therefore the lemma follows from the above
  stability of the sets of associated prime ideals.
 \end{proof}

 We now give:

\begin{proof}[Proof of Lemma~\ref{lemma:well behaved polynomial}]
 We prove the lemma for $l = 0$ only. For $l = 1$, the proof goes through exactly the same way. We set
 \begin{equation}\label{lemma:well behaved polynomial: equation 1}
  f(n,i) := \lambda_A\left(\Hom_A\left(k,\Ext_A^{2i}(M,N/I^nN)\right)\right) \quad\mbox{for all }~ n, i \ge 0.
 \end{equation}
 By virtue of Lemma~\ref{lemma:polynomial}, we get that $f(n,i)$ is given by a polynomial in $n,i$ with rational
 coefficients for all sufficiently large $(n,i)$. If $f(n,i) = 0$ for all $n, i \gg 0$, then we have nothing to prove. Suppose this
 is not the case. Then we claim that
 \begin{equation}\label{well behaved claim}
  \Hom_A\left(k,\Ext_A^{2i}(M,N/I^nN)\right) \neq 0 \quad \mbox{for all }n, i \gg 0.
 \end{equation}
 
 For every fixed $n \ge 0$, we consider the following short exact sequence of $A$-modules:
 \[
  0 \longrightarrow I^n N/I^{n+1} N \longrightarrow N/I^{n+1} N \longrightarrow N/I^n N \longrightarrow 0,
 \]
 which yields an exact sequence of $A$-modules for each $n,i$:
 \begin{align*}
  &\Ext_A^i\left(M, \dfrac{I^n N}{I^{n+1} N}\right) \longrightarrow \Ext_A^i\left(M, \dfrac{N}{I^{n+1} N}\right) \longrightarrow
  \Ext_A^i\left(M, \dfrac{N}{I^n N}\right) \\
  \longrightarrow &\Ext_A^{i+1}\left(M, \dfrac{I^n N}{I^{n+1} N}\right).
 \end{align*}
 Taking direct sum over $n,i$, and using the naturality of the cohomology operators $t_j$, we obtain an exact sequence of bigraded
 $\mathscr{S} = \mathscr{R}(I)[t_1,\ldots,t_c]$-modules:
 \[
  U \longrightarrow V(1,0) \longrightarrow V \longrightarrow U(0,1),
 \]
 where
 \begin{align*}
  U &= \bigoplus_{n, i \ge 0} U_{(n,i)} := \bigoplus_{n, i \ge 0} \Ext_A^i(M, I^n N/I^{n+1} N) ~~\mbox{ and}\\
  V &= \bigoplus_{n, i \ge 0} V_{(n,i)} := \bigoplus_{n, i \ge 0} \Ext_A^i(M, N/I^n N).
 \end{align*}
 Setting
 \[
  X := \Image\big( U \rightarrow V(1,0) \big), ~ Y := \Image\big( V(1,0) \rightarrow V \big)
  ~\mbox{ and }~ Z := \Image\big( V \rightarrow U(0,1) \big),
 \]
 we have the following short exact sequences of bigraded $\mathscr{S}$-modules:
 \[
  0 \to X \to V(1,0) \to Y \to 0 \quad \mbox{and} \quad 0 \to Y \to V \to Z \to 0.
 \]
 Now applying $\Hom_A(k,-)$ to these short exact sequences, we obtain the following exact sequences of bigraded $\mathscr{S}$-modules:
 \begin{align}
  &0 \longrightarrow \Hom_A(k,X) \longrightarrow \Hom_A\big(k,V(1,0)\big) \longrightarrow \Hom_A(k,Y) \longrightarrow C\longrightarrow 0,
  \label{lemma:well behaved polynomial: equation 2}\\
  &0 \longrightarrow \Hom_A(k,Y) \longrightarrow \Hom_A(k,V) \longrightarrow D \longrightarrow 0,
  \label{lemma:well behaved polynomial: equation 3}
 \end{align}
 where
 \begin{align*}
  C & := \Image\left( \Hom_A(k,Y) \longrightarrow \Ext_A^1(k,X) \right) \quad\mbox{and}\\
  D & := \Image\big( \Hom_A(k,V) \longrightarrow \Hom_A(k,Z) \big).
 \end{align*}
 
 In view of Theorem~\ref{theorem:finitely generated}, we have that
 \[
  U = \bigoplus_{n \ge 0} \bigoplus_{i \ge 0} \Ext_A^i\left( M, \dfrac{I^n N}{I^{n+1} N} \right)
 \]
 is a finitely generated bigraded $\mathscr{S}$-module, and hence
 \[
  X = \Image\big( U \rightarrow V(1,0) \big) \quad \mbox{and} \quad Z = \Image\big( V \rightarrow U(0,1) \big)
 \]
 are so. Therefore $\Hom_A(k,X)$, $\Ext_A^1(k,X)$ and $\Hom_A(k,Z)$ are finitely generated bigraded $\mathscr{S}$-modules.
 Hence $C$ and $D$ are finitely generated bigraded $\mathscr{S} = \mathscr{R}(I)[t_1,\ldots,t_c]$-modules.
 So, in view of Lemma~\ref{lemma: West Proposition 5.1}, we obtain that
 \begin{align}
  \label{lemma:well behaved polynomial: equation 4.i}
  & \left. \begin{array}{r}
           \mbox{either} \quad \Hom_A\left(k,X_{(n,2i)}\right) \neq 0 ~\mbox{ for all }~ n, i \gg 0,\\
           \mbox{or}     \quad \Hom_A\left(k,X_{(n,2i)}\right) = 0 ~\mbox{ for all }~ n, i \gg 0; \end{array} \right\}\\
  \label{lemma:well behaved polynomial: equation 4.ii}
  & \left. \begin{array}{r}
           \mbox{either} \quad C_{(n,2i)} \neq 0 ~\mbox{ for all }~ n, i \gg 0,\\
           \mbox{or}     \quad C_{(n,2i)} = 0 ~\mbox{ for all }~ n, i \gg 0; \end{array} \right\}\\
  \label{lemma:well behaved polynomial: equation 4.iii}
  & \left. \begin{array}{r}
           \mbox{either} \quad D_{(n,2i)} \neq 0 ~\mbox{ for all }~ n, i \gg 0,\\
           \mbox{or}     \quad D_{(n,2i)} = 0 ~\mbox{ for all }~ n, i \gg 0. \end{array} \right\}
 \end{align}
 
 For every $n, i \ge 0$, the $(n,2i)$th components of \eqref{lemma:well behaved polynomial: equation 2}
 and \eqref{lemma:well behaved polynomial: equation 3} give the following exact sequences of $A$-modules:
 \begin{align}
  & 0 \longrightarrow \Hom_A\left( k, X_{(n,2i)} \right) \longrightarrow \Hom_A\left( k, V_{(n+1,2i)} \right) \longrightarrow
  \Hom_A\left( k, Y_{(n,2i)} \right) \longrightarrow C_{(n,2i)} \longrightarrow 0,
								\label{lemma:well behaved polynomial: equation ses 1} \\
  & 0 \longrightarrow \Hom_A\left( k, Y_{(n,2i)} \right) \longrightarrow \Hom_A\left( k, V_{(n,2i)} \right) \longrightarrow D_{(n,2i)}
  \longrightarrow 0.  						\label{lemma:well behaved polynomial: equation ses 2}
 \end{align}
 We now prove the claim \eqref{well behaved claim} (i.e., $\Hom_A\left( k, V_{(n,2i)} \right) \neq 0$ for all $n, i \gg 0$)
 by considering the following four cases:
 
 {\bf Case 1.} Assume that $\Hom_A\left( k, X_{(n,2i)} \right) \neq 0$ for all $n, i \gg 0$. Then,
 in view of \eqref{lemma:well behaved polynomial: equation ses 1}, we have that $\Hom_A\left( k, V_{(n,2i)} \right) \neq 0$ for all $n, i \gg 0$.
 So, in this case, we are done.
 
 {\bf Case 2.} Assume that $C_{(n,2i)} \neq 0$ for all $n, i \gg 0$. So again,
 in view of \eqref{lemma:well behaved polynomial: equation ses 1}, we get that $\Hom_A\left( k, Y_{(n,2i)} \right) \neq 0$ for all $n, i \gg 0$.
 Hence \eqref{lemma:well behaved polynomial: equation ses 2} yields that $\Hom_A\left( k, V_{(n,2i)} \right) \neq 0$ for all $n, i \gg 0$. So, in
 this case also, we are done.
 
 {\bf Case 3.} Assume that $D_{(n,2i)} \neq 0$ for all $n, i \gg 0$. In this case,
 by considering \eqref{lemma:well behaved polynomial: equation ses 2},
 we obtain that $\Hom_A\left( k, V_{(n,2i)} \right) \neq 0$ for all $n, i \gg 0$, and we are done.
 
 Now note that if none of the above three cases holds, then by virtue of \eqref{lemma:well behaved polynomial: equation 4.i},
 \eqref{lemma:well behaved polynomial: equation 4.ii} and \eqref{lemma:well behaved polynomial: equation 4.iii}, we have
 the following case:
 
 {\bf Case 4.} Assume that $\Hom_A\left( k, X_{(n,2i)} \right) = 0$ for all $n, i \gg 0$, $C_{(n,2i)} = 0$ for all $n, i \gg 0$, and
 $D_{(n,2i)} = 0$ for all $n, i \gg 0$. Then the exact sequences \eqref{lemma:well behaved polynomial: equation ses 1}
 and \eqref{lemma:well behaved polynomial: equation ses 2} yield the following isomorphisms:
 \begin{equation}\label{lemma:well behaved polynomial: equation 5}
  \Hom_A\left( k, V_{(n+1,2i)} \right) \cong \Hom_A\left( k, Y_{(n,2i)} \right) \cong \Hom_A\left( k, V_{(n,2i)} \right) \quad \mbox{for all } n, i \gg 0.
 \end{equation}
 These isomorphisms (with the setting \eqref{lemma:well behaved polynomial: equation 1}) give the following equalities:
 \begin{equation}\label{lemma:well behaved polynomial: equation 6}
  f(n+1,i) = f(n,i) \quad \mbox{for all }~ n, i \gg 0.
 \end{equation}
 We write the polynomial expression of $f(n,i)$ in the following way:
 \begin{equation}\label{lemma:well behaved polynomial: equation 7}
  f(n,i) = h_0(i) n^a + h_1(i) n^{a-1} + \cdots + h_{a-1}(i) n + h_a(i) \quad \mbox{for all }~ n, i \gg 0,
 \end{equation}
 where $h_j(i)$ ($j = 0,1,\ldots,a$) are polynomials in $i$ over $\mathbb{Q}$. Without loss of generality, we may assume that $h_0$
 is a non-zero polynomial. So $h_0$ may have only finitely many roots. Let $i' \ge 0$ be such that $h_0(i) \neq 0$ for all $i \ge i'$.
 In view of \eqref{lemma:well behaved polynomial: equation 6} and \eqref{lemma:well behaved polynomial: equation 7}, there exist some
 $n_0$ ($\ge 0$) and $i_0$ ($\ge i'$, say) such that for all $n \ge n_0$ and $i \ge i_0$, we have
 \begin{align*}
  &f(n+1,i) = f(n,i) \quad \mbox{and}\\
  &f(n,i) = h_0(i) n^a + h_1(i) n^{a-1} + \cdots + h_{a-1}(i) n + h_a(i).
 \end{align*}
 Therefore $a$ must be equal to $0$, and hence $f(n,i) = h_0(i)$ for all $n \ge n_0$ and $i \ge i_0$. Thus we obtain that
 $f(n,i) \neq 0$ for all $n \ge n_0$ and $i \ge i_0$, i.e.,
 \[
  \Hom_A\left( k, V_{(n,2i)} \right) \neq 0 \quad \mbox{for all $n \ge n_0$ and $i \ge i_0$},
 \]
 which completes the proof of the lemma.
\end{proof}

\section{Asymptotic Associate Primes: The Geometric Case}\label{Asymptotic associated primes: The geometric case}

 Let $V$ be an affine\index{affine variety} or projective variety\index{projective variety} over an algebraically closed field $K$.
 Let $A$ be the coordinate ring\index{coordinate ring of a variety} of $V$. Then $V$ is said to be a {\it locally complete
 intersection variety}\index{locally complete intersection variety} if all its local rings are complete intersection local rings.
 Thus:\index{geometric locally complete intersection}
\begin{itemize}
 \item[(i)] in the affine case, $A_\mathfrak{p}$ is a local complete intersection ring for every $\mathfrak{p}\in \Spec(A)$;
 \item[(ii)] in the projective case, $A_{(\mathfrak{p})}$ is a local complete intersection ring for every
       $\mathfrak{p}\in \Proj(A)$\index{Proj of a graded ring}. Recall that $A_{(\mathfrak{p})}$ is called the
       {\it homogeneous localization}\index{homogeneous localization} (or {\it degree zero localization})\index{degree zero localization}
       of $A$ at the homogeneous prime ideal $\mathfrak{p}$ which is defined to be the degree zero part of the graded ring $S^{-1}A$,
       where $S$ is the collection of all homogeneous elements in $A \smallsetminus \mathfrak{p}$.
\end{itemize}

In this section, we prove the results analogous to Theorems \ref{theorem:Q mod f finiteness} and
\ref{theorem:Q mod f stability} for the coordinate rings of locally complete intersection varieties.
In the affine case, we prove the following general results.

\begin{theorem}\label{theorem:affine case}\index{finiteness of associate primes}\index{stability of associate primes}
 Let $A = Q/\mathfrak{a}$, where $Q$ is a regular ring of finite Krull dimension and $\mathfrak{a} \subseteq Q$ is an ideal so that
 $\mathfrak{a}_{\mathfrak{q}} \subseteq Q_{\mathfrak{q}}$ is generated by a $Q_{\mathfrak{q}}$-regular sequence for each
 $\mathfrak{q} \in \Var(\mathfrak{a})$\index{variety of an ideal}. Let $M$ and $N$ be finitely generated $A$-modules, and let $I$
 be an ideal of $A$. Then the set
 \[
  \bigcup_{n \ge 0} \bigcup_{i \ge 0} \Ass_A\left( \Ext_A^i(M, N/I^n N) \right)
 \]
 is finite. Moreover, there exist some non-negative integers $n_0$ and $i_0$ such that for all $n \ge n_0$ and $i \ge i_0$, we have that
 \begin{align*}
  \Ass_A\left(\Ext_A^{2i}(M,N/I^nN)\right) &= \Ass_A\left(\Ext_A^{2 i_0}(M,N/I^{n_0}N)\right), \\
  \Ass_A\left(\Ext_A^{2i+1}(M,N/I^nN)\right) &= \Ass_A\left(\Ext_A^{2 i_0 + 1}(M,N/I^{n_0}N)\right).
 \end{align*}
\end{theorem}

\begin{proof}
 For each $x \in A$, we set\index{basic open set of $\Spec(A)$}
 \[
  D(x) := \{ \mathfrak{p} \in \Spec(A) : x \notin \mathfrak{p} \}.
 \]
 As in \cite[page 384, {\it Proof of Theorem~6.1}]{Put13}, we have
 \begin{equation}\label{equation:affine 0}
  \Spec(A) = D(g_1) \cup \cdots \cup D(g_m)
 \end{equation}
 for some $g_1, \ldots, g_m \in A$ such that the localization $A_{g_j}$ by $\{{g_j}^l : l \in \mathbb{N}\}$ has the form
 $Q_j/\mathfrak{a}_j$ for some regular ring $Q_j$ of finite Krull dimension and some ideal $\mathfrak{a}_j$ of $Q_j$ generated by
 a $Q_j$-regular sequence.
 
 It can be easily observed from \eqref{equation:affine 0} that for a finitely generated $A$-module $E$, we have
 \begin{equation}\label{equation:affine 1}
  \Ass_A(E) = \bigcup\left\{ \mathfrak{q}\cap A : \mathfrak{q} \in \Ass_{A_{g_j}}(E_{g_j})\mbox{ for some } j = 1,\ldots,m \right\}.
 \end{equation}
 Since localization $A_{g_j}$ is flat over $A$, we have
 \begin{equation}\label{equation:affine 2}
 \left(\Ext_A^i(M,N/I^n N)\right)_{g_j} = ~\Ext_{A_{g_j}}^i\left(M_{g_j},N_{g_j}/(IA_{g_j})^n N_{g_j}\right)
 \end{equation}
 for all $n, i \ge 0$ and $j=1,\ldots,m$. Therefore, in view of \eqref{equation:affine 1} and \eqref{equation:affine 2}, it is
 enough to prove the result for the ring $A_{g_j}=Q_j/\mathfrak{a}_j$ for each $j$. Note that $Q_j$ is a regular ring
 of finite Krull dimension, and hence $\projdim_{Q_j}(M_{g_j})$ is finite. Therefore the result now follows by
 applying the Theorems \ref{theorem:Q mod f finiteness} and \ref{theorem:Q mod f stability} to each
 $A_{g_j}=Q_j/\mathfrak{a}_j$.
\end{proof}
 
 Now we prove the analogous results to Theorem~\ref{theorem:affine case} in the projective case. Let us fix the following hypothesis:
 
\begin{hypothesis}\label{hypothesis: the projective locally CI variety}
  Let $K$ be a field not necessarily algebraically closed, and let $Q = K[X_0,X_1,\ldots,X_r]$ be a polynomial ring over $K$,
  where $\deg(X_i) = 1$ for all $i = 0,1,\dots,r$. Let $\mathfrak{a}$ be a homogeneous ideal of $Q$. Set $A := Q/\mathfrak{a}$. Suppose
  $A_{(\mathfrak{p})}$ is a complete intersection ring for every $\mathfrak{p} \in \Proj(A)$.
\end{hypothesis}
 
 Let $\mathfrak{m}$ be the unique maximal homogeneous ideal of $A$. Let $E$ be a finitely generated graded $A$-module. Note that all the associate
 primes of $E$ are homogeneous prime ideals. Define the set of {\it relevant associated prime ideals}\index{relevant associate primes}
 of $E$ as
 \[
  {}^*\Ass_A(E) := \Ass_A(E) \cap \Proj(A) = \Ass_A(E) \smallsetminus \{\mathfrak{m}\}.
 \]  
 
 In the projective case, we prove the following:
 
\begin{theorem}\label{theorem:projective case}\index{finiteness of associate primes}\index{stability of associate primes}
  With the {\rm Hypothesis~\ref{hypothesis: the projective locally CI variety}}, let $M$ and $N$ be finitely generated graded
  $A$-modules, and let $I$ be a homogeneous ideal of $A$. Then the set
  \[
    \bigcup_{n \ge 0} \bigcup_{i \ge 0} {}^*\Ass_A \left( \Ext_A^i(M, N/I^n N) \right)
  \]
  is finite. Moreover, there exist some non-negative integers $n_0$ and $i_0$ such that for all $n \ge n_0$ and $i \ge i_0$, we have that
  \begin{align*}
   {}^*\Ass_A\left(\Ext_A^{2i}(M,N/I^nN)\right) &= {}^*\Ass_A\left(\Ext_A^{2 i_0}(M,N/I^{n_0}N)\right), \\
   {}^*\Ass_A\left(\Ext_A^{2i+1}(M,N/I^nN)\right) &= {}^*\Ass_A\left(\Ext_A^{2 i_0 + 1}(M,N/I^{n_0}N)\right).
  \end{align*}
\end{theorem}
 
\begin{proof}
  For each homogeneous element $x \in A$, we set\index{basic open set of $\Spec(A)$}
  \[
   {}^*D(x) = \{\mathfrak{p} \in \Proj(A) : x \notin \mathfrak{p}\}.
  \]
  As in \cite[page 386, {\it Proof of Theorem~6.3}]{Put13}, we have
 \[
  \Proj(A) = {}^*D(g_1) \cup \cdots \cup {}^*D(g_m)
 \]
 for some homogeneous $g_1,\ldots,g_m \in A$ such that $A_{g_j} = Q_j/\mathfrak{a}_j$ for some regular ring $Q_j$ of finite Krull
 dimension and some ideal $\mathfrak{a}_j$ of $Q_j$ generated by a $Q_j$-regular sequence. Clearly, for any graded
 $A$-module $E$, we obtain
 \[
  {}^*\Ass_A(E) = \bigcup\left\{ \mathfrak{q} \cap A : \mathfrak{q} \in \Ass_{A_{g_j}}(E_{g_j}) \mbox{ for some } j=1,\ldots,m \right\}.
 \]
 Similarly, as in the proof of Theorem~\ref{theorem:affine case}, the result now follows by applying
 Theorems~\ref{theorem:Q mod f finiteness} and \ref{theorem:Q mod f stability} to each $A_{g_j} = Q_j/\mathfrak{a}_j$.
\end{proof}
 
\newpage
\thispagestyle{empty}
\cleardoublepage

\chapter{Asymptotic Stability of Complexities over Complete Intersection Rings}\label{Chapter: Asymptotic Stability of Complexities}

 Throughout this chapter, let $(A,\mathfrak{m},k)$ be a local complete intersection ring of codimension $c$. Let $M$ and $N$ be finitely generated  $A$-modules. Recall that the {\it complexity of the pair}\index{complexity of a pair of modules} $(M,N)$ is
 defined to be the number
 \[
  \cx_A(M,N) = \inf\left\{b\in\mathbb{N} ~\middle|~ \limsup_{n\rightarrow\infty}\dfrac{\mu(\Ext_A^n(M,N))}{n^{b-1}}<\infty \right\},
 \]
 where $\mu(D)$ denotes the minimal number of generators of an  $A$-module $D$. Let $I$ be an ideal of $A$.
 In this chapter, we show that
 \begin{center}
  (C1) \hfill $\cx_A(M, N/I^j N)$\quad is independent of $j$ for all sufficiently large $j$; \hfill \;
 \end{center}
 see Theorem~\ref{theorem:Stability of complexity}. To prove this result, we take advantage of the notion of support variety which was
 introduced by L. L. Avramov and R.-O. Buchweitz in \cite[2.1]{AB00}.
 
 The organization of this chapter is as follows. In Section~\ref{Section: Support Varieties}, we recall how the notions of
 complexity and support variety are related. We also give some preliminaries which we need to prove our result on complexity.
 Finally, in Section~\ref{Section: Stability of Complexities}, we prove Theorem~\ref{theorem:Stability of complexity}.

\section{Support Varieties}\label{Section: Support Varieties}
 
 To prove Theorem~\ref{theorem:Stability of complexity}, we first reduce to the case when our local complete intersection ring
 $A$ is complete and its residue field $k$ is algebraically closed.
 
\begin{para}[\textbf{Complexity through Flat Local Extension}]\label{para: flat local extension}
 \index{complexity through flat local extension} Suppose $(A',\mathfrak{m}')$ is a flat local extension of $(A,\mathfrak{m})$ such that
 $\mathfrak{m}' = \mathfrak{m} A'$. It is shown in \cite[Theorem~7.4.3]{Avr98} that $A'$ is also a local complete intersection ring.
 For an $A$-module $E$, we set $E' = E \otimes_A A'$. Note that $I' \cong I A'$. So we may consider $I'$ as an ideal of $A'$.
 It can be easily checked that
 \[
  \cx_{A'}\big(M', N'/{(I')}^j N'\big) = \cx_A(M, N/I^j N) \quad\mbox{for all } j \ge 0.
 \]
\end{para}

\begin{para}[\textbf{Reduction to the Case when $A$ is Complete with Algebraically Closed Residue Field}]
                                                                       \label{para: reduction to A s.t. k is alg closed}
 By \cite[Chapitre 9, appendice, corollaire du th\'{e}or\'{e}me 1, p. IX.41]{Bou83}, there exists a flat local extension
 $A \subseteq \widetilde{A}$ such that $\widetilde{\mathfrak{m}} = \mathfrak{m} \widetilde{A}$ is the maximal ideal of $\widetilde{A}$
 and the residue field $\widetilde{k}$ of $\widetilde{A}$ is an algebraically closed extension of $k$. Therefore, by the observations
 made in Section~\ref{para: flat local extension}, we may assume $k$ to be algebraically closed field\index{algebraically closed field}.
 We now consider the completion $\widehat{A}$ of $A$. Since $\widehat{A}$ is a flat local extension of $A$ such that
 $\mathfrak{m} \widehat{A}$ is the maximal ideal of $\widehat{A}$, we may as well assume that our local complete intersection ring $A$:
 \begin{itemize}
  \item[(i)] is complete, and hence $A = Q/({\bf f})$, where $(Q,\mathfrak{n})$ is a regular local ring and
             ${\bf f} = f_1, \ldots, f_c \in {\mathfrak{n}}^2$ is a $Q$-regular sequence;
  \item[(ii)] has an algebraically closed residue field $k$.
 \end{itemize}
\end{para}

\begin{para}[\textbf{Total Ext-module}]
 Let $U$ and $V$ be two finitely generated $A$-modules. Recall from Section~\ref{para:module structure 1} that
 \[
  \Ext_A^{\star}(U,V) := \bigoplus_{i \ge 0} \Ext_A^i(U,V)
 \]
 is the total Ext-module\index{total Ext-module} of $U$ and $V$ over the graded ring $A[t_1,\ldots,t_c]$ of cohomology operators $t_j$
 defined by ${\bf f}$, where $\deg(t_j) = 2$ for all $j = 1,\ldots,c$. We set
 \[
  \mathcal{C}(U,V) := \Ext_A^{\star}(U,V)\otimes_A k.
 \]
 Since $\Ext_A^{\star}(U,V)$ is a finitely generated graded $A[t_1,\ldots,t_c]$-module, $\mathcal{C}(U,V)$ is a finitely
 generated graded module over 
 \[
  \overline{T} := A[t_1,\ldots,t_c] \otimes_A k = k[t_1,\ldots,t_c].
 \]
\end{para}

\begin{para}[\textbf{Support Variety and Complexity}]\label{para: Support Variety defn}
 Define the {\it support variety}\index{support variety} $\mathscr{V}(U,V)$ of $U, V$ as the zero set in $k^c$ of the annihilator
 of $\mathcal{C}(U,V)$ in $\overline{T}$, that is
 \[
  \mathscr{V}(U,V) := \left\{ (b_1,\ldots,b_c) \in k^c :
  P(b_1,\ldots,b_c)=0 \mbox{ for all } P \in \ann_{\overline{T}}\big(\mathcal{C}(U,V)\big) \right\} \cup \{0\}.
 \]
 It is shown in \cite[Proposition~2.4(2)]{AB00} that
 \begin{center}\index{complexity via support variety}
  ($\dag$)\hfill $\cx_A(U,V) = \dim(\mathscr{V}(U,V)) = \dim_{\overline{T}}(\mathcal{C}(U,V))$.\index{dimension of a variety} \hfill \;
 \end{center}
\end{para}

 The following lemma is well-known. We give its proof also for the reader's convenience.

\begin{lemma}\label{lemma:stability of dim}
 Let $\mathscr{R} = \bigoplus_{n \ge 0} \mathscr{R}_n$ be a standard $\mathbb{N}$-graded ring, and let
 $\mathscr{M} = \bigoplus_{n \ge 0} \mathscr{M}_n$ be a finitely generated $\mathbb{N}$-graded $\mathscr{R}$-module. Then there
 exists some non-negative integer $j_0$ such that
 \[
  \dim_{\mathscr{R}_0}(\mathscr{M}_j) = \dim_{\mathscr{R}_0}(\mathscr{M}_{j_0})\quad\mbox{for all }j\ge j_0.
 \]
\end{lemma}

\begin{proof}
 Since $\mathscr{M}$ is a finitely generated $\mathbb{N}$-graded module over a standard $\mathbb{N}$-graded ring
 $\mathscr{R}$, there exists some $j' \ge 0$ such that
 \[
  \mathscr{M}_j = {\left(\mathscr{R}_{1}\right)}^{j - j'} \mathscr{M}_{j'} \quad\mbox{for all } j \ge j',
 \]
 which gives the following ascending chain of ideals:
 \[ 
  \ann_{\mathscr{R}_0}(\mathscr{M}_{j'}) \subseteq \ann_{\mathscr{R}_0}(\mathscr{M}_{j' + 1})
  \subseteq \ann_{\mathscr{R}_0}(\mathscr{M}_{j' + 2}) \subseteq \cdots.
 \]
 Since $\mathscr{R}_0$ is Noetherian, there exists some $j_0$ ($\ge j'$) such that $\ann_{\mathscr{R}_0}(\mathscr{M}_j) =
 \ann_{\mathscr{R}_0}(\mathscr{M}_{j_0})$, and hence $\dim_{\mathscr{R}_0}(\mathscr{M}_j) = \dim_{\mathscr{R}_0}(\mathscr{M}_{j_0})$
 for all $j \ge j_0$.
\end{proof}

\begin{para}[\textbf{Asymptotic Stability of Complexities $\cx_A(M,I^jN/I^{j+1}N)$}]\label{observation}
 Since
 \[
  \gr_I(N) = \bigoplus_{j \ge 0}I^jN/I^{j+1}N
 \]
 is a finitely generated graded $\mathscr{R}(I)$-module, in view of Theorem~\ref{theorem:finitely generated}, we have that
 \[
  \bigoplus_{j \ge 0} \Ext_A^{\star}(M,I^jN/I^{j+1}N) = \bigoplus_{j \ge 0} \bigoplus_{i \ge 0} \Ext_A^i(M,I^jN/I^{j+1}N)
 \]
 is a finitely generated graded $\mathscr{R}(I)[t_1,\ldots,t_c]$-module, and hence 
 \[
  \bigoplus_{j \ge 0}\Ext_A^{\star}(M,I^jN/I^{j+1}N) \otimes_A k
 \]
 is a finitely generated graded module over
 \[
  \mathscr{R}(I)[t_1,\ldots,t_c] \otimes_A k = F(I)[t_1,\ldots,t_c],
 \]
 where $F(I)$ is the fiber cone\index{fiber cone} of $I$ which is a finitely generated $k$-algebra. Writing
 \[
  F(I)[t_1,\ldots,t_c] = k[x_1,\ldots,x_m][t_1,\ldots,t_c] = \overline{T}[x_1,\ldots,x_m],
 \]
 we can say that
 \[
  \bigoplus_{j \ge 0}\Ext_A^{\star}(M, I^j N/I^{j+1} N) \otimes_A k
 \]
 is a finitely generated graded $\overline{T}[x_1,\ldots,x_m]$-module. So, by using Lemma~\ref{lemma:stability of dim}, we have
 \[ 
  \dim_{\overline{T}}\left(\Ext_A^\star(M,I^jN/I^{j+1}N) \otimes_A k\right) \quad \mbox{is constant for all }j \gg 0.
 \]
 Therefore, in view of ($\dag$) in Section~\ref{para: Support Variety defn}, we obtain that
 \begin{center}
  (C2) \hfill $\cx_A(M,I^jN/I^{j+1}N)$ \quad is constant for all $j \gg 0$. \hfill \;
 \end{center}
\end{para}

\section{Asymptotic Stability of Complexities}\label{Section: Stability of Complexities}
  
 Now we are in a position to prove the main result of this chapter.

\begin{theorem}\label{theorem:Stability of complexity}\index{stability of complexities}
 Let $(A,\mathfrak{m},k)$ be a local complete intersection ring. Let $M$ and $N$ be two finitely generated $A$-modules, and let
 $I$ be an ideal of $A$. Then
 \[
  \cx_A(M,N/I^jN)\quad\mbox{is constant for all $j \gg 0$}.
 \]
\end{theorem}

\begin{proof}
 By the observations made in Section~\ref{para: reduction to A s.t. k is alg closed}, we may assume that $A$ is complete and its
 residue field $k$ is algebraically closed.
 
 Fix $j \ge 0$. Consider the short exact sequence of $A$-modules
 \[
  0 \longrightarrow I^jN/I^{j+1}N \longrightarrow N/I^{j+1}N \longrightarrow N/I^jN \longrightarrow 0,
 \]
 which induces the following exact sequence of $A$-modules for each $i$:
 \begin{align*}
				              \Ext_A^{i-1}(M,N/I^jN) &\longrightarrow \\
  \Ext_A^i(M,I^jN/I^{j+1}N) \longrightarrow \Ext_A^i(M,N/I^{j+1}N) \longrightarrow \Ext_A^i(M,N/I^jN) &\longrightarrow \\
  \Ext_A^{i+1}(M,I^jN/I^{j+1}N). {~} \qquad\qquad\qquad\qquad\qquad\qquad\qquad\qquad\qquad {~} \quad  &
 \end{align*}
 Taking direct sum over $i$ and setting
 \[
   U_j := \bigoplus_{i \ge 0}\Ext_A^i(M,I^jN/I^{j+1}N) \quad\mbox{and}\quad V_j := \bigoplus_{i \ge 0}\Ext_A^i(M,N/I^jN),
 \]
 we obtain an exact sequence of $A[t_1,\ldots,t_c]$-modules:
 \[ 
   V_j(-1) \stackrel{\varphi_1}{\longrightarrow} U_j \stackrel{\varphi_2}{\longrightarrow}
   V_{j+1} \stackrel{\varphi_3}{\longrightarrow} V_j  \stackrel{\varphi_4}{\longrightarrow} U_j(1). 
 \]
 Now we set
 \[
  Z_j := \Image(\varphi_1), {~} X_j := \Image(\varphi_2) {~}\mbox{ and }{~} Y_j := \Image(\varphi_3).
 \]
 Thus we have the following commutative diagram of exact sequences:
 \[
  \xymatrixrowsep{2mm} \xymatrixcolsep{4mm}
  \xymatrix{
	      V_j(-1) \ar[rd] \ar[rr] && U_j \ar[rd] \ar[rr] && V_{j+1} \ar[rd] \ar[rr] && V_j \ar[rd] \ar[rr] && U_j(1) \\
	      & Z_j \ar[ru] \ar[rd] && X_j \ar[ru] \ar[rd] && Y_j \ar[ru] \ar[rd] && Z_j(1) \ar[ru] \ar[rd] \\
	      0 \ar[ru] && 0 \ar[ru] && 0 \ar[ru] && 0 \ar[ru] && 0. }
 \]
 Consider the following two short exact sequences of $A[t_1,\ldots,t_c]$-modules:
 \[
  0 \longrightarrow X_j \longrightarrow V_{j+1} \longrightarrow Y_j \longrightarrow 0 \quad\mbox{ and }\quad
  0 \longrightarrow Y_j \longrightarrow V_j \longrightarrow Z_j(1) \longrightarrow 0.
 \]
 Tensoring these sequences with $k$ over $A$, we get the following exact sequences of $\overline{T} = k[t_1,\ldots,t_c]$-modules:
 \begin{align*}
  X_j\otimes_A k \stackrel{\Phi_j}{\longrightarrow} &V_{j+1}\otimes_A k \longrightarrow Y_j\otimes_A k \longrightarrow 0,\\
  Y_j\otimes_A k \stackrel{\Psi_j}{\longrightarrow} &V_j\otimes_A k  \longrightarrow Z_j(1)\otimes_A k \longrightarrow 0.
 \end{align*}
 Now for each $j \ge 0$, we set
 \[
  X'_j := \Image(\Phi_j) {~}\mbox{ and }{~} Y'_j := \Image(\Psi_j)
 \]
 to get the following short exact sequences of $\overline{T}$-modules:
 \begin{align}
  0 \longrightarrow X'_j \longrightarrow & V_{j+1}\otimes_A k \longrightarrow Y_j\otimes_A k \longrightarrow 0,
											      \label{equation:support 1}\\
  0 \longrightarrow Y'_j \longrightarrow & V_j\otimes_A k  \longrightarrow Z_j(1)\otimes_A k \longrightarrow 0.
											      \label{equation:support 2}
 \end{align}
 
 By the observations made in Section~\ref{observation}, we have that $\bigoplus_{j \ge 0} U_j$ is a finitely generated graded
 $\mathscr{R}(I) [t_1, \ldots, t_c]$-module, and hence its submodule $\bigoplus_{j \ge 0} Z_j$ is also so.
 Therefore $\bigoplus_{j \ge 0} (Z_j\otimes_A k)$ is a finitely generated graded module over
 \[
  \mathscr{R}(I)[t_1,\ldots,t_c]\otimes_A k = F(I)[t_1,\ldots,t_c] = \overline{T}[x_1,\ldots,x_m].
 \]
 Therefore, by Lemma~\ref{lemma:stability of dim}, $\dim_{\overline{T}}(Z_j \otimes_A k) = z$ for all sufficiently large $j$, where $z$
 is some constant. Now considering the short exact sequences \eqref{equation:support 1} and \eqref{equation:support 2}, we obtain that
 \begin{align}
  \dim_{\overline{T}}(V_{j+1}\otimes_A k) &= \max\{ \dim_{\overline{T}}(X'_j), \dim_{\overline{T}}(Y_j\otimes_A k)\},
											  \label{equation:support dim 1} \\
  \dim_{\overline{T}}(V_j\otimes_A k) &= \max\{ \dim_{\overline{T}}(Y'_j), z \} \ge z  \label{equation:support dim 2}
 \end{align}
 for all sufficiently large $j$, say $j\ge j_0$.
 
 Note that $\dim_{\overline{T}}(V_j \otimes_A k) = \cx_A(M,N/I^jN)$ for all $j \ge 0$; see ($\dag$)
 in Section~\ref{para: Support Variety defn}. Therefore it is enough to prove that the stability of
 $\dim_{\overline{T}}(V_j \otimes_A k)$ holds for all sufficiently large $j$.
 
 If $\dim_{\overline{T}}(V_j \otimes_A k) = z$ for all $j \ge j_0$, then we are done. Otherwise there exists some $j \ge j_0$
 such that $\dim_{\overline{T}}(V_j \otimes_A k) > z$, and hence for this $j$, we have
 \[
  \dim_{\overline{T}}(V_j \otimes_A k) = \dim_{\overline{T}}(Y'_j) \le \dim_{\overline{T}}(Y_j \otimes_A k)
  \le \dim_{\overline{T}}(V_{j+1} \otimes_A k).
 \]
 First equality above occurs from \eqref{equation:support dim 2}, second inequality occurs because $Y'_j$ is a
 quotient module of $Y_j\otimes_A k$, and the last inequality occurs from \eqref{equation:support dim 1}.
 Note that $\dim_{\overline{T}}(V_{j+1} \otimes_A k) > z$. So, by applying a similar procedure, we have that
 \[
  \dim_{\overline{T}}(V_{j+1} \otimes_A k) \le \dim_{\overline{T}}(V_{j+2} \otimes_A k).
 \]
 In this way, we obtain a bounded non-decreasing sequence of non-negative integers:
 \[ 
  \dim_{\overline{T}}(V_j \otimes_A k) \le \dim_{\overline{T}}(V_{j+1} \otimes_A k) \le \dim_{\overline{T}}(V_{j+2} \otimes_A k)
  \le \cdots \le \dim\left(\overline{T}\right) < \infty,
 \]
 which eventually stabilizes somewhere, and hence the required stability holds.
\end{proof}

\newpage
\thispagestyle{empty}
\cleardoublepage

\chapter{Asymptotic Linear Bounds of Castelnuovo-Mumford Regularity}
						    \label{Chapter: Asymptotic linear bounds of Castelnuovo-Mumford regularity}
 
 Suppose $A$ is a standard $\mathbb{N}$-graded algebra over an Artinian local ring $A_0$. Let $I_1,\ldots,I_t$ be
 homogeneous ideals of $A$, and let $M$ be a finitely generated $\mathbb{N}$-graded $A$-module. Our main goal in this chapter is to show
 that there exist two integers $k_1$ and $k'_1$ such that
 \[
	  \reg(I_1^{n_1} \cdots I_t^{n_t} M) \le (n_1 + \cdots + n_t) k_1 + k'_1 \quad\mbox{for all }~ n_1,\ldots,n_t \in \mathbb{N}.
 \]
 We prove this result in a quite general set-up; see Hypothesis~\ref{hypothesis}
 and Theorem~\ref{theorem: bounds of regularity for multigraded module}.
 As a consequence, we also obtain the following: If $A_0$ is a field, then there exist two integers $k_2$ and $k'_2$ such that
 \[
	  \reg\left( \overline{I_1^{n_1}} \cdots \overline{I_t^{n_t}} M \right) \le (n_1 + \cdots + n_t) k_2 + k'_2
	  \quad\mbox{for all }~ n_1,\ldots,n_t \in \mathbb{N},
 \]
 where $\overline{I}$ denotes the integral closure of an ideal $I$ of $A$; see Definition~\ref{definition: integral closure}.
 
 We use the following notations throughout this chapter.

\begin{customnotations}{\ref{Chapter: Asymptotic linear bounds of Castelnuovo-Mumford regularity}.1}
 Throughout, $\mathbb{N}$ denotes the set of all non-negative integers and $t$ is any fixed positive integer. We use
 small letters with underline (e.g., $\underline{n}$) to denote elements of $\mathbb{N}^t$, and we use subscripts mainly to denote
 the coordinates of such an element, e.g., $\underline{n} = (n_1,n_2,\ldots,n_t)$. In particular, for every $1 \le i \le t$,
 $\underline{e}^i$ denotes the $i$th standard basis element of $\mathbb{N}^t$. We denote $\underline{0}$ the element of
 $\mathbb{N}^t$ with all components $0$. Throughout, we use the partial order on $\mathbb{N}^t$ defined by
 $\underline{n} \ge \underline{m}$ if and only if $n_i \ge m_i$ for all $1 \le i \le t$. For every $\underline{n} \in \mathbb{N}^t$,
 we set $|\underline{n}| := n_1 + \cdots + n_t$.
 If $R$ is an $\mathbb{N}^t$-graded ring and $L$ is an $\mathbb{N}^t$-graded $R$-module, then by $L_{\underline{n}}$,
 we always mean the $\underline{n}$th graded component of $L$.
\end{customnotations}
 
 By a {\it standard multigraded ring}\index{standard graded ring}, we mean a multigraded ring which is generated in total degree one,
 i.e., $R$ is a standard $\mathbb{N}^t$-graded ring if $R = R_{\underline{0}}[R_{\underline{e}^1},\ldots,R_{\underline{e}^t}]$.
 
 The rest of this chapter is organized as follows. We start by recalling some of the well-known basic results on Castelnuovo-Mumford
 regularity in Section~\ref{Castelnuovo-Mumford Regularity}; while in Section~\ref{section: preliminaries on multigraded modules}, we
 give some preliminaries on multigraded modules which we use in order to prove our main results on regularity. The announced linear
 boundedness results of regularity are proved in Section~\ref{section: linear bounds of regularity} through several steps.
 Finally, in Section~\ref{About linearity of regularity}, we discuss about asymptotic linearity of regularity
 for powers of several ideals.
 
\section{Castelnuovo-Mumford Regularity}\label{Castelnuovo-Mumford Regularity}
 
 Let $A = A_0[x_1,\ldots,x_d]$ be a standard $\mathbb{N}$-graded ring. Let $A_{+}$ be the irrelevant ideal
 $\langle x_1,\ldots,x_d \rangle$ of $A$ generated by the homogeneous elements of positive degree. Let $M$ be a finitely generated $\mathbb{N}$-graded $A$-module. For every integer $i \ge 0$, we denote the $i$th local cohomology module of $M$ with respect to $A_{+}$ by $H_{A_{+}}^i(M)$. For every integer $i \ge 0$, we set
 \[ 
	  a_i(M) := \left\{  \begin{array}{l l}
                     \max\left\{ \mu : H_{A_{+}}^i(M)_{\mu} \neq 0 \right\} & \quad \text{if } H_{A_{+}}^i(M) \neq 0\\
                    -\infty                                                & \quad \text{if } H_{A_{+}}^i(M) = 0.  \end{array} \right.
 \]
 Recall that the {\it Castelnuovo-Mumford regularity} of $M$ is defined by\index{Castelnuovo-Mumford regularity}
 \[
  \reg(M) := \max\left\{ a_i(M) + i : i \ge 0 \right\}.
 \]
 
 For a given short exact sequence of graded modules, by considering the corresponding long exact sequence of local
 cohomology modules, we can prove the following well-known result:
 
 \begin{lemma}\label{lemma: properties of regularity}\index{regularity and short exact sequence}
  Let $A$ be a standard $\mathbb{N}$-graded ring. Let $ 0 \longrightarrow M_1 \longrightarrow M_2 \longrightarrow M_3 \longrightarrow 0 $ be a short exact sequence of finitely generated $\mathbb{N}$-graded $A$-modules. Then we have the following inequalities:
  \begin{enumerate}[{\rm (i)}]
   \item $\reg(M_1) \le \max\{ \reg(M_2), \reg(M_3) + 1\}$.
   \item $\reg(M_2) \le \max\{ \reg(M_1), \reg(M_3)\}$.
   \item $\reg(M_3) \le \max\{ \reg(M_1) - 1, \reg(M_2)\}$.
  \end{enumerate}
 \end{lemma}
 
 The following lemma is well-known, but we cannot locate a reference for it. So we give a proof here for the completeness.
 
 \begin{lemma}\label{lemma: dimension reduction relation of regularity}\index{lemma on regularity}
  Let $A$ be a standard $\mathbb{N}$-graded ring, and let $M$ be a finitely generated $\mathbb{N}$-graded $A$-module. Let $x$ be a homogeneous element in $A$ of positive degree $l$. Then we have the following inequality:
  \[
    \reg(M) \le \max\{\reg(0 :_M x), \reg(M/xM) - l + 1\}.
  \]  
  Over polynomial rings over fields, if $x$ is such that $\dim(0 :_M x) \le 1$, then the inequality could be replaced
  by equality.
 \end{lemma}
 
 \begin{proof}
  We set $N := (0 :_M x)$. To prove the first part of the lemma, it is enough to prove that at least one of the following
  is true:
  \begin{equation}\label{lemma: dimension reduction relation of regularity: equation 1}
   \reg(M) \le \reg(N)\quad\mbox{or}\quad \reg(M) \le \reg(M/xM) - l + 1.
  \end{equation}
  We prove \eqref{lemma: dimension reduction relation of regularity: equation 1} by contradiction. If
  \eqref{lemma: dimension reduction relation of regularity: equation 1} is not true, then we have the following:
  \begin{equation}\label{lemma: dimension reduction relation of regularity: equation 2}
   \reg(M) > \reg(N)\quad\mbox{and}\quad \reg(M) > \reg(M/xM) - l + 1.
  \end{equation}
  Now consider the following short exact sequences of graded $A$-modules:
  \begin{align}
   0 \longrightarrow N(-l) \longrightarrow & M(-l) \stackrel{x\cdot}{\longrightarrow} xM \longrightarrow 0,
	\label{lemma: dimension reduction relation of regularity: equation 3} \\
   0 \longrightarrow xM \longrightarrow & M \longrightarrow M/xM \longrightarrow 0,
	\label{lemma: dimension reduction relation of regularity: equation 4}
  \end{align}
  where $M(-l)$ is same as $M$ but the grading is twisted by $-l$, and the second map of each of the sequences is the
  inclusion map. It directly follows from the definition of regularity that
  \[
   \reg(M(-l)) = \reg(M) + l.
  \]
  Then, from the short exact sequences \eqref{lemma: dimension reduction relation of regularity: equation 3} and
  \eqref{lemma: dimension reduction relation of regularity: equation 4}, by using Lemma~\ref{lemma: properties of regularity}, we have
  \begin{align}
   \reg(M) + l & \le \max\left\{ \reg(N) + l, \reg(xM) \right\} \nonumber \\
               & = \reg(xM) \quad \mbox{ [as $\reg(M) > \reg(N)$]},
                 \label{lemma: dimension reduction relation of regularity: equation 5}\\
     \reg(xM)  & \le \max\left\{ \reg(M), \reg(M/xM) + 1 \right\} \nonumber \\
               & = \reg(M/xM) + 1
      \quad \mbox{ [as $\reg(xM) > \reg(M)$ by \eqref{lemma: dimension reduction relation of regularity: equation 5}]}\nonumber \\
               & < \reg(M) + l \quad \mbox{ [by \eqref{lemma: dimension reduction relation of regularity: equation 2}]}.
                 \label{lemma: dimension reduction relation of regularity: equation 6}
  \end{align}
  Clearly \eqref{lemma: dimension reduction relation of regularity: equation 5} and
  \eqref{lemma: dimension reduction relation of regularity: equation 6} give a contradiction.
  
  For the second part of the lemma, we refer the reader to \cite[Remark~1.4.1]{Cha07}.
 \end{proof}
 
\section{Preliminaries on Multigraded Modules}\label{section: preliminaries on multigraded modules}
 
 Here we give some preliminaries on $\mathbb{N}^t$-graded modules which we use in the next section.
 We start with the following lemma:
 
\begin{lemma}\label{lemma: Artin-Rees}\index{lemmas on multigraded modules}
 Let $R = \bigoplus_{\underline{n}\in\mathbb{N}^t}R_{\underline{n}}$ be an $\mathbb{N}^t$-graded ring, and let
 $L = \bigoplus_{\underline{n} \in \mathbb{N}^t}L_{\underline{n}}$ be a finitely generated $\mathbb{N}^t$-graded $R$-module.
 Set $A := R_{\underline{0}}$. Let $J$ be an ideal of $A$. Then there exists a positive integer $k$ such that
 \[
  J^m L_{\underline{n}}\cap H_J^0(L_{\underline{n}}) = 0 \quad\mbox{for all }~ \underline{n} \in \mathbb{N}^t \mbox{ and }~ m\ge k.
 \]
\end{lemma}

\begin{proof}
 Let $I = JR$ be the ideal of $R$ generated by $J$. Since $R$ is Noetherian and $L$ is a finitely generated $R$-module,
 then by the Artin-Rees Lemma\index{Artin-Rees Lemma}, there exists a positive integer $c$ such that 
 \begin{align}
 (I^m L)\cap H_I^0(L) & = I^{m-c} \left((I^c L)\cap H_I^0(L)\right)\quad\mbox{for all }m\ge c\nonumber \\
                      & \subseteq I^{m-c} H_I^0(L) \quad\mbox{for all }m\ge c.    \label{lemma: Artin_Rees: equation 1}
 \end{align}
 Now consider the ascending chain of submodules of $L$:
 \[
  (0:_L I) \subseteq (0:_L I^2) \subseteq (0:_L I^3)\subseteq \cdots.
 \]
 Since $L$ is a Noetherian $R$-module, there exists some $l$ such that
 \begin{equation}\label{lemma: Artin_Rees: equation 2}
  (0:_L I^l) = (0:_L I^{l+1}) = (0:_L I^{l+2}) = \cdots = H_I^0(L).
 \end{equation}
 Set $k:= c+l$. Then, in view of \eqref{lemma: Artin_Rees: equation 1} and \eqref{lemma: Artin_Rees: equation 2}, we have
 \[
  (I^m L)\cap H_I^0(L) \subseteq I^{m-c} (0 :_L I^{m-c}) = 0 \quad\mbox{for all }m\ge k,
 \]
 which gives
 \[
  J^m L_{\underline{n}} \cap H_J^0(L_{\underline{n}}) = 0 \quad\mbox{for all }~ \underline{n} \in \mathbb{N}^t \mbox{ and }~ m\ge k.
 \]
\end{proof}

 Now we are aiming to obtain some invariant of multigraded module with the help of the following lemma:
 
 \begin{lemma}\label{lemma: annihilator stability for multigraded modules}\index{lemmas on multigraded modules}
  Let $R$ be a standard $\mathbb{N}^t$-graded ring and $L$ an $\mathbb{N}^t$-graded $R$-module finitely
  generated in degrees $\le \underline{u}$. Set $A := R_{\underline{0}}$. Then we have the following:
  \begin{enumerate}
   \item[\rm (i)] For every $\underline{v} \ge \underline{u}$, $\ann_A(L_{\underline{v}}) \subseteq
   \ann_A(L_{\underline{n}})$ for all $\underline{n} \ge \underline{v}$, and hence
   $\dim_A(L_{\underline{v}}) \ge \dim_A(L_{\underline{n}})$ for all $\underline{n} \ge \underline{v}$.
   \item[\rm (ii)] There exists $\underline{v} \in \mathbb{N}^t$ such that
    $\ann_A(L_{\underline{n}}) = \ann_A(L_{\underline{v}})$ for all $\underline{n} \ge \underline{v}$, and hence 
    $\dim_A(L_{\underline{n}}) = \dim_A(L_{\underline{v}})$ for all $\underline{n} \ge \underline{v}$.
  \end{enumerate}
 \end{lemma}
 
 \begin{proof}
  (i) Let $\underline{v} \ge \underline{u}$. Since $R$ is standard and $L$ is an $\mathbb{N}^t$-graded $R$-module
  finitely generated in degrees $\le \underline{u}$, for every $\underline{n} \ge \underline{v}$ ($\ge \underline{u}$), we have
  \[ 
    L_{\underline{n}} = R_{\underline{e}^1}^{n_1 - v_1} R_{\underline{e}^2}^{n_2 - v_2} \cdots
    R_{\underline{e}^t}^{n_t - v_t} L_{\underline{v}},
  \]
  which gives $\ann_A(L_{\underline{v}}) \subseteq \ann_A(L_{\underline{n}})$, and hence
  $\dim_A(L_{\underline{v}}) \ge \dim_A(L_{\underline{n}})$ for all $\underline{n} \ge \underline{v}$.
  
  (ii) Consider $\mathcal{C} := \{\ann_A(L_{\underline{n}}) : \underline{n} \ge \underline{u}\}$, a collection of
  ideals of $A$. Since $A$ is Noetherian, $\mathcal{C}$ has a maximal element $\ann_A(L_{\underline{v}})$, say.
  Then, by part (i), it follows that $\ann_A(L_{\underline{n}}) = \ann_A(L_{\underline{v}})$ for all
  $\underline{n} \ge \underline{v}$, and hence $\dim_A(L_{\underline{n}}) = \dim_A(L_{\underline{v}})$
  for all $\underline{n} \ge \underline{v}$.
 \end{proof}
 
 Let us introduce the following invariant of multigraded module on which we apply induction to prove our main result.
 
 \begin{definition}\label{definition: saturated dimension}
  Let $R$ be a standard $\mathbb{N}^t$-graded ring and $L$ a finitely generated $\mathbb{N}^t$-graded
  $R$-module. We call $\underline{v} \in \mathbb{N}^t$ an {\it annihilator stable point}\index{annihilator stable point} of $L$ if 
  \[
    \ann_{R_{\underline{0}}}(L_{\underline{n}}) = \ann_{R_{\underline{0}}}(L_{\underline{v}})\quad
    \mbox{for all }~\underline{n} \ge \underline{v}.
  \]
  In this case, we call $s := \dim_{R_{\underline{0}}}(L_{\underline{v}})$ the
  {\it saturated dimension}\index{saturated dimension} of $L$.
 \end{definition}
 
 \begin{remark}\label{remark: saturated dimension existence}
  Existence of an annihilator stable point of $L$ (with the hypothesis given in the
  Definition~\ref{definition: saturated dimension}) follows from
  Lemma~\ref{lemma: annihilator stability for multigraded modules}(ii). Let $\underline{v}, \underline{w} \in
  \mathbb{N}^t$ be two annihilator stable points of $L$, i.e., 
  \begin{align*}
   &\ann_{R_{\underline{0}}}(L_{\underline{n}}) = \ann_{R_{\underline{0}}}(L_{\underline{v}})\quad
     \mbox{for all }~\underline{n} \ge \underline{v} \\
   \mbox{and}\quad &\ann_{R_{\underline{0}}}(L_{\underline{n}}) = \ann_{R_{\underline{0}}}(L_{\underline{w}})
   \quad \mbox{for all }~\underline{n} \ge \underline{w}.
  \end{align*}
  If we denote $\dim_{R_{\underline{0}}}(L_{\underline{v}})$ and $\dim_{R_{\underline{0}}}(L_{\underline{w}})$ by
  $s(\underline{v})$ and $s(\underline{w})$ respectively, then observe that $s(\underline{v}) = s(\underline{w})$. Thus
  the saturated dimension of $L$ is well-defined.
 \end{remark}
 
 Let us recall the following result from \cite[Lemma~3.3]{Wes04}.
 
 \begin{lemma}\label{lemma: fixing one component of gradings}\index{lemmas on multigraded modules}
  Let $R$ be a standard $\mathbb{N}^t$-graded ring, and let $L$ be a finitely generated $\mathbb{N}^t$-graded
  $R$-module. For any fixed integers $1 \le i \le t$ and $\lambda \in \mathbb{N}$, set
  \[ S_i := \bigoplus_{\{\underline{n}\in\mathbb{N}^t : n_i = 0\}} R_{\underline{n}}\quad\quad\mbox{ and }\quad
     M_{i\lambda} := \bigoplus_{\{\underline{n}\in\mathbb{N}^t : n_i = \lambda \}} L_{\underline{n}}.
  \]
  Then $S_i$ is a Noetherian standard $\mathbb{N}^{t-1}$-graded ring and $M_{i\lambda}$ is a finitely
  generated $\mathbb{N}^{t-1}$-graded $S_i$-module.
 \end{lemma}
 
 \begin{discussion}\label{discussion}
  Let
  \[ 
    R = \bigoplus_{(\underline{n},i)\in \mathbb{N}^{t+1}} R_{(\underline{n},i)} \quad\quad\mbox{and}\quad
    L = \bigoplus_{(\underline{n},i)\in \mathbb{N}^{t+1}} L_{(\underline{n},i)}
  \]
  be an $\mathbb{N}^{t+1}$-graded ring and a finitely generated $\mathbb{N}^{t+1}$-graded
  $R$-module respectively. For every $\underline{n} \in \mathbb{N}^t$, we set
  \[ 
    R_{(\underline{n},\star)} := \bigoplus_{i \in \mathbb{N}} R_{(\underline{n},i)} \quad\quad\mbox{and}\quad
    L_{(\underline{n},\star)} := \bigoplus_{i \in \mathbb{N}} L_{(\underline{n},i)}.
  \]
  We give $\mathbb{N}^t$-grading structures on
  \[ R = \bigoplus_{\underline{n} \in \mathbb{N}^t} R_{(\underline{n},\star)} \quad\quad\mbox{and}\quad
     L = \bigoplus_{\underline{n} \in \mathbb{N}^t} L_{(\underline{n},\star)}
   \]
  in the obvious way, i.e., by setting $R_{(\underline{n},\star)}$ and $L_{(\underline{n},\star)}$ as the
  $\underline{n}$th graded components of $R$ and $L$ respectively. Then clearly, for any
  $\underline{m}, \underline{n} \in \mathbb{N}^t$, we have
  \[
   R_{(\underline{m},\star)} \cdot R_{(\underline{n},\star)} \subseteq R_{(\underline{m} + \underline{n},\star)}
   \quad\mbox{and}\quad 
   R_{(\underline{m},\star)} \cdot L_{(\underline{n},\star)} \subseteq L_{(\underline{m} + \underline{n},\star)}.
  \]
  Thus $R$ is an $\mathbb{N}^t$-graded ring and $L$ is an $\mathbb{N}^t$-graded $R$-module. Since we are changing
  only the grading, $R$ is anyway Noetherian. Since $L$ is finitely generated $\mathbb{N}^{t+1}$-graded $R$-module,
  it is just an observation that $L = \bigoplus_{\underline{n} \in \mathbb{N}^t} L_{(\underline{n},\star)}$ is
  finitely generated as $\mathbb{N}^t$-graded $R$-module. Now we set $A := R_{(\underline{0},\star)}$. Note that
  $A$ is a Noetherian $\mathbb{N}$-graded ring, and for every $\underline{n} \in \mathbb{N}^t$,
  $R_{(\underline{n},\star)}$ and $L_{(\underline{n},\star)}$ are finitely generated $\mathbb{N}$-graded
  $A$-modules.
 \end{discussion}
 
  We are going to refer the following hypothesis repeatedly in the rest of the present chapter.
  
 \begin{hypothesis}\label{hypothesis}
  Let
  \[ R = \bigoplus_{(\underline{n},i)\in \mathbb{N}^{t+1}} R_{(\underline{n},i)} \]
  be an $\mathbb{N}^{t+1}$-graded ring, {\it which need not be standard}. Let
  \[ L = \bigoplus_{(\underline{n},i)\in \mathbb{N}^{t+1}} L_{(\underline{n},i)} \]
  be a finitely generated $\mathbb{N}^{t+1}$-graded $R$-module. For every $\underline{n} \in \mathbb{N}^t$, we set
  \[ 
     R_{(\underline{n},\star)} := \bigoplus_{i \in \mathbb{N}} R_{(\underline{n},i)} \quad\quad\mbox{and}\quad
     L_{(\underline{n},\star)} := \bigoplus_{i \in \mathbb{N}} L_{(\underline{n},i)}.
  \]
  Also set $A := R_{(\underline{0},\star)}$. Suppose
  $R = \bigoplus_{\underline{n} \in \mathbb{N}^t} R_{(\underline{n},\star)}$ and $A = R_{(\underline{0},\star)}$
  are standard as $\mathbb{N}^t$-graded ring and $\mathbb{N}$-graded ring respectively, i.e.,
  \[
    R = R_{(\underline{0},\star)} [R_{(\underline{e}^1,\star)}, R_{(\underline{e}^2,\star)}, \ldots,
                                   R_{(\underline{e}^t,\star)}]
    \quad\mbox{and}\quad R_{(\underline{0},\star)} = R_{(\underline{0},0)}[R_{(\underline{0},1)}].
  \]
  Assume $A_0 = R_{(\underline{0},0)}$ is Artinian local with the
  maximal ideal $\mathfrak{m}$. Since $A$ is a Noetherian standard $\mathbb{N}$-graded ring, we assume that
  $A = A_0[x_1,\ldots,x_d]$ for some $x_1,\ldots,x_d \in A_1$. Let $A_{+} = \langle x_1,\ldots,x_d \rangle$.
 \end{hypothesis} 
 
  With the Hypothesis~\ref{hypothesis}, in view of Discussion~\ref{discussion}, we have the following:
  
  \begin{enumerate}
   \item[(0)] $R = \bigoplus_{(\underline{n},i)\in \mathbb{N}^{t+1}} R_{(\underline{n},i)}$ is not necessarily
   standard as an $\mathbb{N}^{t+1}$-graded ring.
   \item[(1)] $R = \bigoplus_{\underline{n} \in \mathbb{N}^t} R_{(\underline{n},\star)}$ is a standard
   $\mathbb{N}^t$-graded ring.
   \item[(2)] $A = R_{(\underline{0},\star)}$ is a standard $\mathbb{N}$-graded ring.
   \item[(3)] $L = \bigoplus_{\underline{n} \in \mathbb{N}^t} L_{(\underline{n},\star)}$ is a finitely generated
   $\mathbb{N}^t$-graded $R$-module.
   \item[(4)] For every $\underline{n} \in \mathbb{N}^t$, $R_{(\underline{n},\star)}$ and
   $L_{(\underline{n},\star)}$ are finitely generated $\mathbb{N}$-graded $A$-modules.
  \end{enumerate}

  We now give two examples which satisfy the Hypothesis~\ref{hypothesis}.
  
 \begin{example}\label{example 1 satisfying the hypothesis}\index{example on multigraded Rees module}
  Let $A$ be a standard $\mathbb{N}$-graded algebra over an Artinian local ring $A_0$. Let $I_1,\ldots,I_t$ be homogeneous ideals of $A$, and let $M$ be a finitely generated $\mathbb{N}$-graded $A$-module.
  Let $R = A[I_1 T_1,\ldots, I_t T_t]$ be the Rees algebra of $I_1,\ldots,I_t$ over the graded ring $A$, and let
  $L = M[I_1 T_1,\ldots, I_t T_t]$ be the Rees module of $M$ with respect to the ideals $I_1,\ldots,I_t$.
  We give $\mathbb{N}^{t+1}$-grading structures on $R$ and $L$ by setting $(\underline{n},i)$th graded
  components of $R$ and $L$ as the $i$th graded components of the $\mathbb{N}$-graded $A$-modules
  $I_1^{n_1}\cdots I_t^{n_t} A$ and $I_1^{n_1}\cdots I_t^{n_t} M$ respectively. Then clearly, $R$ is
  an $\mathbb{N}^{t+1}$-graded ring, and $L$ is a finitely generated $\mathbb{N}^{t+1}$-graded $R$-module.
  Note that $R$ is not necessarily standard as an $\mathbb{N}^{t+1}$-graded ring. Also note that for every
  $\underline{n} \in \mathbb{N}^t$, we have
  \begin{align*}
   R_{(\underline{n},\star)}
   &= \bigoplus_{i \in \mathbb{N}} R_{(\underline{n},i)} = I_1^{n_1}\cdots I_t^{n_t} A \\
   \mbox{and }\quad L_{(\underline{n},\star)}
   &= \bigoplus_{i \in \mathbb{N}} L_{(\underline{n},i)} = I_1^{n_1}\cdots I_t^{n_t} M.
  \end{align*}
  Recall that $\underline{e}^i$ denotes the $i$th standard basis element of $\mathbb{N}^t$. So
  $R_{(\underline{e}^i,\star)} = I_i$ for all $1 \le i \le t$. Therefore $R = A[I_1 T_1,\ldots, I_t T_t] =
  R_{(\underline{0},\star)}[R_{(\underline{e}^1,\star)},\ldots,R_{(\underline{e}^t,\star)}]$, and hence
  $R = \bigoplus_{\underline{n}\in \mathbb{N}^t}R_{(\underline{n},\star)}$ is standard as an $\mathbb{N}^t$-graded ring.
  Thus $R$ and $L$ are satisfying the Hypothesis~\ref{hypothesis}.
 \end{example}

 Let us recall the definition of the integral closure of an ideal.
 
 \begin{definition}\label{definition: integral closure}
  Let $I$ be an ideal of a ring $A$. An element $r \in A$ is said to be {\it integral} over $I$ if there exist an integer $n$ and
  elements $a_i \in I^i$, $i = 1,\ldots,n$ such that\index{integral element over an ideal}
  \[
   r^n + a_1 r^{n-1} + a_2 r^{n-2} + \cdots + a_{n-1} r + a_n = 0.
  \]
  The set of all elements that are integral over $I$ is called the {\it integral closure}\index{integral closure of an ideal} of $I$,
  and is denoted $\overline{I}$.
 \end{definition}

 Here is another example satisfying the Hypothesis~\ref{hypothesis}.
  
  \begin{example}\label{example 2 satisfying the hypothesis}
   Let $A, I_1, \ldots, I_t$ and $M$ be as in Example~\ref{example 1 satisfying the hypothesis}. Let $A_0$ be a field.
   Set $R := A[I_1 T_1, \ldots, I_t T_t]$ as above. Here we set
   \[
    L := \bigoplus_{\underline{n} \in \mathbb{N}^t} \left( \overline{I_1^{n_1}} \cdots \overline{I_t^{n_t}} M \right)
    T_1^{n_1} \cdots T_t^{n_t}.
   \]
   We give $\mathbb{N}^{t+1}$-grading structure on $L$ by setting $(\underline{n},i)$th graded
   component of $L$ as the $i$th graded component of the $\mathbb{N}$-graded $A$-module
   $\overline{I_1^{n_1}}\cdots \overline{I_t^{n_t}} M$. For a homogeneous ideal $I$ of $A$, since
   \[
    I^m ~\overline{I^n} \subseteq \overline{I^m}~\overline{I^n} \subseteq \overline{I^{m+n}} \quad \mbox{for all }m,n\in\mathbb{N},
   \]
   $L$ is an $\mathbb{N}^{t+1}$-graded $R$-module. Again for a homogeneous ideal $I$ of $A$, there exists an integer
   $n_0$ such that for all $n \ge n_0$, we have $\overline{I^n} = I^{n - n_0}~\overline{I^{n_0}}$ (see \cite[5.3.4(2)]{SH06}).
   Therefore $L$ is a finitely generated $\mathbb{N}^{t+1}$-graded $R$-module. Hence in this case also $R$ and $L$ are
   satisfying the Hypothesis~\ref{hypothesis}.
  \end{example}
  
  From now onwards, by $R$ and $L$, we mean $\mathbb{N}^t$-graded ring
  $\bigoplus_{\underline{n} \in \mathbb{N}^t} R_{(\underline{n},\star)}$ and $\mathbb{N}^t$-graded $R$-module
  $\bigoplus_{\underline{n} \in \mathbb{N}^t} L_{(\underline{n},\star)}$
  (satisfying the Hypothesis~\ref{hypothesis}) respectively.
  
\section{Linear Bounds of Regularity}\label{section: linear bounds of regularity}
 
 In this section, we are aiming to prove that the regularity of $L_{(\underline{n},\star)}$ as an $\mathbb{N}$-graded
 $A$-module is bounded by a linear function of $\underline{n}$ by using induction on the saturated dimension of the
 $\mathbb{N}^t$-graded $R$-module $L$. Here is the base case.
 
 \begin{theorem}\label{theorem: bounds of regularity for saturated dimension 0}\index{linear bounds of regularity}
  With the {\rm Hypothesis~\ref{hypothesis}}, let $L = \bigoplus_{\underline{n} \in \mathbb{N}^t} L_{(\underline{n},\star)}$
  be generated in degrees $\le \underline{u}$. If $\dim_A(L_{(\underline{v},\star)}) = 0$ for some
  $\underline{v} \ge \underline{u}$, then there exists an integer $k$ such that
  \[ \reg(L_{(\underline{n},\star)}) < |\underline{n} - \underline{u}| k + k \quad
     \mbox{for all }~\underline{n} \ge \underline{v}.\]
 \end{theorem}
 
 \begin{proof}
  Let $\dim_A(L_{(\underline{v},\star)}) = 0$ for some $\underline{v} \ge \underline{u}$. Then, by virtue of
  Lemma~\ref{lemma: annihilator stability for multigraded modules}(i), we obtain
  \[
    \dim_A(L_{(\underline{n},\star)}) = 0 \quad\mbox{ for all }~\underline{n} \ge \underline{v}.
  \]
  In view of Grothendieck's Vanishing Theorem (\cite[6.1.2]{BS13})\index{Grothendieck's Vanishing Theorem}, we get
  \[
    H_{A_{+}}^i(L_{(\underline{n},\star)}) = 0 \quad\mbox{ for all }~i > 0\mbox{ and }~\underline{n} \ge \underline{v}.
  \]
  Therefore in this case, we have
  \begin{equation}\label{theorem: bounds of regularity for saturated dimension 0: equation 1}
   \reg(L_{(\underline{n},\star)}) = \max\left\{\mu : H_{A_{+}}^0(L_{(\underline{n},\star)})_{\mu} \neq 0\right\}
   \quad\mbox{ for all }~\underline{n} \ge \underline{v}.
  \end{equation}
  
  Now consider the finite collection
  \[
    \mathcal{D} := \{R_{(\underline{e}^1,\star)},R_{(\underline{e}^2,\star)},\ldots,R_{(\underline{e}^t,\star)},
    L_{(\underline{u},\star)}\}.
  \]
  Since every member of $\mathcal{D}$ is a finitely generated $\mathbb{N}$-graded $A = A_0[x_1,\ldots,x_d]$-module, we may assume
  that every member of $\mathcal{D}$ is generated in degrees $\le k_1$ for some $k_1 \in \mathbb{N}$.
  Since $L$ is a finitely generated $\mathbb{N}^t$-graded $R$-module and $A_{+}$ is an ideal of
  $A$ ($= R_{(\underline{0},\star)}$), in view of Lemma~\ref{lemma: Artin-Rees},
  there exists a positive integer $k_2$ such that
  \begin{equation}\label{theorem: bounds of regularity for saturated dimension 0: equation 2}
   (A_{+})^{k_2} L_{(\underline{n},\star)} \cap H_{A_{+}}^0(L_{(\underline{n},\star)}) = 0 \quad
   \mbox{ for all }~\underline{n} \in \mathbb{N}^t.
  \end{equation}
  Now set $k := k_1 + k_2$. We claim that
  \begin{equation}\label{theorem: bounds of regularity for saturated dimension 0: equation 3}
   H_{A_{+}}^0(L_{(\underline{n},\star)})_{\mu} = 0\quad\mbox{for all }~\underline{n} \ge \underline{v}
   ~\mbox{ and }~\mu \ge |\underline{n} - \underline{u}| k + k.
  \end{equation}
  
  To show \eqref{theorem: bounds of regularity for saturated dimension 0: equation 3},
  fix $\underline{n} \ge \underline{v}$ and $\mu \ge |\underline{n} - \underline{u}| k + k$. Assume
  $X \in H_{A_{+}}^0(L_{(\underline{n},\star)})_{\mu}$. Note that
  the homogeneous (with respect to $\mathbb{N}$-grading over $A$) element $X$ of
  \[
  L_{(\underline{n},\star)} =
  R_{(\underline{e}^1,\star)}^{n_1 - u_1} R_{(\underline{e}^2,\star)}^{n_2 - u_2} \cdots
  R_{(\underline{e}^t,\star)}^{n_t - u_t} L_{(\underline{u},\star)}
  \]
  can be written as a finite sum of elements of the following type:
  \[(r_{11}r_{12}\cdots r_{1~n_1-u_1}) (r_{21}r_{22}\cdots r_{2~n_2-u_2})\cdots (r_{t1}r_{t2}\cdots r_{t~n_t-u_t}) Y\]
  for some homogeneous (with respect to $\mathbb{N}$-grading over $A$) elements
  \[
  r_{i1}, r_{i2},\ldots, r_{i~n_i - u_i} \in R_{(\underline{e}^i,\star)}\quad
  \mbox{for all }~ 1 \le i \le t,\mbox{ and }~Y \in L_{(\underline{u},\star)}.
  \]
  Considering the homogeneous degree with respect to $\mathbb{N}$-grading over $A$, we have
  \[ \deg(Y) + \sum_{i=1}^t\left\{\deg(r_{i1})+\deg(r_{i2})+\cdots+\deg(r_{i~n_i-u_i})\right\} = \mu
     \ge |\underline{n} - \underline{u}| k + k,
  \]
  which gives at least one of the elements
  \[r_{11},r_{12},\ldots,r_{1~n_1-u_1},\ldots,r_{t1},r_{t2},\ldots,r_{t~n_t-u_t}\quad\mbox{and }~Y\]
  is of degree $\ge k$. In first case, we consider $\deg(r_{ij}) \ge k$ for some $i,j$. Since
  $R_{(\underline{e}^i,\star)}$ is an $\mathbb{N}$-graded $A$-module generated in degrees $\le k_1$, we have
  \begin{align*}
   r_{ij} \in \left(R_{(\underline{e}^i,\star)}\right)_{\deg(r_{ij})}
   &= (A_1)^{\deg(r_{ij}) - k_1} \left(R_{(\underline{e}^i,\star)}\right)_{k_1}\\
   &\subseteq (A_{+})^{k_2} R_{(\underline{e}^i,\star)}\quad \mbox{[as $\deg(r_{ij})-k_1 \ge k-k_1 = k_2$]}.
  \end{align*}
  In another case, we consider $\deg(Y) \ge k$. In this case also, since $L_{(\underline{u},\star)}$ is an
  $\mathbb{N}$-graded $A$-module generated in degrees $\le k_1$, we have
  \begin{align*}
   Y \in \left(L_{(\underline{u},\star)}\right)_{\deg(Y)}
   &= (A_1)^{\deg(Y) - k_1} \left(L_{(\underline{u},\star)}\right)_{k_1}\\
   &\subseteq (A_{+})^{k_2} L_{(\underline{u},\star)}\quad \mbox{[as $\deg(Y)-k_1 \ge k-k_1 = k_2$]}.
  \end{align*}
  In both cases, the typical element
  $(r_{11}r_{12}\cdots r_{1~n_1-u_1}) \cdots (r_{t1}r_{t2}\cdots r_{t~n_t-u_t}) Y$ is in
  \[ 
   (A_{+})^{k_2} R_{(\underline{e}^1,\star)}^{n_1-u_1} R_{(\underline{e}^2,\star)}^{n_2-u_2} \cdots
   R_{(\underline{e}^t,\star)}^{n_t-u_t} L_{(\underline{u},\star)} = (A_{+})^{k_2} L_{(\underline{n},\star)},
  \]
  and hence $X \in (A_{+})^{k_2} L_{(\underline{n},\star)}$. Therefore
  \[
   X \in (A_{+})^{k_2} L_{(\underline{n},\star)} \cap H_{A_{+}}^0(L_{(\underline{n},\star)}),
  \]
  which gives $X = 0$ by \eqref{theorem: bounds of regularity for saturated dimension 0: equation 2}. Thus we have
  \[
   H_{A_{+}}^0(L_{(\underline{n},\star)})_{\mu} = 0\quad\mbox{for all }~\underline{n} \ge \underline{v}
   ~\mbox{ and }~\mu \ge |\underline{n} - \underline{u}| k + k,
  \]
  and hence the theorem follows from \eqref{theorem: bounds of regularity for saturated dimension 0: equation 1}.
 \end{proof}
 
 Now we give the inductive step to prove the following linear boundedness result.
 
 \begin{theorem}\label{theorem: bounds of regularity for saturated dimension positive}\index{linear bounds of regularity}
  With the {\rm Hypothesis~\ref{hypothesis}}, there exist $\underline{u} \in \mathbb{N}^t$ and an integer $k$ such that
  \[
    \reg(L_{(\underline{n},\star)}) < |\underline{n}| k + k \quad\mbox{for all }~\underline{n} \ge \underline{u}.
  \]
  In particular, if $t = 1$, then there exist two integers $k$ and $k'$ such that
  \[
    \reg(L_{(n,\star)}) \le n k + k' \quad\mbox{for all }~n \in \mathbb{N}.
  \]
 \end{theorem}
 
 \begin{proof}
  Let $\underline{v} \in \mathbb{N}^t$ be an annihilator stable point of $L$ and $s$ the saturated dimension of $L$.
  Without loss of generality, we may assume that $L$ is finitely generated as $R$-module in degrees
  $\le \underline{v}$. We prove the theorem by induction on $s$. If $s = 0$, then the theorem follows from
  Theorem~\ref{theorem: bounds of regularity for saturated dimension 0} by taking $\underline{u} := \underline{v}$.
  Therefore we may as well assume that $s > 0$ and the theorem holds true for all such finitely generated
  $\mathbb{N}^t$-graded $R$-modules of saturated dimension $\le s - 1$.
  
  Let $\mathfrak{n} := \mathfrak{m} \oplus A_{+}$ be the maximal homogeneous ideal of $A$. We claim that
  \[
    \mathfrak{n} \notin \Min\left(A/\ann_A(L_{(\underline{v},\star)})\right).
  \]
  Since the collection of all minimal prime ideals of $A$ containing $\ann_A(L_{(\underline{v},\star)})$ are
  associated prime ideals of $A/\ann_A(L_{(\underline{v},\star)})$, they are homogeneous, and hence they must be
  contained in $\mathfrak{n}$. Thus if the above claim is not true, then we have
  \[
    \Min\left(A/\ann_A(L_{(\underline{v},\star)})\right) = \{\mathfrak{n}\},
  \]
  and hence $s = \dim_A(L_{(\underline{v},\star)}) = 0$, which is a contradiction. Therefore the above claim is true,
  and hence by Prime Avoidance Lemma\index{Prime Avoidance Lemma} and using the fact that $\mathfrak{n}$ is the only
  homogeneous prime ideal of $A$ containing $A_{+}$ (as $(A_0,\mathfrak{m})$ is Artinian local), we have
  \[
    A_{+} \nsubseteq \bigcup\left\{P:P \in \Min\left(A/\ann_A(L_{(\underline{v},\star)})\right)\right\}.
  \]
  Then, by the Graded Version of Prime Avoidance Lemma\index{graded version of prime avoidance}, we may choose a homogeneous
  element $x$ in $A$ of positive degree such that
  \[
    x \notin \bigcup \left\{ P : P \in \Min\left(A/\ann_A(L_{(\underline{v},\star)})\right) \right\}.
  \]
  Since $\underline{v}$ is an annihilator stable point of $L$ and $s$ is the saturated dimension of $L$
  (Definition~\ref{definition: saturated dimension}), we have that
  $\ann_A(L_{(\underline{n},\star)}) = \ann_A(L_{(\underline{v},\star)})$ for all $\underline{n} \ge \underline{v}$, and $\dim_A(L_{(\underline{v},\star)}) = s$. Therefore, for all $\underline{n} \ge \underline{v}$, we obtain that
  \begin{align*}
   &\dim_A\left(L_{(\underline{n},\star)}/xL_{(\underline{n},\star)}\right),~
    \dim_A\big(0 :_{L_{(\underline{n},\star)}} x\big) \le \dim_A(L_{(\underline{n},\star)}) - 1 = s - 1 \\
   &\mbox{as }\quad \ann_A\left(L_{(\underline{n},\star)}/xL_{(\underline{n},\star)}\right),~
    \ann_A\big(0 :_{L_{(\underline{n},\star)}} x\big)\supseteq \langle \ann_A(L_{(\underline{n},\star)}), x \rangle.
  \end{align*}
  
  Now observe that $L/xL$ and $(0 :_L x)$ are finitely generated $\mathbb{N}^t$-graded $R$-modules with saturated
  dimensions $\le s - 1$. Therefore, by induction hypothesis, there exist $\underline{w}$ and $\underline{w}'$ in
  $\mathbb{N}^t$ and two integers $k_1, k_2$ such that 
  \begin{align*}
   \reg\left(L_{(\underline{n},\star)}/xL_{(\underline{n},\star)}\right)
   &< |\underline{n}| k_1 + k_1\quad\mbox{for all }\underline{n} \ge \underline{w} \\
   \mbox{and}\quad\reg\big(0 :_{L_{(\underline{n},\star)}} x\big)
   &< |\underline{n}| k_2 + k_2\quad\mbox{for all }\underline{n} \ge \underline{w}'.
  \end{align*}
  Set $k := \max\{k_1, k_2\}$ and $\underline{u} := \max\{\underline{w}, \underline{w}'\}$
  (i.e., $u_i := \max\{w_i,w'_i\}$ for all $1 \le i \le t$, and $\underline{u} := (u_1,\ldots,u_t)$). Then, by
  Lemma~\ref{lemma: dimension reduction relation of regularity}, we have
  \begin{align*}
   \reg(L_{(\underline{n},\star)}) &\le
   \max\left\{\reg\left(L_{(\underline{n},\star)}/xL_{(\underline{n},\star)}\right), \reg\big(0 :_{L_{(\underline{n},\star)}} x\big)\right\}\\
   & < |\underline{n}| k + k \quad \mbox{for all }\underline{n} \ge \underline{u}.
  \end{align*}
  This completes the proof of the first part of the theorem.
  
  To prove the second part, assume $t = 1$.  Then, from the first part, there exist $u \in \mathbb{N}$ and an integer $k$ such that
  \[ \reg(L_{(n,\star)}) < n k + k \quad\mbox{for all }~n \ge u.\]
  Set $k' := \max\left\{k, \reg(L_{(n,\star)}) : 0 \le n \le u-1\right\}$. Then clearly, we have
  \[ \reg(L_{(n,\star)}) \le n k + k' \quad\mbox{for all }~n \in \mathbb{N},\]
  which completes the proof of the theorem.
 \end{proof}
 
 Above theorem gives the result that $\reg(L_{(\underline{n},\star)})$ has linear bound for all
 $\underline{n} \ge \underline{u}$, for some $\underline{u} \in \mathbb{N}^t$. Now we prove the result
 for all $\underline{n} \in \mathbb{N}^t$.
 
 \begin{theorem}\label{theorem: bounds of regularity for multigraded module}\index{linear bounds of regularity}
  With the {\rm Hypothesis~\ref{hypothesis}}, there exist two integers $k$ and $k'$ such that
  \[ 
    \reg\left(L_{(\underline{n},\star)}\right) \le (n_1 + \cdots + n_t) k + k' \quad\mbox{for all }~\underline{n} \in \mathbb{N}^t.
  \]
 \end{theorem}
 
 \begin{proof}
  We prove the theorem by induction on $t$. If $t = 1$, then the theorem follows from the second part of the
  Theorem~\ref{theorem: bounds of regularity for saturated dimension positive}. Therefore we may as well assume that
  $t \ge 2$ and the theorem holds true for $t - 1$.
  
  By Theorem~\ref{theorem: bounds of regularity for saturated dimension positive}, there exist
  $\underline{u} \in \mathbb{N}^t$ and an integer $k_1$ such that
  \begin{equation}\label{theorem: bounds of regularity for multigraded module: equation 1}
   \reg\left(L_{(\underline{n},\star)}\right) < |\underline{n}| k_1 + k_1 \quad\mbox{for all }~\underline{n} \ge \underline{u}.
  \end{equation}
  Now for each $1 \le i \le t$ and $0 \le \lambda < u_i$, we set
  \[
   S_i := \bigoplus_{\{\underline{n} \in \mathbb{N}^t: n_i = 0\}} R_{(\underline{n},\star)}\quad\quad
   \mbox{and}\quad M_{i\lambda} := \bigoplus_{\{\underline{n} \in \mathbb{N}^t: n_i
                                 = \lambda\}} L_{(\underline{n},\star)}.
  \]
  Then, by Lemma~\ref{lemma: fixing one component of gradings}, $S_i$ is a  standard
  $\mathbb{N}^{t-1}$-graded ring and $M_{i\lambda}$ is a finitely generated $\mathbb{N}^{t-1}$-graded $S_i$-module.
  Therefore, by induction hypothesis, for each $1 \le i \le t$ and $0 \le \lambda < u_i$, there exist two integers
  $k_{i\lambda}$ and $k'_{i\lambda}$ such that
  \begin{align}\label{theorem: bounds of regularity for multigraded module: equation 2}
   \reg\left(L_{(n_1,\ldots,n_{i-1},\lambda,n_{i+1},\ldots,n_t,\star)}\right)
   &\le (n_1+\cdots+n_{i-1}+n_{i+1}+\cdots+n_t) k_{i\lambda} + k'_{i\lambda}                      \nonumber \\
   &= (n_1+\cdots+n_{i-1}+\lambda+n_{i+1}+\cdots+n_t) k_{i\lambda} + k''_{i\lambda}                          \\
   &\quad\quad\quad\quad\quad~~\mbox{ for all }n_1,\ldots,n_{i-1},n_{i+1},\ldots,n_t \in \mathbb{N},\nonumber
  \end{align}
  where $k''_{i\lambda} = k'_{i\lambda} - \lambda k_{i\lambda}$. Now set
  \begin{align*}
   k  &:= \max\left\{k_1, k_{i\lambda} : 1 \le i \le t, 0 \le \lambda < u_i\right\}\quad\mbox{and}\\
   k' &:= \max\left\{k_1, k''_{i\lambda} : 1 \le i \le t, 0 \le \lambda < u_i\right\}.
  \end{align*}
  We claim that
  \begin{equation}\label{theorem: bounds of regularity for multigraded module: equation 3}
   \reg\left(L_{(\underline{n},\star)}\right) \le |\underline{n}| k + k' \quad\mbox{for all }~\underline{n} \in \mathbb{N}^t.
  \end{equation}
  To prove \eqref{theorem: bounds of regularity for multigraded module: equation 3}, consider an arbitrary
  $\underline{n} \in \mathbb{N}^t$. If $\underline{n} \ge \underline{u}$, then
  \eqref{theorem: bounds of regularity for multigraded module: equation 3} follows from
  \eqref{theorem: bounds of regularity for multigraded module: equation 1}. Otherwise if
  $\underline{n} \ngeqslant \underline{u}$, then we have $n_i < u_i$ for at least one $i \in \{1,\ldots,t\}$,
  and hence in this case, \eqref{theorem: bounds of regularity for multigraded module: equation 3}
  holds true by \eqref{theorem: bounds of regularity for multigraded module: equation 2}.
 \end{proof}
 
 Now we have arrived at the main goal of this chapter.
 
 \begin{corollary}\label{corollary: bounds of regularity of ideals power times module}\index{linear bounds of regularity}
  Let $A$ be a  standard $\mathbb{N}$-graded algebra over an Artinian local ring $A_0$. Let $I_1,\ldots,I_t$ be homogeneous ideals of $A$, and let $M$ be a finitely generated $\mathbb{N}$-graded $A$-module.
  Then there exist two integers $k$ and $k'$ such that
  \[
    \reg(I_1^{n_1}\cdots I_t^{n_t} M) \le (n_1 + \cdots + n_t) k + k' \quad\mbox{for all }~ n_1,\ldots,n_t \in \mathbb{N}.
  \]
 \end{corollary}
 
 \begin{proof}
  Let $R = A[I_1 T_1,\ldots, I_t T_t]$ be the Rees algebra of $I_1,\ldots,I_t$ over the graded ring $A$ and let
  $L = M[I_1 T_1,\ldots, I_t T_t]$ be the Rees module of $M$ with respect to the ideals $I_1,\ldots,I_t$. We give
  $\mathbb{N}^{t+1}$-grading structures on $R$ and $L$ by setting $(\underline{n},i)$th graded
  components of $R$ and $L$ as the $i$th graded components of the $\mathbb{N}$-graded $A$-modules
  $I_1^{n_1}\cdots I_t^{n_t} A$ and $I_1^{n_1}\cdots I_t^{n_t} M$ respectively.
  From Example~\ref{example 1 satisfying the hypothesis}, note that $R$ and $L$ are satisfying the
  Hypothesis~\ref{hypothesis}, and in this case
  \[ L_{(\underline{n},\star)} = I_1^{n_1}\cdots I_t^{n_t} M \quad\mbox{ for all }~\underline{n} \in \mathbb{N}^t.\]
  Therefore the corollary follows from Theorem~\ref{theorem: bounds of regularity for multigraded module}.
 \end{proof}
 
 Here is another corollary of the Theorem~\ref{theorem: bounds of regularity for multigraded module}.
 
 \begin{corollary}\label{corollary: bounds of regularity, integral closure of ideals power}\index{linear bounds of regularity}
  Let $A$ be a  standard $\mathbb{N}$-graded algebra over a field $A_0$. Let  $I_1,\ldots,I_t$ be homogeneous ideals of $A$, and let $M$ be a finitely generated $\mathbb{N}$-graded $A$-module. Then there exist two integers $k$ and $k'$ such that
  \[
    \reg\left(\overline{I_1^{n_1}}\cdots \overline{I_t^{n_t}} M\right) \le (n_1 + \cdots + n_t) k + k'
    \quad\mbox{for all }~ n_1,\ldots,n_t \in \mathbb{N}.
  \]
 \end{corollary}
 
 \begin{proof}
  We set $R$ and $L$ as in Example~\ref{example 2 satisfying the hypothesis}. From
  Example~\ref{example 2 satisfying the hypothesis}, note that $R$ and $L$ are satisfying the Hypothesis~\ref{hypothesis},
  and in this case
  \[
   L_{(\underline{n},\star)} = \overline{I_1^{n_1}}\cdots \overline{I_t^{n_t}} M \quad\mbox{for all }~\underline{n} \in \mathbb{N}^t.
  \]
  Hence the corollary follows from Theorem~\ref{theorem: bounds of regularity for multigraded module}.
 \end{proof}
 
\section{About Linearity of Regularity}\label{About linearity of regularity}

  In \cite[page 252]{CHT99}, S. D. Cutkosky, J. Herzog and N. V. Trung remarked that over a polynomial ring $S = k[X_1,\ldots,X_d]$
  over a field $k$, the asymptotic linearity of regularity holds for a collection of ideals, i.e.,\index{Cutkosky, Herzog and Trung}
  \[
    \reg(I_1^{n_1}\cdots I_t^{n_t}) = a_1 n_1 + \cdots + a_t n_t + b \quad \mbox{for all }~n_1,\ldots,n_t\gg 0
  \]\index{linearity of regularity}
  and for some constants $a_1, \ldots, a_t, b$. {\it However, their proof has a gap which we now describe.}
  
  \begin{para}\label{explanation about maximum of linear functions}
   In the proof of Theorem~3.4 in \cite{CHT99}, it is used repeatedly that ``the maximum of finitely many linear functions
   is asymptotically linear". But this is not true in general for functions of more than one variable (see
   Example~\ref{example: maximum of linear functions need not be linear}). Therefore we cannot conclude that for every $i \ge 0$,
   $\reg_i(I_1^{n_1}\cdots I_t^{n_t})$ is linear in $(n_1,\ldots,n_t)$ for all sufficiently large $n_1,\ldots,n_t$, where for a finitely generated
   $\mathbb{N}$-graded $S$-module $N$, $\reg_i(N)$ is defined to be
   \[
	    \reg_i(N) := \sup\{ n : \Tor_i^S(N,k)_n \neq 0 \} - i.
   \]
   Even if $\reg_i(I_1^{n_1}\cdots I_t^{n_t})$ is asymptotically linear for all $i \ge 0$, then
   also it is not clear whether
   \[ 
	    \reg(I_1^{n_1}\cdots I_t^{n_t}) = \max\{\reg_i(I_1^{n_1}\cdots I_t^{n_t}) : i \ge 0\}
   \]
   is linear in $(n_1,\ldots,n_t)$ for all sufficiently large $n_1,\ldots,n_t$.
  \end{para}
  
  Here is an example which shows that the maximum of finitely many linear functions need not be asymptotically linear.
  
  \begin{example}\label{example: maximum of linear functions need not be linear}\index{example on linear functions}
   Set $h(m,n) := \max\{2m + n, m + 3n \}$ for all $m ,n \in \mathbb{N}$. We claim that $h(m,n)$ is not asymptotically a
   linear function in $(m, n)$.
   
   If possible, assume that $h(m,n) = am + bn + c$ for all $m, n \gg 0$, say for all $(m,n)\ge (u,v)$, where $a,b$ and $c$ are
   some constants.
   
   Note that for every fixed $n \in \mathbb{N}$, $h(m,n) = 2m + n$ for all $m \gg 0$. Therefore for every fixed $n \ge v$, we
   have $am + bn + c = 2m + n$ for all $m \gg 0$, which implies that $a = 2$, and hence $bn + c = n$. Thus for all
   $n \ge v$, we have $bn + c = n$, which gives $b = 1$ and $c = 0$. In this way, we have $a = 2$, $b = 1$ and $c = 0$.
   
   Again for every fixed $m \in \mathbb{N}$, $h(m,n) = m + 3n$ for all $n \gg 0$. Therefore for every fixed $m \ge u$, we
   have $am + bn + c = m + 3n$ for all $n \gg 0$, which implies that $b = 3$, and hence $am + c = m$. Thus for all
   $m \ge u$, we have $am + c = m$, which gives $a = 1$ and $c = 0$. In this way, we have $a = 1$, $b = 3$ and $c = 0$,
   which gives a contradiction. Therefore $h(m,n)$ is not asymptotically a linear function in $(m, n)$.
  \end{example}
  
\newpage
\thispagestyle{empty}
\cleardoublepage

\chapter{Characterizations of Regular Local Rings via Syzygy Modules}
                                                         \label{Chapter: Characterizations of Regular Local Rings via Syzygy Modules}
 
 This chapter shows some instances where properties of a local ring are closely connected with the homological properties of a single
 module.
 
 Let $A$ be a  local ring with residue field $k$. The aim of this chapter is to prove that if a finite direct
 sum of syzygy modules of $k$ maps onto `a semidualizing module' or `a non-zero maximal Cohen-Macaulay module of finite injective
 dimension', then $A$ is regular (see Corollaries~\ref{corollary: RLR and surjection onto semidualizing}
 and \ref{corollary: RLR and surjection onto finite injdim}). We also obtain one new characterization of regular local
 rings, namely, the local ring $A$ is regular if and only if some syzygy module of $k$ has a non-zero direct summand of finite
 injective dimension; see Theorem~\ref{theorem: characterization of RLR, injdim}. Moreover, this result has a dual companion.
 
 We use the following notations throughout this chapter.

\begin{customnotations}{\ref{Chapter: Characterizations of Regular Local Rings via Syzygy Modules}.1}\label{Notation: Syzygy}
 Throughout the present chapter, {\it $A$ always denotes a  local ring with maximal ideal $\mathfrak{m}$ and residue field
 $k$}. Let $M$ be a finitely generated $A$-module. For a non-negative integer $n$, we denote $\Omega_n^A(M)$ (resp.
 $\Omega_{-n}^A(M)$) the $n$th syzygy (resp. cosyzygy) module of $M$. We denote a finite collection of non-negative integers by
 $\Lambda$.
\end{customnotations}

 The structure of this chapter is as follows. In Section~\ref{Section: Preliminaries on Syzygy Modules}, we provide some preliminaries
 on syzygy modules which we use in the next section. Finally, we prove our main results of this chapter in
 Section~\ref{Section: Characterizations of Regular Local Rings}.
 
\section{Preliminaries on Syzygy Modules}\label{Section: Preliminaries on Syzygy Modules}
 
 In the present section, we give some preliminaries which we use in order to prove our main results of this chapter.
 We start with the following lemma which gives a relation between the socle of the ring and the annihilator of the syzygy modules.
 
 \begin{lemma}\label{lemma: socle, syzygy}\index{socle and syzygy relation}
  With the Notation~\ref{Notation: Syzygy}, for every positive integer $n$, we have
  \[
   \Soc(A) \subseteq \ann_A\left( \Omega_n^A(M) \right).
  \]
  In particular, if $A \neq k$ {\rm (}i.e., if $\mathfrak{m} \neq 0${\rm )}, then
  \[
   \Soc(A) \subseteq \ann_A\left( \Omega_n^A(k) \right) \quad \mbox{for all } n \ge 0.
  \]
 \end{lemma}
 
 \begin{proof}
  Fix $n \ge 1$. If $\Omega_n^A(M) = 0$, then we are done. So we may assume $\Omega_n^A(M) \neq 0$.
  Consider the following commutative diagram in a minimal free resolution of $M$:
  \[
   \xymatrixrowsep{12mm} \xymatrixcolsep{8mm}
   \xymatrix{ \cdots \ar[r] & A^{b_n} \ar@{->>}[rd]_{f} \ar[rr]^{\delta} 	& 			& A^{b_{n-1}}\ar[r] & \cdots.\\
   			    &                                     		& \Omega_n^A(M)\ar@{^{(}->}[ru]_g	}
  \]
  Let $a \in \Soc(A)$, i.e., $a\mathfrak{m} = 0$. Suppose $x \in \Omega_n^A(M)$. Since $f$ is surjective, there exists
  $y \in A^{b_n}$ such that $f(y) = x$. Note that $\delta(ay) = a\delta(y) = 0$ as
  $\delta\left(A^{b_n}\right) \subseteq \mathfrak{m} A^{b_{n-1}}$ and $a\mathfrak{m} = 0$.  Therefore
  $g(ax) = g(f(ay)) = \delta(ay) = 0$, which gives $ax = 0$ as $g$ is injective.
  Thus $\Soc(A) \subseteq \ann_A\left(\Omega_n^A(M)\right)$.
  
  For the last part, note that $\Soc(A) \subseteq \mathfrak{m} = \ann_A\left(\Omega_0^A(k)\right)$ if $\mathfrak{m} \neq 0$.
 \end{proof}
 
 Let us recall the following well-known result initially obtained by Nagata.
 
 \begin{proposition}{\rm \cite[Corollary~5.3]{Tak06}}\label{proposition: Nagata}\index{Nagata's result on syzygy modules}
  Let $x \in \mathfrak{m} \smallsetminus \mathfrak{m}^2$ be an $A$-regular element. Set $\overline{(-)} := (-) \otimes_A A/(x)$.
  Then we have
  \[ 
   \overline{\Omega_n^A(k)} \cong \Omega_n^{\overline{A}}(k)\oplus \Omega_{n-1}^{\overline{A}}(k)\quad\mbox{for every integer $n\ge 1$}.
  \]
 \end{proposition}

 We notice two properties satisfied by semidualizing modules (see Definition~\ref{definition: semidualizing module})
 and maximal Cohen-Macaulay modules of finite injective dimension.
 
 \begin{definition}\label{definition: star property}
  Let $\mathcal{P}$ be a property of modules over  local rings. We say that $\mathcal{P}$ is a $(*)$-property if $\mathcal{P}$ satisfies
  the following:\index{$(*)$-property of modules}
  \begin{enumerate}[(i)]
   \item An $A$-module $M$ satisfies $\mathcal{P}$ implies that the $A/(x)$-module $M/xM$ satisfies $\mathcal{P}$, where $x \in A$ is an
         $A$-regular element.
   \item An $A$-module $M$ satisfies $\mathcal{P}$ and $\depth(A) = 0$ together imply that $\ann_A(M) = 0$.
  \end{enumerate}
 \end{definition}
 
 Now we give a few examples of $(*)$-properties.
  
 \begin{example}\label{example: star property: semidualizing}\index{examples of $(*)$-properties}
  The property $\mathcal{P}_1$ $:=$ `semidualizing modules over local rings' is a $(*)$-property.
 \end{example}
 
 \begin{proof}
  Let $C$ be a semidualizing $A$-module. It is shown in \cite[page~68]{Gol84} that $C/xC$ is a semidualizing $A/(x)$-module, where
  $x \in \mathfrak{m}$ is an $A$-regular element. Since $\Hom_A(C,C) \cong A$, we have $\ann_A(C) = 0$
  (without any restriction on $\depth(A)$).
 \end{proof}
 
 Here is another example of $(*)$-property.
 
 \begin{example}\label{example: star property: MCM, finite injdim}\index{examples of $(*)$-properties}
  The property $\mathcal{P}_2$ $:=$ `non-zero maximal Cohen-Macaulay modules of finite injective dimension over Cohen-Macaulay local
  rings' is a $(*)$-property.
 \end{example}
 
 \begin{proof}
  Let $A$ be a Cohen-Macaulay local ring, and let $L$ be a non-zero maximal Cohen-Macaulay $A$-module of finite injective dimension.
  Suppose $x \in A$ is an $A$-regular element. Since $L$ is a maximal Cohen-Macaulay $A$-module, $x$ is $L$-regular as well. Therefore
  $L/xL$ is a non-zero maximal Cohen-Macaulay module of finite injective dimension over the Cohen-Macaulay local ring $A/(x)$
  (see \cite[Corollary~3.1.15]{BH98}).
  
  Now further assume that $\depth(A) = 0$. Then $A$ is an Artinian local ring, and $\injdim_A(L) = \depth(A) = 0$. Therefore,
  by \cite[Theorem~3.2.8]{BH98}, we have that $L \cong E^r$, where $E$ is the injective hull of $k$, and
  $r = \rank_k\left(\Hom_A(k,L)\right)$. It is well-known that $\Hom_A(E, E) \cong A$ as $A$ is an Artinian local ring.
  Hence $\ann_A(L) = \ann_A(E) = 0$.
 \end{proof}
 
\section{Characterizations of Regular Local Rings}\label{Section: Characterizations of Regular Local Rings}
 
 Now we can achieve the aim of this chapter. First of all, we prove that if a finite direct sum of syzygy modules of the
 residue field maps onto a non-zero module satisfying a $(*)$-property, then the ring is regular.
 
\begin{theorem}\label{theorem: RLR and surjection onto star module}\index{regular ring via syzygy modules}
 Assume that $\mathcal{P}$ is a $(*)$-property {\rm (}see {\rm Definition~\ref{definition: star property})}. Let
 \[
   f : \bigoplus_{n \in \Lambda} {\left( \Omega_n^A(k) \right)}^{j_n} \longrightarrow L
   \quad\quad\mbox{$(j_n \ge 1$ for each $n \in \Lambda)$}
 \]\index{characterization of regular rings}
 be a surjective $A$-module homomorphism, where $L$ {\rm(}$\neq 0${\rm)} satisfies $\mathcal{P}$. Then $A$ is regular.
\end{theorem}
 
 \begin{proof}
  We prove the theorem by using induction on $t := \depth(A)$. Let us first assume that $t = 0$. If possible, let $A \neq k$, i.e.,
  $\mathfrak{m} \neq 0$. Since $\depth(A) = 0$, we have $\Soc(A) \neq 0$. But, by virtue of Lemma~\ref{lemma: socle, syzygy}, we obtain
  \begin{align*}
   \Soc(A) & \subseteq \bigcap_{n \in \Lambda} \ann_A\left(\Omega_n^A(k)\right) \\
           & = \ann_A\left( \bigoplus_{n \in \Lambda} {\left( \Omega_n^A(k) \right)}^{j_n} \right) \\
           & \subseteq \ann_A(L) \quad
                    \mbox{[as $f : \bigoplus_{n\in\Lambda} {\left( \Omega_n^A(k) \right)}^{j_n} \longrightarrow L$ is surjective]}\\
           & = 0 \qquad\quad\quad \mbox{[as $L$ satisfies $\mathcal{P}$ which is a $(*)$-property]}.
  \end{align*}
  This gives a contradiction. Therefore $A$ ($= k$) is a regular local ring.
  
  Now we assume that $t \ge 1$. Suppose the theorem holds true for all such rings of depth smaller than $t$. Since
  $\depth(A) \ge 1$, there exists an element $x \in \mathfrak{m} \smallsetminus \mathfrak{m}^2$ which is $A$-regular. We set
  $\overline{(-)} := (-) \otimes_A A/(x)$. Clearly,
  \[
    \overline{f} : \bigoplus_{n\in\Lambda} {\left( \overline{\Omega_n^A(k)} \right)}^{j_n} \longrightarrow  \overline{L}
  \]
  is a surjective $\overline{A}$-module homomorphism, where the $\overline{A}$-module $\overline{L}$ {\rm(}$\neq 0${\rm)}
  satisfies $\mathcal{P}$ as $\mathcal{P}$ is a $(*)$-property. Since $x \in \mathfrak{m} \smallsetminus \mathfrak{m}^2$ is an
  $A$-regular element, by Proposition~\ref{proposition: Nagata}, we have
  \[ 
     \bigoplus_{n \in \Lambda} \left( \overline{\Omega_n^A(k)} \right)^{j_n} \cong
     \bigoplus_{n\in\Lambda} \left( \Omega_n^{\overline{A}}(k) \oplus \Omega_{n-1}^{\overline{A}}(k) \right)^{j_n}
     \quad \mbox{[by setting $\Omega_{-1}^{\overline{A}}(k) := 0$]}.
  \]
  Since $\depth(\overline{A}) = \depth(A) - 1$, by the induction hypothesis, $\overline{A}$ is a regular local ring, and hence $A$
  is a regular local ring as $x \in \mathfrak{m} \smallsetminus \mathfrak{m}^2$ is an $A$-regular element.
 \end{proof}
 
 As a few applications of the above theorem, we obtain the following necessary and sufficient conditions for a  local ring
 to be regular.
 
 \begin{corollary}\label{corollary: RLR and surjection onto semidualizing}\index{regular ring via syzygy modules}
  Let\index{characterization of regular rings}
  \[
    f : \bigoplus_{n\in\Lambda} \left( \Omega_n^A(k) \right)^{j_n} \longrightarrow L
  \]
  be a surjective $A$-module homomorphism, where $L$ is a semidualizing $A$-module. Then $A$ is regular.
 \end{corollary}
 
 \begin{proof}
  The corollary follows from Theorem~\ref{theorem: RLR and surjection onto star module}
  and Example~\ref{example: star property: semidualizing}.
 \end{proof}

 \begin{remark}\label{remark: Dutta and Takahashi's result}
	We can recover Theorem~\ref{theorem: Takahashi; regular; semidualizing} (in particular, Theorem~\ref{theorem: Dutta} because $A$ itself is a semidualizing $A$-module) as a consequence of Corollary~\ref{corollary: RLR and surjection onto semidualizing}. In fact the above result is even stronger than Theorem~\ref{theorem: Takahashi; regular; semidualizing}.
 \end{remark}
 
 Now we give a partial answer to Question~\ref{question: surjection onto finite injdim}.
 
 \begin{corollary}\label{corollary: RLR and surjection onto finite injdim}\index{regular ring via syzygy modules}
  Let\index{characterization of regular rings}
  \[ 
    f : \bigoplus_{n\in\Lambda} \left( \Omega_n^A(k) \right)^{j_n} \longrightarrow L
  \]
  be a surjective $A$-module homomorphism, where $L$ {\rm (}$\neq 0${\rm )} is a maximal Cohen-Macaulay $A$-module of finite injective dimension. Then $A$ is regular.
 \end{corollary}
 
 \begin{proof}
 	Existence of a non-zero finitely generated $A$-module of finite injective dimension ensures that\index{Bass's conjecture} $A$ is Cohen-Macaulay
 	(see \cite[Corollary~9.6.2 and Remark~9.6.4(a)(ii)]{BH98} and \cite{Rob87}). Therefore we may assume that $A$ is a Cohen-Macaulay local ring. Then the corollary follows from Theorem~\ref{theorem: RLR and surjection onto star module} and Example~\ref{example: star property: MCM, finite injdim}.
 \end{proof}
 
 \begin{remark}\label{remark: partial answer for Artinian rings}
  It is clear from the above corollary that Question~\ref{question: surjection onto finite injdim} has an affirmative answer for
  Artinian local rings.
 \end{remark}

 Let $A$ be a Cohen-Macaulay local ring. Recall that a maximal Cohen-Macaulay $A$-module $\omega$ of type $1$ and of finite injective
 dimension is called the {\it canonical module}\index{canonical module} of $A$. It is well-known that the canonical module $\omega$ of
 $A$ is a semidualizing $A$-module; see, e.g., \cite[Theorem~3.3.10]{BH98}. So, both
 Corollary~\ref{corollary: RLR and surjection onto semidualizing} and Corollary~\ref{corollary: RLR and surjection onto finite injdim}
 yield the following result (independently) which strengthens Corollary~\ref{corollary: Takahashi; regular; canonical module}.
 
 \begin{corollary}\label{corollary: RLR and surjection onto omega}\index{regular ring via syzygy modules}
  Let $(A,\mathfrak{m},k)$ be a Cohen-Macaulay local ring with canonical module $\omega$, and let
  \[
   f : \bigoplus_{n \in \Lambda} \left( \Omega_n^A(k) \right)^{j_n} \longrightarrow \omega
  \]\index{characterization of regular rings}
  be a surjective $A$-module homomorphism. Then $A$ is regular.
 \end{corollary}
 
 Here we obtain one new characterization of regular local rings. The following characterization is based on the existence of a
 non-zero direct summand with finite injective dimension of some syzygy module of the residue field.
 
\begin{theorem}\label{theorem: characterization of RLR, injdim}\index{regular ring via syzygy modules}
 The following statements are equivalent:\index{characterization of regular rings}
 \begin{enumerate}[{\rm (1)}]
  \item $A$ is a regular local ring.
  \item $\Omega_n^A(k)$ has a non-zero direct summand of finite injective dimension for some non-negative integer $n$.
 \end{enumerate}
\end{theorem}

\begin{proof}
 (1) $\Rightarrow$ (2). If $A$ is regular, then $\Omega_0^A(k)$ $(= k)$ itself is a non-zero $A$-module of finite injective
 dimension. Hence the implication follows.
 
 (2) $\Rightarrow$ (1). Without loss of generality, we may assume that $A$ is complete. Existence of a non-zero finitely generated $A$-module of finite injective dimension ensures that\index{Bass's conjecture} $A$ is Cohen-Macaulay. Therefore we may as well assume that $A$ is a Cohen-Macaulay complete local ring.
 
 Suppose that $L$ is a non-zero direct summand of $\Omega_n^A(k)$ with $\injdim_A(L)$ finite for some integer $n \ge 0$. We prove the
 implication by induction on $d := \dim(A)$. If $d = 0$, then the implication follows from
 Corollary~\ref{corollary: RLR and surjection onto finite injdim}.
 
 Now we assume that $d \ge 1$. Suppose the implication holds true for all such rings of dimension smaller than $d$. Since $A$ is
 Cohen-Macaulay and $\dim(A) \ge 1$, there exists $x \in \mathfrak{m} \smallsetminus \mathfrak{m}^2$ which is $A$-regular.
 We set
 \[
  \overline{(-)} := (-) \otimes_A A/(x).
 \]
 If $n = 0$, then the direct summand $L$ of $\Omega_0^A(k)$ $(= k)$ must be equal to
 $k$, and hence $\injdim_A(k)$ is finite, which gives $A$ is regular. Therefore we may assume that $n \ge 1$. Hence $x$ is
 $\Omega_n^A(k)$-regular. Since $L$ is a direct summand of $\Omega_n^A(k)$, $x$ is $L$-regular as well. This gives
 $\injdim_{\overline{A}}(\overline{L})$ is finite. Now we fix an indecomposable direct summand $L'$ of $\overline{L}$. Then
 $\injdim_{\overline{A}}(L')$ is also finite. Note that the $\overline{A}$-module $\overline{L}$ is a direct summand
 of $\overline{\Omega_n^A(k)}$. Hence $L'$ is an indecomposable direct summand of $\overline{\Omega_n^A(k)}$.
 Since $x \in \mathfrak{m} \smallsetminus \mathfrak{m}^2$ is an $A$-regular element, by Proposition~\ref{proposition: Nagata}, we have
 \begin{equation*}
  \overline{\Omega_n^A(k)} \cong \Omega_n^{\overline{A}}(k)\oplus \Omega_{n-1}^{\overline{A}}(k).
 \end{equation*}
 The uniqueness of Krull-Schmidt decomposition\index{Krull-Schmidt decomposition} (\cite[Theorem~(21.35)]{Lam01})
 yields that $L'$ is isomorphic to a direct summand of $\Omega_n^{\overline{A}}(k)$ or $\Omega_{n-1}^{\overline{A}}(k)$. Since
 $\dim(\overline{A}) = \dim(A) - 1$, by the induction hypothesis, $\overline{A}$ is a regular local ring, and hence $A$ is a regular
 local ring as $x \in \mathfrak{m} \smallsetminus \mathfrak{m}^2$ is an $A$-regular element.
\end{proof}
 
 Let $M$ be a finitely generated $A$-module. Consider the augmented minimal injective resolution of $M$:
 \[
  0 \longrightarrow M \stackrel{d^{-1}}{\longrightarrow} I^0 \stackrel{d^0}{\longrightarrow} I^1 \stackrel{d^1}{\longrightarrow}
  I^2 \longrightarrow \cdots \longrightarrow I^{n-1} \stackrel{d^{n-1}}{\longrightarrow} I^n \longrightarrow \cdots.
 \]
 Recall that the $n$th {\it cosyzygy module}\index{cosyzygy module} of $M$ is defined by
 \[
  \Omega_{-n}^A(M) := \Image(d^{n-1}) \quad\mbox{for all } n \ge 0.
 \]
 
 The following result is dual to Theorem~\ref{theorem: characterization of RLR, injdim} which gives another characterization of
 regular local rings via cosyzygy modules of the residue field.
 
 \begin{corollary}\label{corollary: dual result, cosyzygy}\index{regular ring via cosyzygy modules}
  The following statements are equivalent:\index{characterization of regular rings}
  \begin{enumerate}[{\rm (1)}]
   \item $A$ is a regular local ring.
   \item Cosyzygy module $\Omega_{-n}^A(k)$ has a non-zero finitely generated direct summand of finite projective
         dimension for some integer $n \ge 0$.
  \end{enumerate}
 \end{corollary}
 
 \begin{proof}
  (1) $\Rightarrow$ (2). If $A$ is regular, then $\Omega_0^A(k)$ $(= k)$ has finite projective dimension. Hence the implication follows.
  
  (2) $\Rightarrow$ (1). Without loss of generality, we may assume that $A$ is complete. Suppose
  \[
   \Omega_{-n}^A(k) \cong P \oplus Q \quad\mbox{for some integer $n \ge 0$},
  \]
  where $P$ is a non-zero finitely generated $A$-module of finite projective dimension. Consider the
  following part of the minimal injective resolution of $k$:
  \begin{equation}\label{theorem: dual result, cosyzygy: equation 1}
   0 \to k \to E \to E^{\mu_1} \to \cdots \longrightarrow E^{\mu_{n-1}} \longrightarrow \Omega_{-n}^A(k) \cong
   P \oplus Q \longrightarrow 0,
  \end{equation}
  where $E$ is the injective hull of $k$. Dualizing \eqref{theorem: dual result, cosyzygy: equation 1} with respect to $E$, and
  using
  \[
    \Hom_A(k,E) \cong k \quad \mbox{and} \quad \Hom_A(E,E) \cong A
  \]
  (cf. \cite[Proposition~3.2.12(a) and Theorem~3.2.13(a)]{BH98}), we get the following part of a minimal free resolution of $k$:
  \[ 
    0 \longrightarrow \Hom_A(P,E) \oplus \Hom_A(Q,E) \cong \Omega_n^A(k) \longrightarrow A^{\mu_{n-1}} \longrightarrow \cdots \to
    A^{\mu_1} \to A \to k \to 0.
  \]
  Clearly, $\Hom_A(P,E)$ is non-zero and of finite injective dimension as $P$ is non-zero and of finite projective dimension.
  Therefore the implication follows from Theorem~\ref{theorem: characterization of RLR, injdim}.
 \end{proof}
 
 Now we give an example to ensure that the existence of an injective homomorphism from a `special module' to a finite direct sum of
 syzygy modules of the residue field does not necessarily imply that the ring is regular.
 
 \begin{example}\label{example: counter, injective map}\index{example on syzygy}
  Let $(A,\mathfrak{m},k)$ be a Gorenstein local domain of dimension $d$. Then $\Omega_d^A(k)$ is a maximal Cohen-Macaulay $A$-module;
  see, e.g., \cite[Exercise~1.3.7]{BH98}. Since $A$ is Gorenstein, $\Omega_d^A(k)$ is a reflexive $A$-module
  (cf. \cite[Theorem~3.3.10]{BH98}), and hence it is torsion-free. Then, by mapping $1$ to a non-zero element of $\Omega_d^A(k)$,
  we get an injective $A$-module homomorphism
  \[
   f : A \longrightarrow \Omega_d^A(k).
  \]
  Note that $\injdim_A(A)$ is finite. But a Gorenstein local domain need not be a regular local ring.
 \end{example}

\newpage
\thispagestyle{empty}
\cleardoublepage

\chapter{Summary and Conclusions}\label{Summary and Conclusions}

 In this chapter, we give a brief summary of the work discussed in this dissertation. We also present a few open questions for possible 
 further future work.

\section{Asymptotic Prime Divisors Related to Derived Functors}
 
 Let $A$ be a  ring, and let $I$ be an ideal of $A$. Let $M$ and $N$ be finitely generated $A$-modules. The main problem
 we study in Chapter~\ref{Chapter: Asymptotic Prime Divisors over Complete Intersection Rings} is the following:
 
 \begin{customquestion}{\ref{Summary and Conclusions}.I}\label{Question 1 on associate primes}
  Is the set $~\displaystyle \bigcup_{i \ge 0}\bigcup_{n \ge 1} \Ass_A\left( \Ext_A^i(M, N/I^n N) \right)~$ finite?
 \end{customquestion}
 
 We prove that Question~\ref{Question 1 on associate primes} has an affirmative answer for local complete intersection rings.
 Moreover, we show that if $A$ is a local complete intersection ring, then there are non-negative integers $n_0$ and $i_0$ such that
 for all $n \ge n_0$ and $i \ge i_0$, the set
 \[
  \Ass_A\left( \Ext_A^i(M, N/I^n N) \right) ~\mbox{ depends only on whether $i$ is even or odd}.
 \]
 
 Next rings to consider after complete intersection rings are Gorenstein rings. So one may ask
 Question~\ref{Question 1 on associate primes} for Gorenstein local rings.
 
 In this dissertation, we prove the finiteness of the set of associate primes of a certain family of Ext-modules. It seems natural to
 ask the following counterpart for Tor-modules.
 
 \begin{customquestion}{\ref{Summary and Conclusions}.II}\label{Question 2 on associate primes}
  Is the set $~\displaystyle \bigcup_{i \ge 0} \bigcup_{n \ge 1} \Ass_A\left( \Tor_i^A(M, N/I^n N) \right)~$ finite?
 \end{customquestion}
 
 More particularly, one can ask the following:
 
 \begin{customquestion}{\ref{Summary and Conclusions}.III}\label{Question 3 on associate primes}
  Is the set $~\displaystyle \bigcup_{i \ge 0} \Ass_A\left( \Tor_i^A(M, N) \right)~$ finite?
 \end{customquestion}
 
 If $A$ is a regular local ring, then
 \[
  \Tor_i^A(M,N/I^n N) = 0 ~\mbox{ for all } i > \dim(A) \mbox{ and for all } n \ge 1,
 \]
 since $\projdim_A(M) \le \dim(A)$. Therefore it follows from the result of D. Katz and E. West (\cite[Corollary~3.5]{KW04}) that
 Questions~\ref{Question 2 on associate primes} and \ref{Question 3 on associate primes}
 have positive answers for regular local rings. But these two questions are open for local complete intersection rings.

\section{Castelnuovo-Mumford Regularity of Powers of Several Ideals}
 
 In Chapter~\ref{Chapter: Asymptotic linear bounds of Castelnuovo-Mumford regularity}, we study the Castelnuovo-Mumford regularity of
 powers of several ideals. Let $A = A_0[x_1,\ldots,x_d]$ be a  standard $\mathbb{N}$-graded algebra over an Artinian local ring $(A_0,\mathfrak{m})$. Let $I_1,\ldots,I_t$ be homogeneous ideals of $A$, and let $M$ be a finitely generated $\mathbb{N}$-graded $A$-module. We show that there exist two integers $k_1$ and $k_1'$ such that
 \[
   \reg(I_1^{n_1} \cdots I_t^{n_t} M) \le (n_1 + \cdots + n_t) k_1 + k'_1 \quad \mbox{for all }~ n_1, \ldots, n_t \in \mathbb{N}.
 \]  
 Moreover, we prove that if $A_0$ is a field, then there exist two integers $k_2$ and $k'_2$ such that
 \[
   \reg\left( \overline{I_1^{n_1}}\cdots \overline{I_t^{n_t}} M \right) \le (n_1 + \cdots + n_t) k_2 + k'_2
   \quad \mbox{for all }~ n_1, \ldots, n_t \in \mathbb{N},
 \]
 where $\overline{I}$ denotes the integral closure of an ideal $I$ of $A$.

 Keeping the single ideal case (i.e., the result \cite[Theorem~3.2]{TW05} of N. V. Trung and H.-J. Wang) in mind,
 we are now interested in the following question: 
 
\begin{customquestion}{\ref{Summary and Conclusions}.IV}
 Are there some integers $a_1,\ldots,a_t$ and $b$ such that
 \[
  \reg(I_1^{n_1} \cdots I_t^{n_t} M) = a_1 n_1 + \cdots + a_t n_t + b \quad \mbox{for all } n_1,\ldots,n_t \gg 0?
 \]
\end{customquestion}

 We observe that the method of Trung and Wang does not generalize to the several ideals case. So we need to build some new techniques to solve the above question.

 In \cite[Corollary~3.4]{TW05}, Trung and Wang showed that if $A$ is a standard $\mathbb{N}$-graded algebra over a field $A_0$, and if $I$ is a homogeneous ideal of $A$, then $\reg(\overline{I^n})$ is asymptotically a linear function of $n$. So one may ask now the following natural question for several ideals.

\begin{customquestion}{\ref{Summary and Conclusions}.V}
 Are there some integers $a_1,\ldots,a_t$ and $b$ such that
 \[
	  \reg \left( \overline{I_1^{n_1}} \cdots \overline{I_t^{n_t}} M \right) = a_1 n_1 + \cdots + a_t n_t + b \quad \mbox{for all } n_1,\ldots,n_t \gg 0?
 \]
\end{customquestion}

 Another question which seems to be interesting is the following:
 
\begin{customquestion}{\ref{Summary and Conclusions}.VI}
 What is the behaviour of $\reg\left( \overline{I_1^{n_1}\cdots I_t^{n_t}} M \right)$ as a function of $(n_1,\ldots,n_t)$?
\end{customquestion}

\section{Characterizations of Regular Rings via Syzygy Modules}

 In Chapter~\ref{Chapter: Characterizations of Regular Local Rings via Syzygy Modules}, we obtain a few necessary and sufficient
 conditions for a  local ring to be regular in terms of syzygy modules of the residue field. Let $A$ be a  local
 ring with residue field $k$.
 We show that if a finite direct sum of syzygy modules of $k$ maps onto a semidualizing $A$-module, then $A$ is regular.
 
 In \cite[Proposition~7]{Mar96}, A. Martsinkovsky proved that if a finite direct sum of syzygy modules of $k$ maps onto a non-zero
 $A$-module of finite projective dimension, then $A$ is regular. So it seems natural to ask the following counterpart for injective
 dimension:
 
\begin{customquestion}{\ref{Summary and Conclusions}.VII}
 If a finite direct sum of syzygy modules of $k$ maps onto a non-zero $A$-module of finite injective dimension,
 then is the ring $A$ regular?
\end{customquestion}
 
 This is not known in general. However, we give a partial answer to this question by proving that if a finite direct sum of syzygy
 modules of $k$ maps onto a non-zero maximal Cohen-Macaulay $A$-module of finite injective dimension, then $A$ is regular. We
 also obtain one new characterization of regular local rings, namely, the local ring $A$ is regular if and only if some syzygy module
 of $k$ has a non-zero direct summand of finite injective dimension.
 
\newpage
\thispagestyle{empty}
\cleardoublepage

\bibliographystyle{plain}

\begin{thebibliography}{AAAA}
\markboth{Bibliography}{Bibliography}		
\addcontentsline{toc}{chapter}{Bibliography} 	
\bibitem[Avr89]{Avr89} L. L. Avramov, {\it Modules of finite virtual projective dimension}, Invent. Math. {\bf 96} (1989), 71--101. Cited on

\bibitem[Avr96]{Avr96} L. L. Avramov, {\it Modules with extremal resolutions}, Math. Res. Lett. {\bf 3} (1996), 319--328. Cited on

\bibitem[Avr98]{Avr98} L. L. Avramov, {\it Infinite free resolutions}, pp. 1--118 in Six lectures on commutative algebra (Bellaterra, 1996), edited by J. Elias et al., Progr. Math. {\bf 166}, Birkh\"{a}user, Basel, 1998. Cited on

\bibitem[AB00]{AB00} L. L. Avramov and R.-O. Buchweitz, {\it Support varieties and cohomology over complete intersections}, Invent. Math. {\bf 142} (2000), 285--318. Cited on

\bibitem[BEL91]{BEL91} A. Bertram, L. Ein and R. Lazarsfeld, {\it Vanishing theorems, a theorem of Severi, and the equations defining projective varieties}, J. Amer. Math. Soc. {\bf 4} (1991), 587--602. Cited on

\bibitem[Bou83]{Bou83} N. Bourbaki, {\it Alg\`{e}bre commutative}, chapitres 8 et 9, Masson, Paris, 1983. Reprinted Springer, Berlin, 2006. Cited on

\bibitem[Bro79]{Bro79} M. Brodmann, {\it Asymptotic stability of $\Ass(M/I^n M)$}, Proc. Amer. Math. Soc. {\bf 74} (1979), 16--18. Cited on

\bibitem[BS13]{BS13} M. P. Brodmann and R. Y. Sharp, {\it Local Cohomology: An Algebraic Introduction with Geometric Applications}, Cambridge Studies in Advanced Mathematics {\bf 136}, Second Edition, Cambridge University Press, Cambridge, 2013. Cited on

\bibitem[BH98]{BH98} W. Bruns and J. Herzog, {\it Cohen-Macaulay Rings}, Cambridge Studies in Advanced Mathematics {\bf 39}, Revised Edition, Cambridge University Press, Cambridge, 1998. Cited on

\bibitem[Cha97]{Cha97} K. A. Chandler, {\it Regularity of the powers of an ideal}, Comm. Algebra {\bf 25} (1997), 3773--3776. Cited on

\bibitem[Cha07]{Cha07} M. Chardin, {\it Some results and questions on Castelnuovo-Mumford regularity}, in Syzygies and Hilbert Functions, edited by I. Peeva, Lect. Notes Pure Appl. Math. {\bf 254}, Chapman and Hall/CRC, New York, 2007. Cited on

\bibitem[CHT99]{CHT99} S. D. Cutkosky, J. Herzog and N. V. Trung, {\it Asymptotic behaviour of the Castelnuovo-Mumford regularity}, Compositio Math. {\bf 118} (1999), 243--261. Cited on

\bibitem[Dut89]{Dut89} S. P. Dutta, {\it Syzygies and homological conjectures}, in: Commutative Algebra, Berkeley, CA, 1987, in: Math. Sci. Res. Inst. Publ., vol. 15, Springer, New York, 1989, pp. 139--156. Cited on

\bibitem[Eis80]{Eis80} D. Eisenbud, {\it Homological algebra on a complete intersection, with an application to group representations}, Trans. Amer. Math. Soc. {\bf 260} (1980), 35--64. Cited on

\bibitem[GGP95]{GGP95} A. V. Geramita, A. Gimigliano and Y. Pitteloud, {\it Graded Betti numbers of some embedded rational $n$-folds}, Math. Ann. {\bf 301} (1995), 363--380. Cited on

\bibitem[Gol84]{Gol84} E. S. Golod, {\it G-dimension and generalized perfect ideals}, in: Algebraic Geometry and Its Applications, Trudy Mat. Inst. Steklov. {\bf 165} (1984), 62--66. Cited on

\bibitem[Gul74]{Gul74} T. H. Gulliksen, {\it A change of ring theorem with applications to Poincar\'{e} series and intersection multiplicity}. Math. Scand. {\bf 34} (1974), 167--183. Cited on

\bibitem[KW04]{KW04} D. Katz and E. West, {\it A linear function associated to asymptotic prime divisors}, Proc. Amer. Math. Soc. {\bf 132} (2004), 1589--1597. Cited on

\bibitem[Kod00]{Kod00} V. Kodiyalam, {\it Asymptotic behaviour of Castelnuovo-Mumford regularity}, Proc. Amer. Math. Soc. {\bf 128} (2000), 407--411. Cited on

\bibitem[Lam01]{Lam01} T. Y. Lam, {\it A First Course in Noncommutative Rings}, Second Edition, Springer-Verlag, New York, 2001. Cited on

\bibitem[LV68]{LV68} G. Levin and W. V. Vasconcelos, {\it Homological dimensions and Macaulay rings}, Pacific J. Math {\bf 25} (1968), 315--323. Cited on

\bibitem[Mar96]{Mar96} A. Martsinkovsky, {\it A remarkable property of the (co) syzygy modules of the residue field of a nonregular local ring},  J. Pure Appl. Algebra {\bf 110} (1996), 9--13. Cited on

\bibitem[Mat86]{Mat86} H. Matsumura, {\it Commutative Ring Theory}, Cambridge University Press, Cambridge, 1986. Cited on

\bibitem[MS93]{MS93} L. Melkersson and P. Schenzel, {\it Asymptotic prime ideals related to derived functors}, Proc. Amer. Math. Soc. {\bf 117} (1993), 935--938. Cited on

\bibitem[Put13]{Put13} T. J. Puthenpurakal, {\it On the finite generation of a family of Ext modules}, Pacific J. Math. {\bf 266} (2013), 367--389. Cited on

\bibitem[Rob87]{Rob87} P. C. Roberts, {\it Le th\'{e}or\`{e}me d'intersection},  C. R. Acad. Sci. Paris S\'{e}r. I Math. {\bf 304} (1987), 177--180. Cited on

\bibitem[Rob98]{Rob98} P. C. Roberts, {\it Multiplicities and Chern classes in local algebra}, Cambridge Tracts in Mathematics {\bf 133}, Cambridge University Press, Cambridge, 1998. Cited on

\bibitem[Sin00]{Sin00} A. K. Singh, {\it $p$-torsion elements in local cohomology modules}, Math. Res. Lett. {\bf 7} (2000), 165--176. Cited on

\bibitem[Stu00]{Stu00} B. Sturmfels, {\it Four counterexamples in combinatorial algebraic geometry},  J. Algebra {\bf 230} (2000), 282--294. Cited on

\bibitem[Swa97]{Swa97} I. Swanson, {\it Powers of ideals. Primary decompositions, Artin-Rees lemma and regularity}, Math. Ann. {\bf 307} (1997), 299--313. Cited on

\bibitem[SH06]{SH06} I. Swanson and C. Huneke, {\it Integral closure of ideals, rings, and modules}, London Mathematical Society Lecture Note Series {\bf 336}, Cambridge University Press, Cambridge, 2006. Cited on

\bibitem[Tak06]{Tak06} R. Takahashi, {\it Syzygy modules with semidualizing or G-projective summands}, J. Algebra {\bf 295} (2006), 179--194. Cited on

\bibitem[TW05]{TW05} N. V. Trung and H.-J. Wang, {\it On the asymptotic linearity of Castelnuovo-Mumford regularity},  J. Pure Appl. Algebra {\bf 201} (2005), 42--48. Cited on

\bibitem[Vas98]{Vas98} W. V. Vasconcelos, {\it Cohomological degrees of graded modules}, pp. 345--392 in Six lectures on commutative algebra (Bellaterra, 1996), edited by J. Elias et al., Progr. Math., {\bf 166}, Birkh\"{a}user, Basel, 1998. Cited on

\bibitem[Wes04]{Wes04} E. West, {\it Primes associated to multigraded modules}, J. Algebra {\bf 271} (2004), 427--453. Cited on

\end{thebibliography}

\newpage
\thispagestyle{empty}
\cleardoublepage

\chapter*{List of Publications}
\addcontentsline{toc}{chapter}{List of Publications}
\markboth{List of Publications}{List of Publications}

\subsection*{Published Papers in Journals} 
\begin{enumerate}
 \item[(1)] Dipankar Ghosh, {\it Asymptotic linear bounds of Castelnuovo--Mumford regularity in multigraded modules}, J. Algebra {\bf 445} (2016), 103--114.
 \item[(2)] Dipankar Ghosh and Tony J. Puthenpurakal, {\it Asymptotic prime divisors over complete intersection rings}, Math. Proc. Cambridge Philos. Soc. {\bf 160} (2016), 423--436.
\end{enumerate}

\subsection*{Accepted Paper in Journal} 
\begin{enumerate}
 \item[(3)] Dipankar Ghosh, Anjan Gupta and Tony J. Puthenpurakal, {\it Characterizations of regular local rings via syzygy modules of the residue field}, To appear in Journal of Commutative Algebra.
\end{enumerate}

\newpage
\thispagestyle{empty}
\cleardoublepage

\chapter*{Acknowledgements}
\addcontentsline{toc}{chapter}{Acknowledgements}
\markboth{Acknowledgements}{Acknowledgements}
 {\slshape
 It is a great pleasure to express my sincere gratitude to all those who contributed in many ways to the successful completion of
 this dissertation. 
 
 First and foremost,  I would like to express my heartfelt gratitude to my supervisor Prof. Tony J. Puthenpurakal for his constant support, motivation and encouragement which helped me to grow as a researcher. Under Prof. Puthenpurakal's guidance, the journey towards completion of this dissertation has become a fulfilling, enjoyable and unforgettable experience. 
 
 Besides my supervisor, it is a great pleasure for me to thank the rest of my research progress committee members, Prof. Jugal K. Verma and
 Prof. Ananthnarayan Hariharan, for their insightful comments, suggestions and encouragement during the entire course of my Ph.D journey.
 Their advice on both research as well as on my career have been priceless. My Ph.D. career has been blessed with beneficial courses taught
 by Prof. Ameer Athavale, Prof. Shripad M. Garge, Prof. Sudhir R. Ghorpade, Prof. Ananthnarayan Hariharan, Prof. Manoj K. Keshari, Prof. Ravi Raghunathan, Prof. Gopala K. Srinivasan, Prof. Tony J. Puthenpurakal and Prof. Jugal K. Verma.
 
 I would like to express my gratefulness to the anonymous reviewers for their careful reading of this dissertation, and for giving many insightful comments and suggestions which helped me to improve the overall quality of the dissertation.

 I express my sincere thanks and gratitude to all my teachers during my master's degree program at IIT Guwahati, especially, Prof. Anupam Saikia, Prof. Sriparna Bandopadhyay, Prof. Guru P. Prasad and Prof. Anjan K. Chakrabarty who motivated me for better prospects of my career. It was my
 great pleasure to have Prof. Saikia not only as my master's project guide but also as a good friend. My deep sense of gratitude is
 extended to my teacher, Prof. Deepankar Das, who taught me mathematics during my bachelor's degree program and encouraged me to pursue my 
 career at IIT. All my teachers' contributions in my academic career are much appreciated.

 I would like to express my gratefulness to Vivek Sadhu and Husney Parvez Sarwar for initiating the weekly students' algebra
 seminar in our department which we have continued later on. I have learned a lot from these seminars. I express my sincere
 thanks to all the members of this seminar series, especially, Rakesh Reddy, Prasant, Rajiv, Parangama, Jai Laxmi and Sudeshna for
 making it successful.
 
 I owe my deep sense of gratitude and thanks to some special individuals, specifically, to Debu da for always being there for me
 in my difficult time; to Vivek da for his motivation and generous guidance throughout my doctoral study; to Anjan da for his encouragement on mathematical discussions
 and for his practical suggestions all the time;
 to Priyo da for his support and unconditional help all the time in many purposes; to Manna da and Dipjyoti da for having given us
 their valuable time and place where we used to have our weekly get together and enjoyed delicious food. My heartfelt thanks to all
 who were there in those weekly get together, and made it fun-filled and memorable.
 
 I would like to acknowledge my sincere thanks to all of my seniors, especially, Soham, Amiya, Shubhankar, Ramesh M., Swapnil, Abhay,
 Parvez, Rakesh R., Sayan, Ishapathik, Ali Zinna, Satya Prakash, Sajith, Harsh, Ashok, Rakesh K., Ramesh K., Nidhin,
 Mahendra, Triveni, Aditya, Zafar and Asha for their encouragement, continuous support and help in many ways during the last five
 years. 
 
 I express my heartfelt thanks to all my batchmates: Arunava, Bhimsen, Ganesh, Gouranga, Harsh, Mrinmoy, Ojas, Parangama,
 Prasant, Radhika, Rajiv, Ram Murti, Saranya, Shuchita, Sunil, Tushar for their companionship, support and many good times which we
 shared during my stay at IIT Bombay. I am also thankful to my other fellow research scholars: Jai Laxmi, Provanjan, Arindam,
 Samriddho, Avijit, Gobinda, Praphulla, Sudeshna, Sudeep, Kuldeep, Rakhi, Souvik, Kriti, Venkitesh, Dibyendu, Hiranya, Mukesh, Akansha,
 Pramod and Vivek for their love and moral support.
 
 Many thanks to all of my hostel mates Oayes Midda, Bapi Bera, Nirmalya Chatterjee, Siddhartha Sarkar, Ashok Shaw, Saurav Mitra, Tanmoy
 Chakrabarty, Pritam Biswas, Arif Iqbal Mallick, Anjan Roy, Deb Datta Mandal, Tanmoy Sarkar, Rakhahari Saha, Mahitosh
 Biswas, Ananta Sarkar, Amal Sarkar, Sandip Mandal, Sudip Das, Arup Chakraborty and many more for the great time we had in our hostel,
 especially, over lunch and dinner shall always be cherished.
 
 Many people have walked in and out of my life. But very few have left their footprints in my heart. My heartiest thanks to Sunil for
 being one of them, and for being there with me always, especially, when times were hard. Many thanks to Gouranga for being a part of
 this journey which has started from IIT Guwahati. My heartfelt thanks to all of my post graduate friends, more specially, Sahadeb,
 Abhishek and Debarchana, for their constant support, love and encouragement. My academic journey with Sahadeb which has started from
 Midnapore College was cheerful and memorable. I thank my childhood friend Pintu for making my village life fun-filled and enjoyable.
 
 A lot of thanks to IIT Bombay for providing an excellent atmosphere and all the facilities which are suitable not only for good
 research but also to live an enjoyable life. I express my gratitude to all the staff members of Mathematics Department for their
 timely help and support.
 
 Many thanks to National Board for Higher Mathematics, Dept. of Atomic Energy, Govt. of India for providing financial support for this
 study. Without this support, it would have been difficult for me to continue my research.

 And finally, needless to say, I am deeply and forever indebted to my parents, grandparents for their unconditional love, patience,
 continuous support and encouragement throughout my entire life. This dissertation is dedicated to my parents and grandparents. I am
 very much fortunate to have two lovable sisters, Papiya and Priyanka, who have always been my source of inspiration. I cannot express how
 much my parents and sisters mean to me. I thank them for just being who they are. It is my pleasure to thank my cousins Kinkar,
 Shubhankar, Debashri, Bidhan, Raz, my sister-in-law Mina and my nephew Shyam for their love and affection. The appreciations as well
 as high expectations of my relatives about my career always motivate me to do something better.
 My heartfelt thanks to all my well wishers.
 }
 
\vspace*{0.3cm}
\rule{0cm}{0.4pt} \hfill {\bf Dipankar Ghosh} \hspace*{0.65cm}\\
\rule{0cm}{0.4pt} \hfill {\bf IIT Bombay} \hspace*{1cm}\\
\rule{0cm}{0.4pt}\hfill {\bf 2016} \hspace*{1.75cm}

\newpage
\thispagestyle{empty}
\cleardoublepage

\addcontentsline{toc}{chapter}{Index}
\markboth{Index}{Index}

\printindex

\end{document}